\documentclass{amsart}
\usepackage{amssymb,amsmath, amsthm,latexsym}
\usepackage[toc,page]{appendix} 
\usepackage{graphics}
\usepackage{psfrag}
\usepackage{amscd}
\usepackage{graphicx}
\usepackage{here}
\newcommand{\cal}[1]{\mathcal{#1}}
\theoremstyle{plain}
\newtheorem{theorem}{Theorem}
\newtheorem{coro}[theorem]{Corollary}
\newtheorem{prop}[theorem]{Proposition}
\newtheorem{lemma}{Lemma}[section]
\newtheorem{theo}[lemma]{Theorem}
\newtheorem{proposition}[lemma]{Proposition}
\newtheorem{corollary}[lemma]{Corollary}
\theoremstyle{definition}
\newtheorem{defi}[theorem] {Definition}
\newtheorem{definition}[lemma]{Definition}
\newtheorem{remark}[lemma]{Remark}
\newtheorem{example}[lemma]{Example}
\newtheorem*{conjecture}{Conjecture}
\newtheorem*{question}{Question}
\parskip=\bigskipamount

\let\egthree=\phi
\let\phi=\varphi
\let\varphi=\egthree
 
\begin{document}
\title[Boundary]
{An ${\cal E\cal Z}$-structure for the mapping class group}
\author{Ursula Hamenst\"adt}
\date{July 4, 2026}

\begin{abstract}
We construct a boundary for 
the mapping class group ${\rm Mod}(S)$
of a surface $S$ of finite type. The action of ${\rm Mod}(S)$ on this 
boundary is
minimal, strongly proximal and topologically free. 
The boundary is the boundary of an ${\cal E\cal Z}$-structure for 
${\rm Mod}(S)$. 
\end{abstract}

\maketitle


\section{Introduction}

The \emph{mapping class group} ${\rm Mod}(S)$ of
a closed oriented surface $S$ 
of genus $g\geq 0$ from which $m\geq 0$ points have been removed 
and so that $3g-3+m\geq 1$ 
is the group of isotopy classes of diffeomorphisms 
of $S$. 
The mapping class group is well known to be finitely presented, and
it admits explicit torsion free finite index subgroups.

A torsion free finite index subgroup $\Gamma$ of ${\rm Mod}(S)$ 
admits a \emph{finite} classifying space. Such a classifying space
can be constructed as follows.

Since the Euler characteristic of $S$ is negative, 
the \emph{Teichm\"uller space} ${\cal T}(S)$ of $S$ of all
\emph{marked} finite area complete hyperbolic structures on $S$ is defined. 
By elementary hyperbolic geometry, there exists a
number $\epsilon_0>0$ such that any two closed geodesics
on a hyperbolic surface of length at most $\epsilon_0$ are disjoint.
The \emph{systole} ${\rm systole}(X)$ of a hyperbolic metric $X$ is the length of 
a shortest closed geodesic. For $\epsilon <\epsilon_0$ define the \emph{$\epsilon$-thick part}
${\cal T}_\epsilon(S)$ of Teichm\"uller space as 
\[{\cal T}_\epsilon(S)=\{X\in {\cal T}(S)\mid
 {\rm systole}(X)\geq \epsilon\}.\]
The following is due to Ji and Wolpert \cite{JW10}, extending an earlier result of
Ivanov \cite{Iv02}, see Proposition 3.1 and Theorem 3.9
of \cite{J14} for an explicit statement.

\begin{theorem}[Ji-Wolpert]
For sufficiently small $\epsilon <\epsilon_0$, the set  
${\cal T}_\epsilon(S)$ is a manifold with corners which is
  a deformation retract of ${\cal T}(S)$. The mapping class
  group ${\rm Mod}(S)$ acts on ${\cal T}_\epsilon(S)$
properly and cocompactly.  
\end{theorem}

Since ${\cal T}(S)$ is homeomorphic to $\mathbb{R}^{6g-6+2m}$, we obtain that
${\cal T}_\epsilon(S)$ is contractible, locally
contractible and finite dimensional. 
As torsion free finite index subgroups $\Gamma$ of ${\rm Mod}(S)$
act freely on ${\cal T}_\epsilon(S)$, this implies that
$\Gamma\backslash {\cal T}_\epsilon(S)$ is a finite 
classifying space for $\Gamma$. In particular, $\Gamma$ is of type $F$. 

Motivated by the construction of the Borel-Serre bordification of an arithmetic
group which can be used to compute its \emph{virtual cohomological dimension},
that is, the cohomological dimension of a torsion free finite index subgroup,  
Harer \cite{Har86} initiated the construction of a bordification of ${\cal T}_\epsilon(S)$ 
which computes the virtual cohomological dimension 
${\rm vcd}({\rm Mod}(S))$ of the mapping class group. 
This program was completed by Ivanov \cite{Iv02} and consists in 
attaching to ${\cal T}_\epsilon(S)$ the \emph{curve complex} as an analog of a spherical
building.  
The bordification, which indeed computes the virtual cohomological dimension of 
${\rm Mod}(S)$, has the homotopy type of an infinite wedge of spheres
\cite{IJ08} and does not compactify the space ${\cal T}_\epsilon(S)$. 

In the setting of hyperbolic groups $\Gamma$, 
it turns out that the \emph{Gromov boundary} of $\Gamma$  
can be used to compute the virtual cohomological 
dimension \cite{BM91}. This Gromov boundary is the 
boundary of a compactification
of $\Gamma$ in the following sense.

\begin{defi}[Small boundary]\label{boundary}
  A \emph{boundary} of a finitely generated group $\Gamma$
  is a compact $\Gamma$-space 
  $Z$ with the following properties. 
\begin{itemize}
\item
There exists a topology on $\Gamma\cup Z$ which restricts to the
discrete topology on $\Gamma$, to the given topology on $Z$ and is such that
$\Gamma\cup Z$ is compact.
\item 
 The left action of $\Gamma$ on itself extends to 
the $\Gamma$-action on $Z$. 
\end{itemize}
The boundary is called \emph{small}
if the right action of $\Gamma$ extends to the trivial action of $\Gamma$ on $Z$.
\end{defi}

The Gromov boundary of a hyperbolic group has additional desirable properties. 
One of these is captured in the following definition, which 
is Lemma 1.3 of  \cite{B96} and Definition 1.1 of \cite{FL05}. Here the notion ${\cal E\cal Z}$-structure
stems from the idea of compactifying a classifying space for proper actions of a not necessarily
torsion free group $\Gamma$. That the Gromov boundary of a torsion free hyperbolic group defines
an ${\cal E\cal Z}$-structure for the group was established by Bestvina and Mess \cite{BM91}.

\begin{defi}[${\cal E\cal Z}$-structure]\label{ez}
  An \emph{${\cal E\cal Z}$-structure} for a finitely generated torsion free group $\Gamma$ 
    consists of a pair $(\overline{X},Z)$
    of finite dimensional compact metrizable spaces, with $Z$ nowhere
    dense in $\overline{X}$, and the 
 following additional properties.
    \begin{enumerate}
    \item $X=\overline{X}-Z$ is contractible and locally contractible.
    \item For every $z\in Z$ and every neighborhood
      $U$ of $z$ in $\overline{X}$ there exists a neighborhood
      $V\subset U$ of $z$ such that the inclusion $V-Z\to U-Z$
      is null-homotopic. 
    \item $X$ admits a covering space action of $\Gamma$ with compact
      quotient.
    \item The collection of all translates of a compact set in $X$
      form a null sequence in $\overline{X}$: that is,
      for every open cover ${\cal U}$ of $\overline{X}$,
      all but finitely many translates are \emph{${\cal U}$-small}, which means that
      they are entirely contained in one of the sets from the covering ${\cal U}$.  
    \item The action of $\Gamma$ on $X$ extends to an action on $\overline{X}$. 
\end{enumerate}
    \end{defi}  

If $\Gamma$ is a group with torsion, we define an ${\cal E\cal Z}$-structure for $\Gamma$ 
as a pair $(\overline{X},Z)$ as in Definition \ref{ez}, but where property (3) is replaced by 
property 
\begin{enumerate}
\item[($3^\prime$)]  $X$ admits a properly discontinuous action of $\Gamma$ with compact
      quotient.
      \end{enumerate}
Note that if $\Gamma$ is torsion free, then ($3^\prime$) is equivalent to (3).

The significance of an ${\cal E\cal Z}$-structure $(\overline{X},Z)$ for a torsion free 
group $\Gamma$ lies in the fact that the \v{C}ech cohomology of the space $Z$ 
computes the cohomological dimension ${\rm cd}(\Gamma)$ 
of the group, with a dimension shift of one
(Theorem 1.7 of \cite{B96}). Thus an ${\cal E\cal Z}$-structure for a 
finitely generated group which admits a torsion free finite index subgroup 
computes its virtual cohomological dimension. 
Furthermore, torsion free groups with an ${\cal E\cal Z}$-structure
admit an ${\cal E\cal Z}$-structure of the form $(\mathbb{D}^n,\Delta)$ where 
$\Delta$ is a closed subset of $\partial \mathbb{D}^n=S^{n-1}$, and 
the Novikov conjecture 
holds for these groups (Theorem 1.1 and Theorem 1.2 of \cite{FL05}). 
Theorem 1.3 of \cite{FL05} shows that it also implies a partial result
towards the Farell Jones conjecture.

An action of a group $G$ on a compact topological space $Z$ is called 
\emph{minimal} if every $G$-orbit is dense. It is called 
\emph{topologically free} if for every $\phi\in G-\{1\}$ the fixed point set of $\phi$
has empty interior. Furthermore, it is called \emph{strongly proximal} if
the action of $G$ on the space of Borel probability measures on $Z$ is such that the closure
of any orbit contains some Dirac measure. Compact $G$-spaces on which the 
$G$-action is minimal and strongly proximal were introduced by Furstenberg and 
are commonly called \emph{$G$-boundaries}, using a slightly misleading terminology. They have
many applications, in particular towards rigidity questions. For example, the existence
of a $G$-boundary on which the $G$-action is in addition topological free implies
$C^*$-simplicity of $G$ (which is known for the mapping class group), see  \cite{BKKO17} for this result and
a more comprehensive discussion of further applications and references. The Gromov boundary of a 
non-elementary hyperbolic group
has all these properties, but there are many other examples of groups to which these constructions
apply, including
lattices in semisimple Lie groups of non-compact type. 
The following is our main result.

\begin{theorem}\label{main}
There exists a compactification $\overline{\cal T}(S)$ of 
${\cal T}_\epsilon(S)$ with the following properties.
\begin{enumerate}
    \item ${\cal X}(S)=\overline{\cal T}(S)\setminus {\cal T}_\epsilon(S)$ is
  a small boundary for ${\rm Mod}(S)$.
  \item The action of
  ${\rm Mod}(S)$ on ${\cal X}(S)$ is minimal, strongly proximal
  and topologically free. 
  \item The pair $(\overline{\cal T}(S),{\cal X}(S))$ is an 
  ${\cal E\cal Z}$-structure for ${\rm Mod}(S)$.
\end{enumerate}
\end{theorem}

We call the space ${\cal X}(S)$ the \emph{geometric boundary} of 
${\rm Mod}(S)$. 

By the main results of \cite{FL05}, Theorem \ref{main} 
implies the following.

 \begin{coro}\label{farelljones}
Any torsion free finite index subgroup of ${\rm Mod}(S)$ satisfies the Novikov conjecture.
\end{coro}

The statement in the corollary was earlier 
established with different methods, see 
\cite{H09,K10,BBF15,BaB19}. 
The Farell Jones conjecture for mapping class groups, which implies the Novikov conjecture but
can not be deduced from Theorem \ref{main},  is due to 
Bartels and Bestvina \cite{BaB19}.

An alternative approach to the construction of an ${\cal E\cal Z}$-structure 
for the mapping class group, based on \emph{hierarchical hyperbolicity}, 
is due to Durham, Minsky and Sisto
\cite{DMS25}. They construct an ${\cal E\cal Z}$-structure for such a group 
as an abstract boundary from a property 
called asymptotically ${\rm CAT}(0)$ which is extracted from the 
combinatorial properties of hierarchically hyperbolic spaces. 
Earlier Durham, Hagen and Sisto \cite{DHS17} constructed
a boundary for hierarchially hyperbolic groups in the sense of Definition \ref{boundary}. 
In the case of the mapping class group, this boundary can be identified with 
the boundary constructed in Theorem \ref{main} as a set, however the topology 
is different. There are open sets in the boundary of \cite{DHS17} which do not contain
any open subset of the boundary we construct. 
We do not have information on the relation to the construction in 
\cite{DMS25}. Note to this
end that an  ${\cal E\cal Z}$-structure for a discrete group is by no means unique. 
Hierarchical hyperbolicity for ${\rm Mod}(S)$ only appears indirectly in this
article, but our approach shares with \cite{DMS25} the strategy to view 
the mapping class group as a  ${\rm CAT}(0)$-space on the large scale.

 As the virtual cohomological dimension 
${\rm vcd}({\rm Mod}(S))$ equals 
$4g-5$ if $g\geq 2$ and 
$m=0$,  $4g-4+m$ if $g\geq 1$ and $m-3$ if $g=0$ 
\cite{Har86}, the covering dimension of the space ${\cal X}(S)$ equals 
$4g-6$ if $g\geq 2$ and
$m=0$, $4g-5+m$ if $g\geq 1$ and $m>0$, and $m-4$ if $g=0$ \cite{B96}.
Note that for any torsion free finite index subgroup $\Gamma$ of ${\rm Mod}(S)$, the cohomology group
$H^{{\rm vcd}({\rm Mod}(S))}(\Gamma, \mathbb{Z}\Gamma)$ identifies with the 
$2g-2+m$-th homology group of the curve complex.
Since the curve complex has the homotopy type of an infinite wedge of spheres of dimension
$2g-2+m$ (Theorem 1.4 of \cite{IJ08}), this implies that 
the top dimensional \v{C}ech  cohomology group of ${\cal X}(S)$ is also 
infinite dimensional by Proposition 1.5 of \cite{B96}. 

Theorem \ref{main} can be viewed as giving some evidence that the 
\emph{asymptotic dimension} of ${\rm Mod}(S)$, which is known to be finite and 
at most quadratic in the virtual cohomological dimension, in fact equals the 
virtual cohomological dimension
of ${\rm Mod}(S)$. We refer to \cite{BB19} for a more detailed discussion on this and related
questions and results.

The following is a consequence of Theorem \ref{main} and 
Theorem 1.1 and Theorem 1.2 of \cite{FL05}. In its formulation
we denote by $\mathbb{D}^n$ the standard ball in $\mathbb{R}^{n}$.

\begin{coro}\label{concrete}
If $3g-3+m\geq 3$ then 
there exists a closed subset $\Delta$ of $S^{6g-5+2m}$ such that 
${\rm Mod}(S)$ admits an ${\cal E\cal Z}$-structure of the form 
$(\mathbb{D}^{6g-4+2m},\Delta)$. 
\end{coro}

\subsection*{Description of  
the boundary ${\cal X}(S)$ of ${\rm Mod}(S)$.}

The curve complex ${\cal C\cal G}(S_0)$ 
of a (not necessarily proper) essential subsurface 
$S_0$ of $S$ different from a pair of pants or an annulus
is the simplicial complex whose vertices are
isotopy classes of simple closed curves and where $k$ such curves span 
a $k-1$-simplex if they can be realized disjointly. 
If $S_0$ is a four-holed sphere or a one holed torus, then
this definition has to be modified by connecting two vertices by an edge
if they intersect in the minimal number of points. 

The curve complex, equipped with the natural simplicial metric,
is a hyperbolic geodesic metric space of infinite diameter \cite{MM99}. 
Its \emph{Gromov boundary} $\partial {\cal C\cal G}(S_0)$ is 
the space of
\emph{minimal geodesic laminations} on $S_0$ which \emph{fill} $S_0$,
that is, which intersect every essential simple closed curve on $S_0$ 
transversely. The topology on $\partial {\cal C\cal G}(S_0)$ is the 
\emph{coarse Hausdorff topology}. With respect to this topology, a sequence
$\lambda_i$ of minimal filling laminations converges to the lamination $\lambda$
if and only if the limit of any subsequence which converges
in the Hausdorff topology on compact subsets of $S_0$ 
contains $\lambda$ as a sublamination \cite{H06,K99}. The 
space $\partial {\cal C\cal G}(S_0)$ is separable and metrizable. 
Define the boundary of the curve complex of an essential annulus $A\subset S$ 
with core curve $c$ to consist of two points $c^+,c^-$.

If $S_1,\dots,S_k$ is a collection of isotopy classes of pairwise disjoint subsurfaces of $S$,
then we can form the join 
\[{\cal J}(\cup_{i=1}^kS_i)=\partial {\cal C\cal G}(S_1)* \cdots * \partial {\cal C\cal G}(S_k).\]
It  can be viewed as the set of formal sums $\sum_ia_i\lambda_i$ where
$a_i>0$, $\sum_ia_i=1$ and where $\lambda_i\in \partial {\cal C\cal G}(S_i)$ for all $i$.
This join is  a separable metrizable topological space.  
Note that if $S_{i_1},\dots,S_{i_s}$ is a subset of the set of surfaces 
$S_1,\dots,S_k$, then ${\cal J}(\cup_{j=1}^sS_{i_j})$ is naturally a 
non-empty closed subset of 
${\cal J}(\cup_{i=1}^kS_i)$ corresponding to formal sums $\sum_ia_i\lambda_i$ with $a_i=0$ for 
$i\not\in \{i_1,\dots,i_s\}$.
Define
\[{\cal X}(S)=
\cup {\cal J}(\cup_{i=1}^kS_i)\]
where the union is over all collections of pairwise disjoint essential subsurfaces
of $S$ and we use the obvious identification of points which arise in more than one 
way in this union. 
Here we view an essential annulus $A$ as an essential subsurface which is
disjoint from any subsurface which can be moved off $A$ by an isotopy.  
Thus ${\cal X}(S)$ is just the set of 
formal sums $\sum_ia_i\lambda_i$ where $a_i>0,\sum_ia_i=1$, where 
$\lambda_1,\dots,\lambda_k$ are pairwise disjoint
minimal geodesic laminations on $S$ and where each simple closed curve
component $\lambda_i$ is equipped with an additional label $+,-$. 
The mapping class group 
acts naturally on ${\cal X}(S)$ as a \emph{set}.

The topology on the space ${\cal X}(S)$ can be described as follows. 
Each of the subspaces ${\cal J}(\cup_{i=1}^kS_i)$, equipped with its topology as
a join of metric spaces, 
can be thought of 
as the visual boundary of the product of the hyperbolic metric spaces 
${\cal C\cal G}(S_i)$, where the coefficients in a sum $\sum_i a_i\lambda_i$
capture the relative speed in the factors with which a geodesic in the product
converges to a boundary point. It is embedded in ${\cal X}(S)$ as a topological space, which means
that the inclusion ${\cal J}(\cup_iS_i)\to {\cal X}(S)$ 
is a homeomorphism onto its image,
equipped with the subspace topology.

These countably many subspaces are glued together by first 
performing the obvious identifications and then using the Hausdorff topology
on the space of geodesic laminations on $S$ to determine which points in distinct
of the spaces ${\cal J}(\cup_iS_i)$ are close to each other to construct a neighborhood 
basis of every point.  
The idea is that as convergence in the 
coarse Hausdorff topology to a point $\xi\in \partial {\cal C\cal G}(S)$ describes convergence
of a sequence in the metric space ${\cal C\cal G}(S)$ to $\xi$, closeness in 
${\cal X}(S)$ of two points $\sum_i a_i\xi_i$ and $\sum_j b_j\zeta_j$ in two distinct of the subspaces 
${\cal J}(\cup_i S_i)$ 
is captured by closeness in the coarse Hausdorff topology of the laminations $\cup_i \xi_i$ and 
$\cup_j \zeta_j$, with a careful bookkeeping of the role of the coefficients in the joins 
that mimics what 
one would observe in a ${\rm CAT}(0)$-space containing two distinct convex subspaces of infinite diameter. 
Note that we use here specific properties of surfaces which have no analog in the context
of hierarchically hyperbolic spaces, although it seems possible that the construction of 
the topology of ${\cal X}(S)$ can be carried out in this more general context as well, using 
combinatorial tools.

The following result summarizes 
some technical properties of the geometric boundary. For its formulation,
let us invoke the Nielsen Thurston classification which states that
any nontrivial mapping class has
a finite power $\phi$ with the following property. There exists a decomposition 
$S=S_1\cup \cdots \cup S_k$ of $S$ into subsurfaces that are preserved by 
$\phi$ and such that for all $i<k$, the surface $S_i$ is connected and 
the restriction of $\phi$ to $S_i$ is pseudo-Anosov 
if $S_i$ is not an annulus, and it is a Dehn twist if $S_i$ is an annulus.
The restriction of $\phi$ to $S_k$ is trivial. We call a mapping class 
with this property a \emph{Nielsen Thurston mapping class}. 

Let $\phi$ be a Nielsen Thurston mapping class. For each $i<k$ such that
$S_i$ is not an annulus, the
restriction $\phi_i$ of $\phi$ to $S_i$ preserves precisely two geodesic laminations 
$\xi_i^{\pm}$ which are 
the attracting and repelling laminations of $\phi_i$. Similarly, for any component $S_i$ which 
is an annulus, the two labeled copies $\xi_i^{\pm}$ 
of the core curve of the annulus are preserved as well. 
Thus $\phi$ fixes any point of the 
form $\sum_ia_i\zeta_i$ where $\zeta_i$ is one of the laminations $\xi_i^{\pm}$ if $i<k$ and 
where $\zeta_k$ is an arbitrary point of the geometric boundary of the
(possibly disconnected)
surface $S_k$. We call points of this form the \emph{obvious fixed point set}.

\begin{prop}\label{details}
Let ${\cal X}(S)$ be the geometric boundary of ${\rm Mod}(S)$.
\begin{enumerate}
\item 
For any collection $S_1,\dots,S_k$ of pairwise disjoint subsurfaces
of $S$, the inclusion 
${\cal J}(\cup_{i=1}^kS_i)\to {\cal X}(S)$ is an embedding. 
In particular, 
the covering dimension of $\partial {\cal C\cal G}(S)$ is at most 
${\rm vcd}({\rm Mod}(S))-1$. 
\item The fixed point set for the action of a Nielsen Thurston mapping class 
$\phi$ on 
${\cal X}(S)$ is precisely the obvious fixed point set of $\phi$. 
\end{enumerate}
\end{prop}

That the covering dimension of $\partial {\cal C\cal G}(S)$ 
is bounded from above by ${\rm vcd}({\rm Mod}(S))$ 
 is due to 
Gabai (Proposition 16.3 of \cite{Ga14}). In view of the expectation that the asymptotic 
dimension of the mapping class group equals its virtual cohomological dimension, and 
that this dimension is captured by the covering dimension of the Gromov boundary of the 
curve graph paralleling the result of Harer \cite{Har86}, 
we expect that the bound we find in part (1) of 
Proposition \ref{details} is sharp. We are not aware of any lower bound 
available in the literature.

In \cite{BB19}, Bestvina and Bromberg converted Gabai's upper bound on the covering dimension of 
$\partial {\cal C\cal G}(S)$ into the same upper bound for its capacity dimension. We do not 
know whether this can also be achieved for the bound we find.

The Gromov boundary of the curve complex of $S$
can be obtained from 
the space  ${\cal F\cal M\cal L}\subset 
{\cal P\cal M\cal L}=S^{6g-7+2m}$ of projective measured geodesic laminations with 
minimal filling support, equipped with the weak$^*$-topology, by an equivariant continuous surjective map
${\cal F\cal M\cal L}\to \partial {\cal C\cal G}(S)$. This map is however not injective and the
following statement, which can be deduced from Theorem \ref{main}, Proposition \ref{details} and the 
results in \cite{FL05}, requires a proof. 

\begin{coro}\label{bdcurvegr}
The boundary $\partial {\cal C\cal G}(S)$ of the curve complex of $S$ 
admits an embedding into a manifold of dimension $6g-6+2m$ and into
$S^{6g-5+2m}$.
\end{coro}

Our construction is valid for the mapping class group of a once punctured torus or a 
four punctured sphere. In this case the mapping class group is virtually free and,
in particular, it is a hyperbolic group whose Gromov boundary is a Cantor set. 
The boundary we find is the Gromov boundary
of the group.

The construction of the boundary ${\cal X}(S)$ is motivated by the construction of the 
visual boundary of a ${\rm CAT}(0)$-space. 
Along the way we identify in Section \ref{titsbd} an analog of the familiar
Tits boundary of a symmetric space of higher rank. Although this has no obvious application, 
it draws a parallel between the mapping class group and a 
lattice in a higher rank symmetric space and connects to Harer's partial bordification of 
${\rm Mod}(S)$.

The advantage of our construction of a geometric 
boundary ${\cal X}(S)$ is that as a topological ${\rm Mod}(S)$-space,  
it is completely explicit
and can be used among others to study subgroups of ${\rm Mod}(S)$.

\subsection*{Overview of the article} In the first  part of the article, 
contained in Sections 2-3, we
define a topology on the \emph{set} ${\cal X}(S)$ which is invariant
under the action of ${\rm Mod}(S)$.
We also observe that
the action of ${\rm Mod}(S)$ on ${\cal X}(S)$ is strongly proximal. 
The more preliminary Section \ref{titsbd} introduces the 
\emph{oriented curve complex} and shows that it can be viewed
as a Tits type boundary for the mapping class group. It seems likely that 
the constructions and results in this first part carry over to colorable hierarchically
hyperbolic groups, however carrying this program out would add an extra
layer of notations and combinatorial discussions.

The second part of this article is devoted to 
showing that this topology
extends to ${\cal T}_\epsilon(S)\cup {\cal X}(S)$ and defines a compactification
of ${\cal T}_\epsilon(S)$ in such a way that every point in ${\cal X}(S)$ has a neighborbood
basis consisting of sets whose intersections with ${\cal T}_\epsilon(S)$ are contractible.
This is carried out in Section \ref{small} and Section \ref{sec:neighborhoodbasis}
and is the most involved part of this article.  We use the augmented Teichm\"uller
space as a witness of ${\rm CAT}(0)$ properties, thus mainly relying on geometry rather
than combinatorics. 
The proof of Theorem \ref{main} and the corollaries is completed in
Section \ref{Sec:metrizability}.

\noindent
{\bf Acknowledgement:} I am grateful to Alessandro Sisto for informing me about the article
\cite{DMS25}. Special thanks go to the referee for their careful reading and 
the many excellent suggestions for improving the 
exposition. 
This work was partially supported by 
the Hausdorff Center Bonn. The article was put into the current form while the author visited the Newton Institute in 
Cambridge during the program Operators, Graphs, Groups in summer 2025.

\section{The Tits boundary of ${\rm Mod}(S)$} \label{titsbd}

The \emph{join} $X_1*X_2$ of two topological spaces 
$X_1,X_2$ is defined to be the quotient $X_1\times X_2\times [0,1]/\sim$ where
the equivalence relation $\sim$ collapses $X_1\times X_2\times \{0\}$ to
$X_1$ and collapses $X_1\times X_2\times \{1\}$ to $X_2$. 
For example, the join $S_1^0*S_2^0$ of two $0$-spheres is the circle $S^1$, 
thought of as a union of four intervals glued at the endpoints, where each interval
has one endpoint in $S_1^0$ and the second endpoint in $S_2^0$. 
The join of two spaces 
$X_1,X_2$ contains an embedded copy of $X_1,X_2$. 
 
\begin{example}\label{boundaryjoin}
The product of two hyperbolic planes $\mathbb{H}^2\times \mathbb{H}^2$ is a 
complete simply connected Riemannian manifold of non-positive curvature. 
Its \emph{visual boundary} is the join $S^1*S^1$ of two circles that are
the Gromov boundaries of the embedded copies $\mathbb{H}^2\times \{y\}$ and 
$\{x\}\times \mathbb{H}^2$ 
of $\mathbb{H}^2$. This corresponds
to the fact that the projection of any geodesic in $\mathbb{H}^2\times \mathbb{H}^2$
to each of the two factors is a geodesic. Note that the join of two circles
is homeomorphic to $S^3$. \hfill $\blacksquare$
\end{example}

Define the \emph{oriented curve complex}
${\cal O\cal G}(S)$ of 
an oriented connected surface $S$ of genus $g$ with $m$ punctures and
$3g-3+m\geq 2$  to be 
the complex whose vertices are isotopy classes of essential 
oriented simple closed curves in $S$
and whose one-skeleton consists of 
edges (of length $1$) connecting two vertices if
they can be realized disjointly and are not homotopic up to orientation.
Thus any simple closed curve in $S$ defines two distinct 
vertices in ${\cal O\cal G}(S)$, and
these vertices are not connected by an edge.
Furthermore, we require that any collection of $k\geq 2$ oriented 
disjoint simple closed curves which are distinct as unoriented  curves 
span a simplex. The union of these simplices
defined by a fixed collection of $k$ 
curves equipped with all combinations of orientations 
is a sphere of dimension $k-1$. If $3g-3+m=1$ then we define the oriented curve
complex in the same way, but defining edges in the one-skeleton by connecting
two vertices if up to homotopy, they intersect transversely in the minimal number of points 
(one for the once punctured torus and two for the four punctured sphere).

Note that a point
in ${\cal O\cal G}(S)$ can be viewed
as a formal linear combination
$\sum_{i=1}^ka_i\lambda_i$ where for some $k\geq 1$, 
$\lambda_1,\dots,\lambda_k$ 
are pairwise disjoint oriented simple closed curves,  
where $a_i>0$ for all $i$ and $\sum_ia_i=1$.  In other words,
a point in the oriented curve complex can be viewed as a point in 
the join of a finite collection of oriented pairwise disjoint simple
closed curves. 
If $S$ is a once punctured torus or a four punctured sphere, 
then the oriented curve complex is defined in the same way
except that two oriented curves are connected by an edge if 
they intersect in the minimal number of points (one for the once punctured
torus and two for the four punctured sphere). 

\begin{remark}\label{spherical}
  If we choose the length of the edges of the oriented
 curve complex 
  to be $\pi/2$, then this is consistent with
the idea that the oriented curve complex can be thought of as
being contained in the Tits boundary of 
${\rm Mod}(S)$, equipped with the angular length metric which 
identifies each sphere with a sphere of constant curvature one. 
Any such sphere will be viewed as the visual boundary of a \emph{Dehn twist flat},
that is, a free abelian subgroup of ${\rm Mod}(S)$ generated by Dehn twists
about a collection of disjoint simple closed curves. \hfill $\blacksquare$
\end{remark}

A simple closed curve $c$ on $S$ is the core curve of an embedded annulus
$A(c)\subset S$. 
Define the "curve graph" ${\cal C\cal G}(A(c))$ 
of the annulus $A(c)$ 
as a graph of isotopy classes
of arcs connecting the two boundary components of $A(c)$ 
and whose endpoints
are allowed to move freely in the complement of a fixed point $p_-,p_+$ on each of the
two boundary circles $\partial_{\pm}A(c)$. 
Two such arcs are connected by an edge if they can be realized disjointly in their homotopy class.
Although this is not the standard definition,
it will be convenient for our purpose. Note that it depends on the choice of 
the points $p_-,p_+$. To apply it to arcs with at least one endpoint at $p_-$ or $p_+$ we 
deform the arc with a small homotopy to a disjoint arc with no endpoint at 
$p_-,p_+$, 
creating a small ambiguity which is unavoidable when one discusses curve graphs of annuli. 

To see that this definition well encodes the action of the infinite cyclic group of Dehn twists of 
$A(c)$ equip the annulus $A(c)$ with an orientation inherited from an orientation
of $S$. Given an arc $\alpha\subset A(c)$ with endpoints in $\partial A(c) \setminus \{p_{\pm}\}$, 
we can slide the endpoint of $\alpha$ on $\partial _+A(c)$ across $p_+$ in positive or 
negative direction, keeping its second endpoint fixed,  
to create a disjoint arc $\alpha^\prime$ which is homotopic in the above sense to 
the image of $\alpha$ under a simple Dehn twists $T_c$ about the core curve $c$. It is
connected to $\alpha$ by an edge in ${\cal C\cal G}(A)$. A repetition of this 
construction gives rise to the arc $T^2_c(\alpha)$ (up to homotopy) which is 
not connected to $\alpha$ by an edge. Thus 
the curve graph ${\cal C\cal G}(A(c))$ is a simplicial line which admits the infinite cyclic 
group of Dehn twists as a vertex transitive group of translations. 
The distinction between
a positive and a negative Dehn twist
about $c$ only depends on the orientation of $S$ but not on the orientation
of $c$. The choice
of an orientation of $c$ can be thought of as
a spiraling direction about $c$ for oriented arcs connecting the two boundary components of 
$A(c)$. 

In the sequel we denote by $c^+$ the point in the Gromov boundary
of ${\cal C\cal G}(A(c))$ (which consists of two points) 
which corresponds to an 
iteration of positive Dehn
twists about $c$, and we denote by $c^-$ the point in
the Gromov boundary of ${\cal C\cal G}(A(c))$ which
corresponds to an  iteration of negative Dehn twists about $c$.
Write ${\cal J}(c)=\{c^+,c^-\}$. It will be convenient to think about 
${\cal J}(c)$ as a set of 
two distinct points in the oriented curve complex of $S$,
with the same underlying curve.

If $S_0$ is a subsurface of $S$ different from a pair of pants or an annulus, 
which can be a surface with boundary and/or punctures, 
then we denote its (non-oriented) curve complex by 
${\cal C\cal G}(S_0)$. 
The vertices of this complex are isotopy classes 
of essential, that is, non-peripheral simple closed curves. 
If $S_0$ is different from a one-holed torus or a 
four-holed sphere, then a collection of $k\geq 2$ such disjoint simple closed
curves span a simplex of dimension $k-1$. If $S_0$ is a one-holed torus or
a four-holed sphere then two simple closed curves are connected by an edge
if they intersect transversely in the minimal number of points. 
The curve complex of $S_0$
is hyperbolic and hence it has a \emph{Gromov boundary}
$\partial {\cal C\cal G}(S_0)$. 
As a set, the Gromov boundary $\partial {\cal C\cal G}(S_0)$
is the set of all minimal filling geodesic laminations on $S_0$.
We refer to \cite{H06} for an account on this result of Klarreich \cite{K99}.

There is a natural metrizable topology on the union 
$\overline{\cal C\cal G}(S_0)$ of ${\cal C\cal G}(S_0)$
with its Gromov boundary, called the \emph{coarse Hausdorff topology}. 
With respect to this topology,
the subspace ${\cal C\cal G}(S_0)$, equipped with its simplicial topology, is an open dense
subset. To define this topology equip the surface $S_0$ with a complete finite volume hyperbolic 
metric with geodesic boundary. This choice defines a Hausdorff topology on the 
space of compact subsets of $S_0$. 
A sequence 
$\lambda_i\subset {\cal C\cal G}(S_0)\subset {\cal C\cal G}(S_0)\cup \partial {\cal C\cal G}(S_0)$ 
of vertices in ${\cal C\cal G}(S_0)$ converges in the coarse Hausdorff 
topology to $\lambda\in \partial {\cal C\cal G}(S_0)$
if and only if the limit of 
any converging subsequence of $\lambda_i$ 
in the Hausdorff topology on compact subsets of $S_0$  contains
$\lambda$ as a sublamination \cite{H06}. 
Define 
\[{\cal J}(S_0)=\partial {\cal C\cal G}(S_0),\] equipped with the 
topology as a subset of $\overline{\cal C\cal G}(S_0)$. 
If $S_0$ is a pair of pants, then we define ${\cal J}(S_0)=\emptyset$.


If $S_1,\dots,S_k$ are \emph{disjoint} connected 
subsurfaces of $S$  (we allow that they share boundary
components, and annuli about such boundary components
may be included in the list), then we define 
\begin{equation}\label{joinformula}
{\cal J}(\cup_iS_i)=\partial {\cal C\cal G}(S_1)*\cdots *\partial {\cal C\cal G}(S_k)\end{equation}
to be the join of the spaces
${\cal J}(S_i)=\partial {\cal C\cal G}(S_i)$.  
For example, if 
$S_1\subset S$ is a subsurface which is the complement of a non-separating simple closed
curve $c$, then 
\[{\cal J}(S_1\cup A(c))=\partial{\cal C\cal G}(S_1)*\{c^+,c^-\}.\] 
A point in ${\cal J}(S_1\cup \cdots \cup S_k)$ can be viewed as a formal linear
combination 
\[\xi=\sum_ia_i \xi_i\] 
where $\xi_i\in \partial {\cal C\cal G}(S_i)$, $a_i\geq 0$ for all $i$ and,
furthermore, $\sum_ia_i=1$.
The union 
\[{\rm supp}(\xi)=\cup_{a_i>0}\xi_i\] is 
a geodesic lamination with minimal components $\xi_i$, and 
$\xi$ can be viewed as a weighted (and partially labeled if there are
simple closed curve components of $\xi$ with positive weight) 
geodesic lamination. 
For all $u\leq k$ there is an inclusion 
${\cal J}(S_1\cup \cdots \cup S_u)\subset {\cal J}(S_1\cup \dots \cup S_k)$
which is a topological embedding.

A collection $S_1,\dots,S_k$ of disjoint connected subsurfaces of $S$ is called \emph{maximal} 
if $S-\cup_iS_i=\emptyset$. By convention, this means that for any  boundary component $c$ of 
one of the surfaces $S_i$, the annulus $A(c)$ is contained in the collection. Any collection 
$S_1,\dots,S_\ell$ 
of disjoint connected subsurfaces of $S$ is contained in a maximal collection of such subsurfaces,
however this maximal collection is  in general not unique. 
For example, there is a canonical
maximal collection containing $S_1,\dots,S_k$ 
which is comprised of the surfaces $S_i$, the annuli 
$A(c)$ where $c$ runs through all boundary 
components of $\cup_iS_i$ which are not already contained in the list,
and all connected components 
of $S-\cup_iS_i$. 

Define 
\begin{equation}\label{exx}
    {\cal X}(S)=\cup {\cal J}(S_1\cup \cdots  \cup S_k)/\sim
\end{equation}
where
the union is over all collections of disjoint subsurfaces $S_1,\dots,S_k$ 
of $S$. 
The equivalence relation 
$\sim$ identifies two points $\sum_ia_i\xi_i$ and 
$\sum_jb_j\zeta_j$ if they coincide as weighted labeled
geodesic laminations.  
Thus a point in ${\cal X}(S)$ is nothing else but 
a formal sum $\sum_{i=1}^k a_i \xi_i$ where
$a_i> 0,\sum_ia_i=1$, where $\xi_1,\dots,\xi_k$ are pairwise
disjoint minimal geodesic laminations on $S$ and where every simple closed
curve component of this collection is in addition equipped with a label $\pm$.
Note that the oriented curve complex ${\cal O\cal G}(S)$
of $S$ can naturally be identified with 
the union of the 
subsets ${\cal J}(A(c_1)\cup \cdots \cup A(c_k))$ of ${\cal X}(S)$, 
and its Gromov boundary (which is just the Gromov boundary 
$\partial {\cal C\cal G}(S)$ of the
non-oriented curve complex of $S$) also is contained in ${\cal X}(S)$.
The mapping class group ${\rm Mod}(S)$ naturally acts on 
the set ${\cal X}(S)$.

\begin{example}\label{oncepuncturedtorus}
The definition (\ref{exx}) makes sense if $S$ is a once
punctured torus or a four punctured sphere.
In this case there are no non-peripheral subsurfaces 
of $S$ different from annuli and pairs of pants, 
and the set ${\cal X}(S)$ is just the union of the Gromov boundary of
the curve graph ${\cal C\cal G}(S)$ with a countable set, consisting of all 
oriented non-peripheral simple closed curves on $S$. 
We discuss the case of the once punctured torus in detail, the case of the 
four punctured sphere is completely analogous.

The curve graph of $S$ is the well-known \emph{Farey graph}. Its vertices can 
be represented by the rational points in the boundary 
$\partial \mathbb{H}^2=\mathbb{R}\cup \{\infty \}$ of the hyperbolic plane. If one represents the 
edges of the Farey graph by geodesics in $\mathbb{H}^2$, then one obtains 
a tesselation of the hyperbolic plane by ideal triangles which is invariant under 
the mapping class group ${\rm PSL}(2,\mathbb{Z})$ of $S$. 
The boundary $\partial T$ of the dual tree $T$ of this tesselation is a Cantor set which admits 
a surjective continuous map onto the boundary $\partial \mathbb{H}^2$ of the hyperbolic plane.
Each irrational point in $\partial \mathbb{H}^2$ corresponds to a point in the Gromov
boundary of ${\cal C\cal G}(S)$ and
has precisely one preimage, and 
the rational points which correspond to the vertices of the curve graph have two preimages. 

The vertices of the Farey graph are also
the fixed points of the parabolic subgroups of 
${\rm PSL}(2,\mathbb{Z})$.
With this interpretation, the set ${\cal X}(S)$ can be identified with the Cantor
set $\partial T$ 
obtained by replacing each rational point in $\mathbb{R}\cup \{\infty\}$ 
by a compact interval and removing 
the interior of the interval. This Cantor set in turn has a natural identification 
with the Gromov boundary $\partial T$ of the virtually free group ${\rm PSL}(2,\mathbb{Z})$.
In particular, there is a natural invariant topology on ${\cal X}(S)$ so that with this
topology, ${\cal X}(S)$ is a compact ${\rm PSL}(2,\mathbb{Z})$-space which contains
the Gromov boundary $\partial {\cal C\cal G}(S)$ of the curve graph of $S$ as 
a dense embedded subset. Furthermore, following \cite{BM91}, 
the set ${\cal X}(S)$ equipped with this topology is the boundary of an ${\cal E\cal Z}$-structure for 
${\rm PSL}(2,\mathbb{Z})$. \hfill $\blacksquare$
\end{example}

\begin{example}\label{joincontained}
Let $S_1,\dots,S_k$ be a disjoint union of subsurfaces of $S$ which are different from 
pairs of pants. Then the join 
${\cal X}(S_1)*\cdots *{\cal X}(S_k)$ is a subset of ${\cal X}(S)$. \hfill $\blacksquare$
\end{example}

The short remainder of this section is not used in the sequel, but was included here
to illustrate the similarities and differences of the geometry of the mapping class group
as encoded in the boundary ${\cal X}(S)$ and classical constructions for Gromov hyperbolic 
spaces and symmetric spaces of higher rank. 

The oriented curve complex of $S$ is connected, and any non-filling
geodesic lamination, that is, a geodesic lamination which is disjoint from
some simple closed curve, is disjoint from some vertex of
${\cal O\cal G}(S)$. Thus if we equip 
${\cal X}(S)\setminus \partial {\cal C\cal G}(S)$ with the topology
of a simplicial complex whose edges are the joins of two disjoint 
(perhaps labeled) geodesic laminations, then this complex
is connected. As a consequence,
the \emph{set} ${\cal X}(S)$ can be equipped with a topology 
which coincides with the topology of a (non-locally finite) simplicial complex
on ${\cal X}(S)\setminus \partial {\cal C\cal G}(S)$ and is such that each point 
in $\partial {\cal C\cal G}(S)$ is isolated. 
 We write ${\cal X}_T(S)$ for ${\cal X}(S)$ equipped with this topology and call 
 ${\cal X}_T(S)$ the \emph{Tits boundary} of ${\rm Mod}(S)$ 
(having the Tits boundary of a ${\rm CAT}(0)$ space as guidance). 
From this description, we obtain

\begin{lemma}\label{mcg}
The mapping class group ${\rm Mod}(S)$ of $S$ acts on ${\cal X}_T(S)$ as a group of 
simplicial automorphisms. 
\end{lemma}
\begin{proof}
The mapping class group acts on the oriented curve complex of $S$ as a group 
of simplicial automorphisms, and this action extends to an action
on the space of formal sums of weighted disjoint minimal geodesic
laminations preserving weight and disjointness. Furthermore, 
it acts on $\partial {\cal C\cal G}(S)$ as
a group of transformations. 
Since the topology on ${\cal X}_T(S)$ is 
the topology of a disconnected simplicial complex, constructed from the 
curve complexes of subsurfaces, 
the lemma follows. 
\end{proof}

\begin{remark}
The Tits boundary of a ${\rm CAT}(0)$ space $X$ can be viewed as
the geometric boundary (that is, the ${\rm CAT}(0)$ boundary) of $X$,
equipped with a topology which in general is finer than the geometric
topology. We shall see in Section \ref{small} that the same holds true for the
Tits boundary and the geometric boundary of ${\rm Mod}(S)$.
\end{remark}

\section{A topology for ${\cal X}(S)$}\label{sec:topology}

The goal of this section is to equip the set ${\cal X}(S)$ with a topology which is coarser
than the Tits topology so that for this
topology, ${\cal X}(S)$ becomes a compact ${\rm Mod}(S)$-space. 

Let
$\xi^j=\sum_m a_m^j \xi_m^j$ be a sequence in ${\cal X}(S)$.
We shall impose in three steps a requirement for the
sequence to converge to a point 
$\zeta=\sum_{i=1}^kb_i\zeta_i\in {\cal X}(S)$.
Here as before, we assume that $a_m^j>0,b_i>0,\sum_i b_i=1=\sum_m a_m^ j$ 
for all $j$ and that furthermore, 
${\rm supp}(\xi_i)$, ${\rm supp}(\zeta)$ are 
disjoint unions of minimal components. The three steps used to construct
the topology are contained in three different subsections.

\subsection{Convergence to a minimal filling lamination}\label{filling}

Recall 
that the space of geodesic 
laminations on $S$ is compact with respect to the Hausdorff topology. We refer to  
Chapter 3 of \cite{CB88} for the case that the surface is compact. The case that $S$ has 
finite type but is non-compact follows in exactly the same way as in this case, any simple
geodesic is contained in a fixed compact subset of $S$, see e.g. Section 4.1 of \cite{CEG87}.  
Equivalently, replacing a puncture $p$ by a boundary component via removing the interior 
of an embedded disk in $S$ punctured at $p$ represents a geodesic lamination as a lamination
on a compact surface with boundary, to which \cite{CB88} can directly be applied. 

\noindent
{\bf Requirement 1:} {\sl Convergence in the coarse Hausdorff topology}\\
Let $\xi^{\ell_n}$ be any
subsequence of the sequence $\xi^j$ such that 
the geodesic laminations ${\rm supp}(\xi^{\ell_n})$ converge in the Hausdorff topology to 
a geodesic lamination $\beta$. Then $\beta$ contains ${\rm supp}(\zeta)$ as 
a sublamination.


\begin{example}\label{minimalfilling}
A geodesic lamination $c$ coarsely determines a point in ${\cal C\cal G}(S)\cup \partial {\cal C\cal G}(S)$.
Namely, if $c$ is minimal and filling, then $c\in \partial {\cal C\cal G}(S)$. Otherwise $c$ is disjoint from a simple
closed curve $c^\prime\in {\cal C\cal G}(S)$.

By a result of Klarreich \cite{K99} as reported in \cite{H06}, a sequence of non-filling 
geodesic laminations $c_i$
converges in the coarse Hausdorff topology to a minimal filling geodesic lamination 
$\eta$ if and only if the simple closed curves  $c_i^\prime\in {\cal C\cal G}(S)$ 
converge in ${\cal C\cal G}(S)\cup \partial {\cal C\cal G}(S)$ to 
$\eta\in \partial {\cal C\cal G}(S)$. \hfill $\blacksquare$
\end{example}

\begin{example}\label{join}
Let $S_1,\dots,S_k\subset S$ be disjoint subsurfaces. Example \ref{joincontained} shows that 
${\cal X}(S)$ contains the join ${\cal X}(S_1)*\cdots *{\cal X}(S_k)$ as a subset.
An element  
$\xi\in {\cal X}(S_1)*\cdots *{\cal X}(S_k)$ can be represented in the form
$\xi=\sum_i a_i\xi_i$
where $\xi_i\in {\cal X}(S_i)$, in particular, 
${\rm supp}(\xi_i)\subset S_i$, and $\sum_ia_i=1$. 
Since the subset of 
geodesic laminations on $S$ which are supported in $S_i$ is closed with respect to the 
Hausdorff topology, this implies that for any topology on ${\cal X}(S)$ which fulfills
the first requirement above, the subspace ${\cal X}(S_1)*\cdots *{\cal X}(S_k)$ of 
${\cal X}(S)$ is closed. \hfill $\blacksquare$
\end{example}

The examples show that the requirement (1) determines completely and geometrically
the convergence of a sequence $\xi_i\subset {\cal X}(S)$ to a point
$\xi\in \partial {\cal C\cal G}(S)\subset {\cal X}(S)$.

\begin{example}\label{farey}
In the case that $S$ is a once punctured torus or a four punctured sphere, then 
any non-trivial subsurface of $S$ different from a pair of pants is an annulus. In particular,
any geodesic lamination which is a disjoint union of minimal components is minimal
and either fills $S$ or is a simple closed curve. 
This easily implies that the topology of ${\cal X}(S)$ is determined by the requirement (1). 
Furthermore, it 
follows from Example \ref{oncepuncturedtorus} and the discussion in Example \ref{minimalfilling}
that the space ${\cal X}(S)$  is naturally homeomorphic to the Gromov boundary of 
the hyperbolic group ${\rm Mod}(S)$. \hfill $\blacksquare$
\end{example}

\subsection{Product spaces}\label{subsec:product}

In this subsection we consider a collection
$S_i$ $(1\leq 1\leq k)$ of pairwise disjoint proper subsurfaces of $S$.
This collection determines the subspace
\[{\cal X}(\cup_iS_i)={\cal X}(S_1)*\cdots *{\cal X}(S_k)\subset
{\cal X}(S).\]
Put 
\[{\cal C\cal G}(\cup_iS_i)={\cal C\cal G}(S_1)\times \cdots \times 
{\cal C\cal G}(S_k).\]
Our goal is to define a topology on the union 
\begin{equation}\label{up}
{\cal Y}(\cup_iS_i)={\cal C\cal G}(\cup_iS_i)\cup 
  \partial {\cal C\cal G}(S_1)*\cdots * \partial {\cal C\cal G}(S_k)=
{\cal C\cal G}(\cup_i S_i)\cup {\cal J}(\cup_i S_i)\end{equation}
which will be used in the construction of a topology on 
${\cal X}(S)$.

The main tool are \emph{complete markings} of (not necessarily proper) essential 
subsurfaces $S_0$ of the surface $S$. 
Such  a marking consists of a pants decomposition $P$ for $S_0$ together with a 
collection of \emph{spanning curves}. For every component $c$ of $P$, there exists 
such a spanning curve which intersects $c$ 
in the minimal number of points 
(one or two) and is disjoint from all other pants
curves. Two spanning curves may not be disjoint, but we require that the number of 
their intersection points is bounded from above by a universal constant.
Since there are only finitely many topological types of pants decompositions,
this can clearly be achieved. 
There is a natural way to equip the set of all markings on $S_0$ with the structure
of a locally finite connected graph on which the mapping class group 
${\rm Mod}(S_0)$ of $S_0$ acts properly and cocompactly. 
We refer to \cite{MM00} for more information on this construction.

Choose a marking $\mu$ on $S$ as a basepoint for the proper cocompact action of 
${\rm Mod}(S)$ on the marking graph. 
For every subsurface $S_0$ of $S$ which is distinct from a pair of pants or an annulus, this marking
can be used to construct a marking ${\rm pr}_{S_0}(\mu)$ of $S_0$ as follows.

There is a coarsely well defined \emph{subsurface projection} 
\[{\rm pr}_{S_0}:{\cal C\cal G}(S)\to {\cal C\cal G}(S_0)\]
which associates to a simple closed curve $c$ 
its intersection ${\rm pr}_{S_0}(c)=c\cap S_0$ with $S_0$ in the following sense. 
If $c\subset S_0$ then put ${\rm pr}_{S_0}(c)=c$, and if 
$c$ is disjoint from $S_0$ then put ${\rm pr}_{S_0}(c)=\emptyset$. 
In all other cases, $c\cap S_0$ consists 
of a collection of pairwise disjoint arcs with endpoints on the
boundary of $S_0$. We then put ${\rm pr}_{S_0}(c)=u$ for 
a simple closed curve $u$ in $S_0$ which is obtained from one of these intersection arcs 
by choosing a component of the boundary of a tubular neighborhood of the 
union of the arc with the boundary components of $S_0$ containing its endpoints.
Informally, we say that the simple closed curve is obtained by surgery on the arc.

Given a marking $\mu$ for $S$, 
the union of the intersections of the marking curves with $S_0$ 
consists of a 
union of 
arcs and simple closed curves on $S_0$ with pairwise uniformly bounded intersection
numbers which decompose
$S_0$ into simply connected regions. Hence via deleting some of these arcs and 
modifying some arcs with a surgery to simple closed curves as described in the previous paragraph, 
the projection of $\mu$ into $S_0$ coarsely defines a 
marking ${\rm pr}_{S_0}(\mu)$ of $S_0$ (there is a small abuse of notation here),
called the \emph{subsurface projection} of
$\mu$ \cite{MM00}. Here a coarse definition means that the construction depends on 
choices, but any two choices give rise to markings which are uniformly close in the 
marking graph of $S_0$, independent of the subsurface $S_0$.

If $S_0$ is an annulus, then a similar construction applies. In this case a marking 
consists of the choice of a marked point on each boundary component of $S_0$ and
an embedded arc in $S_0$ connecting the two distinct boundary components which is 
disjoint from the marked points. With a bit of care, a subsurface projection is 
defined for annuli as well. We refer to \cite{MM00} for more information.

By the above discussion, for every subsurface $S_0$ of $S$ the marking 
$\mu$ coarsely determines a basepoint $x_0\in {\cal C\cal G}(S_0)$ by choosing 
one of the marking curves (or arcs if $S_0$ is an annulus) 
of ${\rm pr}_{S_0}(\mu)$. 
As the intersection number between any two curves (or arcs) 
of ${\rm pr}_{S_0}(\mu)$ is uniformly bounded, the distance in the curve graph of $S_0$ 
between $x_0$ and any other curve from ${\rm pr}_{S_0}(\mu)$ or from any other marking of $S_0$ 
constructed in the same fashion from $\mu$ 
is uniformly bounded.

Let ${\rm Min}_\cup(S)$ be the space of geodesic laminations on $S$
which are disjoint unions of minimal components. 
Using the basepoint $x_0$ for ${\cal C\cal G}(S_0)$, 
we can extend the subsurface projection
${\rm pr}_{S_0}$ to all of ${\rm Min}_\cup(S)$ as follows. 
Let $\nu=\cup_i\nu_i\in {\rm Min}_\cup(S)$. Then there are three possibilities.
\begin{itemize}
 \item If the lamination $\nu$ is disjoint from $S_0$ up to homotopy, 
define ${\rm pr}_{S_0}(\nu)=x_0$. 
\item If there exist components $\nu_1,\dots,\nu_\ell$ of $\nu$ which are
contained in $S_0$ then 
define ${\rm pr}_{S_0}(\nu)=\cup_{i=1}^\ell\nu_i$.
\item If $\nu\cap S_0$ consists of a collection
of disjoint simple arcs with endpoints on the boundary of $S_0$ 
which coarsely define a point in 
${\cal C\cal G}(S_0)$ then define ${\rm pr}_{S_0}(\nu)$ to be any one of these points.
\end{itemize}
Note that by the definition, ${\rm pr}_{S_0}({\rm Min}_\cup(S))\subset {\rm Min}_\cup(S_0)$, and 
if $\nu$ is a disjoint union of simple closed curves, then the same holds true for
${\rm pr}_{S_0}(\nu)$.

Let again $S=\cup_{i=1}^kS_i$ be a collection of pairwise disjoint 
subsurfaces of $S$. It then follows from the above discussion that 
a choice $\mu$ of a marking of $S$ coarsely determines 
a basepoint $x=(x_1,\dots,x_k)$ for the 
product space 
${\cal C\cal G}(\cup_iS_i)$
consisting of the product of the coarsely well defined 
basepoints $x_i\in {\cal C\cal G}(S_i)$. 


Recall from (\ref{joinformula}) the definition of the sets ${\cal J}(\cup_iS_i)$.
Since the curve graph ${\cal C\cal G}(S_i)$ is a hyperbolic geodesic metric space, 
for every $p>1$ and every $p$-quasi-geodesic ray $\gamma:[0,\infty)\to {\cal C\cal G}(S_i)$,
there exists a coarsely well defined shortest distance projection 
$\Pi_\gamma:{\cal C\cal G}(S_i)\to \gamma$ which extends to the complement of 
the endpoint $\gamma(\infty)\in \partial {\cal C\cal G}(S_i)$ in $\partial {\cal C\cal G}(S_i)$.
The following definition should be thought of viewing ${\cal Y}(\cup_iS_i)$ as the union of
a product of hyperbolic spaces with its visual boundary. As curve complexes are non-proper
hyperbolic spaces and we are ultimately interested in the mapping class group, it is 
necessary to simultaneously work on all subsurfaces, that is, capture what happens 
with subsurface projections on nested
subsurfaces. This is implemented in condition (3) below. In the statement,  $p>1$ is 
a fixed number which fulfills some constraint that will be determined later.

\begin{definition}\label{firststep}
Define a topology on
${\cal Y}(\cup_iS_i)$
by the following requirements.
\begin{itemize}
\item 
The product space ${\cal C\cal G}(\cup_iS_i)$ is equipped with 
the product topology and is an open subset of ${\cal Y}(\cup_iS_i)$. 
\item 
The subspace ${\cal J}(\cup_i S_i)$ is equipped with 
the topology as a join of 
the Gromov boundaries of the curve graphs of $S_i$. 
\item 
Let $\xi=\sum_i a_i \xi_i\in 
\partial {\cal C\cal G}(S_1)*\cdots * \partial {\cal C\cal G}(S_k)$
and after reordering, assume that $a_i>0$ for all $i\leq \ell$ 
and $a_i=0$ for $i>\ell$.
A sequence of points $(y_1^j,\dots,y_k^j)_j\subset {\cal C\cal G}(\cup_iS_i)$  
converges to $\xi$ if the following three conditions are fulfilled. 
\begin{enumerate}
\item For each $i\leq \ell$ the components $y_i^j\in {\cal C\cal G}(S_i)$ 
converge as $j\to \infty$ to 
$\xi_i$ in the coarse Hausdorff topology (and hence they converge in 
${\cal C\cal G}(S_i)\cup \partial {\cal C\cal G}(S_i)$ to $\xi_i$, see \cite{H06}). In particular,
we have $d_{{\cal C\cal G}(S_i)}(y_i^j,x_i)\to \infty$ $(j\to \infty)$.  
\item 
For all $i\leq \ell$ denote by $\Pi_i$ the shortest distance
projection of ${\cal C\cal G}(S_i)$ onto a $p$-quasi-geodesic connecting 
the basepoint $x_i$ to $\xi_i$;  
then 
\[\frac{d_{{\cal C\cal G}(S_i)}(\Pi_i(y_i^j),x_i)}
  {d_{{\cal C\cal G}(S_1)}(\Pi_1(y_1^j),x_1)}
  \to \frac{a_i}{a_1}\quad (j\to \infty).\]
\item Let $i>\ell$ and let $V\subset S_i$ be any subsurface; then 
\[\frac{d_{{\cal C\cal G}(V)}({\rm pr}_V(y_i^j),{\rm pr}_V(\mu))}
  {d_{{\cal C\cal G}(S_1)}(\Pi_1(y_1^j),x_1)}
  \to 0\quad (j\to \infty).\]
\end{enumerate}
\end{itemize} 
\end{definition}

\begin{lemma}\label{topology}
The notion of convergence in Definition \ref{firststep} defines a separable topology on 
${\cal Y}(\cup_iS_i)$ which restricts to the given topology on 
$\partial {\cal C\cal G}(S_1)*\cdots * \partial {\cal C\cal G}(S_k)$
and on ${\cal C\cal G}(\cup S_i)$. The subspace 
$\partial {\cal C\cal G}(S_1)*\cdots * \partial {\cal C\cal G}(S_k)$
is closed in ${\cal Y}(\cup_iS_i)$. 
\end{lemma}  
\begin{proof}
Define a subset $A$ of ${\cal Y}(\cup_i S_i)$ 
to be \emph{closed} if $A_1=A\cap 
{\cal C\cal G}(\cup_iS_i)$ is closed, $A_2=A\cap  
\partial {\cal C\cal G}(S_1)*\cdots * \partial {\cal C\cal G}(S_k)$
is closed and if
furthermore the following holds true. If 
$y_i\subset A_1$ is a sequence which
converges in the sense of Definition \ref{firststep} to 
a point $y\in \partial {\cal C\cal G}(S_1)*\cdots * \partial {\cal C\cal G}(S_k)$, then $y\in A_2$. 
Note that by definition, the empty set is closed, and the same holds true for the 
total space.

We have to show that complements of closed sets
defined in this way fulfill
the axioms of a topology, that is,
they are stable under arbitrary unions and finite
intersections. Equivalently, the family of closed
sets is stable under arbitrary intersections
and finite unions. As this holds true for the closed subsets of 
${\cal C\cal G}(\cup_iS_i)$ and for the closed subsets of 
${\cal J}(\cup_iS_i)=\partial {\cal C\cal G}(S_1)*\cdots * \partial {\cal C\cal G}(S_k)$, all we need to observe
is that taking arbitrary intersections and finite unions is 
consistent with the notion of convergence of points in
${\cal C\cal G}(\cup_iS_i)$ to points in the
join
$\partial {\cal C\cal G}(S_1)*\cdots * \partial {\cal C\cal G}(S_k)$ 
in the sense of Definition \ref{firststep}.

Consistency with arbitrary intersections is straightforward. To show consistency 
with taking finite unions let 
$B_1,\dots, B_\ell\subset {\cal Y}(\cup_iS_i)$ be closed in the above sense. Let 
$y_j\subset \cup_k(B_k\cap {\cal C\cal G}(\cup_iS_i))$ be any sequence which 
converges to a point in ${\cal J}(\cup_iS_i)$ according to the definition of 
convergence.  By passing to a
subsequence, we may assume that $y_j\in B_m$ for a fixed $m\leq \ell$ and
all $j$. As $B_m$ is closed and the subsequence also 
fulfills the requirements for convergence, its limit  
is contained in $B_m\subset \cup_kB_k$. Hence indeed, the notion of 
a closed set is consistent with taking finite unions.
%
\end{proof}

\subsection{Projections and a topology on ${\cal X}(S)$}

As mentioned in the previous subsection, to capture the geometry of the mapping class group using curve
graphs we have to work simultaneously with all subsurfaces, which includes the necessity
to study degenerating sequences of mapping classes whose restrictions to a nontrivial
subsurface are all homotopic to the identity. 

To this end define
for a disjoint union $\cup_{i=1}^kS_i$ of subsurfaces of $S$ the \emph{set}
\begin{equation}
{\cal Z}(\cup_iS_i)=\cup_{I}{\cal Y}(\cup_{i\in I}S_{i})*{\cal J}(\cup_{j\in \{1,\dots,k\}
    \setminus I}S_j)\end{equation}
where $I$ runs through all (possibly empty) subsets of the index set $\{1,\dots,k\}$
and we identify the points which appear in several ways in
this union. Namely, 
for every $I$, the space
${\cal Y}(\cup_{i\in I}S_{i})*{\cal J}(\cup_{j\in \{1,\dots,k\}
  \setminus I} S_j)$ contains ${\cal J}(\cup_{i=1}^kS_i)$, the join of the
boundaries of the curve graphs of the surfaces $S_i$, and we perform the obvious
identification of the points which coincide as formal sums $\sum_ia_i \xi_i$ where 
$\sum_ia_i=1$ and $\xi_i\in \partial {\cal C\cal G}(S_i)$.  No other identifications are made.
Note that
${\cal Z}(\cup_iS_i)$ contains ${\cal Y}(\cup_{i\in I}S_i)$ for all $I\subset \{1,\dots,k\}$.

\begin{lemma}\label{topofz}
There exists a unique 
topology on ${\cal Z}(\cup_i S_i)$ with the property that 
a set $U\subset {\cal Z}(\cup_iS_i)$ is open if and only if
its intersection with each of the subspaces
${\cal Y}(\cup_{i\in I}S_i)*{\cal J}(\cup_{j\in \{1,\dots,k\}\setminus I}S_j)$ is open.
\end{lemma}
\begin{proof} For every $I\subset \{1,\dots,k\}$, the set
${\cal J}(\cup_iS_i)$ is a closed subspace of 
\[{\cal Y}(\cup_{i\in I}S_i)*{\cal J}(\cup_{j\in \{1,\dots,k\}\setminus I}S_j),\]
equipped with the topology of a join. Thus the topology described in the lemma is just the quotient topology on
the quotient of the disjoint union of the spaces 
${\cal Y}(\cup_{i\in I}S_i)*{\cal J}(\cup_{j\in \{1,\dots,k\}\setminus I}S_j)$ by the
closed equivalence
relation which identifies the points in the subspaces ${\cal J}(\cup_iS_i)$. 
\end{proof}

We next define a 
projection
\[{\rm pr}_{{\cal Z}(\cup S_i)}:{\cal X}(S)\to {\cal Z}(\cup_iS_i)\]
as follows. Let $\xi=\sum_{j=1}^m a_j\xi_j\in {\cal X}(S)$ 
with $a_j>0$ and 
$\sum_ja_j=1$ and write as before ${\rm supp}(\xi)=\cup_j\xi_j$.
After perhaps a reordering of the components $\xi_j$ and a reordering of the surfaces $S_i$, assume that for some
$u\leq \min\{k,m\}$ the components
$\xi_1,\dots,\xi_u$ fill the 
subsurfaces $S_1,\dots,S_u$, 
that is, they define points in
$\partial {\cal C\cal G}(S_i)$, with the convention of remembering labels of simple closed curves, 
and that for no $j>u$, the component 
$\xi_j$ fills any of the surfaces $S_i$. 
As the components of ${\rm supp}(\xi)$ are disjoint, 
this implies that if
$s,t>u$, if $i\in \{u+1,\dots,k\}$ and if the subsurface projections
${\rm pr}_{S_i}(\xi_s),{\rm pr}_{S_i}(\xi_t)$ 
of $\xi_s,\xi_t$ into $S_i$ are not empty, then they
are of uniformly bounded distance in ${\cal C\cal G}(S_i)$. Recall that this makes
sense even if $\xi_s,\xi_t$ are different from simple closed curves. 
If the lamination ${\rm supp}(\xi)=\cup_i\xi_i$ is disjoint from 
the subsurface $S_\ell$, then the 
projection component is defined to be the basepoint of 
${\cal C\cal G}(S_\ell)$ constructed from the base marking.

Define
\begin{align*}
{\rm pr}_{{\cal Z}(\cup S_i)}(\sum_{j=1}^m a_j\xi_j) &=\sum_{j=1}^ua_j\xi_j+
(1-\sum_{j=1}^ua_j)
({\rm pr}_{{\cal C\cal G}(\cup_{i\geq u+1}S_i)}(\cup_{j\geq u+1}\xi_j))\\
&\in {\cal C\cal G}(\cup_{i=u+1}^kS_i)* {\cal J}(\cup_{i=1}^u S_i).\end{align*}
Here the term on the right hand side is understood in the following sense. 
First, if one of the surfaces $S_i$ $(i\leq u)$ is an annulus then the label of $\xi_i$
is remembered in ${\rm pr}_{{\cal Z}(\cup_iS_i)}(\sum_j a_j \xi_j)$. Second,  
for some $\ell\in \{u+1,\dots,k\}$ let us consider the subsurface $S_\ell$. 
If there exists some $s>u$ such that $\xi_s$ intersects
$S_\ell$, then the component in $S_\ell$ of the projection 
${\rm pr}_{{\cal C\cal G}(\cup_{i\geq u+1}S_i)}(\cup_{j\geq u+1}\xi_j)$ is a point in 
${\cal C\cal G}(S_\ell )$. Although this projection depends on choices, it is 
coarsely well defined, that is, well defined up to a uniformly bounded error.
Finally, recall that we take the join
of the product of the curve graphs of the subsurfaces $S_i$ for $i\leq u$ and 
the join of the Gromov boundaries $\partial {\cal C\cal G}(S_i)$ for $i\geq u+1$, so the sum
$1- \sum_{j=1}^ua_j$ is treated as a single coefficient.

\noindent
{\bf Requirement 2:} A sequence 
$\xi^j=\sum_m a_m^j \xi_m^j\subset {\cal X}(S)$ converges to 
$\zeta=\sum_{i=1}^kb_i\zeta_i\in {\cal X}(S)$
if the following holds true. For each $i\leq k$ let $S_i$ be the subsurface of $S$ filled 
by $\zeta_i$ and 
put $S_{k+1}=S\setminus \cup_{i=1}^kS_i$; then
\[{\rm pr}_{{\cal Z}(\cup_{i=1}^{k+1}S_i)}(\xi^{j})
\to \zeta \text{ in }
{\cal Z}(\cup_{i=1}^{k+1}S_i)\supset {\cal J}(\cup_{i=1}^{k+1}S_i)\supset
{\cal J}(\cup_{i\leq k} S_i).\] 

In other words, for all $i,j$ the subsurface projection of $\xi^j$ to $S_i$ is either
a filling lamination or empty or coarsely determines a  point in 
${\cal C\cal G}(S_i)$, and these projections, equipped with the weight inherited 
from the weights of the components of $\xi^j$, converge in ${\cal Z}(\cup_{i=1}^{k+1}S_i)$ to $\zeta$. 

\begin{remark}\label{comment}
It follows from the above description that for this notion of convergence,
the following holds true. Let $\xi^j$ be a sequence in ${\cal X}(S)$
consisting of minimal geodesic laminations which converges to a point 
$\zeta=\sum_u b_u\zeta_u$.

\begin{enumerate}
\item[(a)]
The lamination 
${\rm supp}(\zeta)$ is a sublamination of
the limit in the Hausdorff topology of any
convergent subsequence of the sequence ${\rm supp}(\xi^j)$.
\item[(b)]  For each $j$ let 
$\eta^j$ be a minimal geodesic lamination disjoint from $\xi^j$ 
(we allow $\eta^j=\xi^j$) and let $s_i\in [0,1]$. Then any limit of a convergent subsequence
of the sequence $\nu^j=s_i\xi^j+(1-s_i)\eta^j$ is of the form 
$s\zeta +(1-s)\eta$ where $\eta$ is a limit of a subsequence of 
the sequence $\eta^j$ and where $s\in [0,1]$.
\item[(c)] Requirement (1) in the definition of the topology on 
${\cal X}(S)$ is a consequence of requirement (2) and was included for added clarity. 
\item[(d)] If $\xi^j=\sum_{i=1}^{m^j} a_i^j\xi_i^j\subset {\cal X}(S)$ is any sequence so that for some
fixed $1\leq m\leq m^j$ it holds $\sum_{i=1}^ma_i^j\to 1$ and if the sequence 
$\frac{1}{\sum_{i=1}^ma_i^j}\sum_{i=1}^m a_i^j\xi_i^j$ converges to $\zeta\in {\cal X}(S)$, then 
the same holds true for the sequence $\xi^j$. This makes the topology weak
enough for our goal and is motivated by the weak$^*$-topology on the 
space of measured geodesic laminations. 
\hfill $\blacksquare$
\end{enumerate}
\end{remark}

\begin{definition}\label{geometrictop}
A subset $A\subset {\cal X}(S)$ is called \emph{closed for the geometric
topology of ${\cal X}(S)$} if the following holds true. Let $\xi_i\subset A$ be
any sequence which converges to a point $\xi\in {\cal X}(S)$ in the sense
described by the requirements (2); then $\xi\in A$.
\end{definition}

Recall that for any collection $S_1,\dots,S_k$ of pairwise disjoint subsurfaces of $S$,
the space ${\cal J}(\cup_{i=1}^kS_i)$ is equipped with a natural topology 
as a join of the Gromov boundaries of the curve graphs of the subsurfaces $S_i$.
The following statement is the first main step towards the proof of 
Theorem \ref{main}.

\begin{proposition}\label{top2}
\begin{enumerate}
\item 
Closed subsets of ${\cal X}(S)$ in the sense of Definition \ref{geometrictop}
define a separable Hausdorff topology ${\cal O}$ on ${\cal X}(S)$. 
\item For any collection $S_1,\dots,S_k$ of pairwise disjoint subsurfaces,
the natural inclusion 
${\cal J}(\cup_{i=1}^kS_i)\to ({\cal X}(S), {\cal O})$ is an embedding.
\end{enumerate}
\end{proposition}
\begin{proof} Let ${\cal O}\subset {\cal X}(S)$ be the family of all subsets of 
${\cal X}(S)$ whose complement is closed in the above sense. 
Sets in ${\cal O}$ are called \emph{open}. We have to show 
that ${\cal O}$ defines a topology on 
${\cal X}(S)$. 

As the empty
set and the entire space ${\cal X}(S)$ are open, to show that  
${\cal O}$ is indeed a topology on ${\cal X}(S)$ it suffices to show
that arbitrary unions of open sets are open, and that finite intersections
of open sets are open as well. Or, equivalently, arbitrary intersections of 
closed sets are closed, and finite unions of closed sets are closed.
This can be established using exactly the same reasoning 
as in the proof of Lemma \ref{topology}.

Namely, that the collection of closed sets is stable under arbitrary intersections
is immediate from the definition. So let $B_1,\dots, B_k$ be closed sets and let
$B=\cup_iB_i$. Choose a sequence $\xi_i\subset B$ which converges in the sense of 
requirement (2) to some point $\zeta$. By passing to a subsequence, we may assume that
$\xi_i\in B_\ell$ for some $\ell\leq k$ and all $i$. But then 
$\zeta\in B_\ell\subset B$ as $B_\ell$ is closed which completes the proof that
${\cal O}$ is indeed a topology on ${\cal O}$.

We show next the second property claimed in the proposition.
Thus let $S_1,\dots,S_k$ be a 
collection  of pairwise disjoint subsurfaces
of $S$. Our goal is to show that  
the inclusion ${\cal J}(\cup_{i=1}^kS_i)\to ({\cal X}(S),{\cal O})$ 
is an embedding. Since the inclusion is injective, and closed 
sets in both spaces are defined via convergence of sequences,
this is equivalent to stating that 
a sequence $\xi^j=\sum_ia_i^j\xi_i^j\subset {\cal J}(\cup_{i=1}^kS_i)$
converges in $({\cal X}(S),{\cal O})$ to a point 
$\zeta\in {\cal J}(\cup_{i=1}^kS_i)$ if and only if 
$\xi^j$ converges in ${\cal J}(\cup_{i=1}^kS_i)$ to $\zeta$.
However, putting $S_{k+1}=S\setminus \cup_iS_i$, 
this is immediate from the definition of the 
topology on ${\cal Z}(\cup_{i=1}^{k+1}S_i)$ and 
requirement (2) in the definition of convergence in 
${\cal X}(S)$ and shows
the second part of the proposition.

Since each of the spaces ${\cal J}(\cup_{i=1}^kS_i)$ is
a finite join of separable metrizable spaces (namely, the Gromov
boundary of a curve graph of a subsurface of $S$) and hence
separable metrizable, 
the second part of the proposition implies 
that $({\cal X}(S),{\cal O})$ is
a countable union of (in general not disjoint) 
separable metrizable spaces and hence is 
separable. 

To show that the topology is Hausdorff let 
$\xi=\sum_ia_i\xi_i\not=\zeta=\sum_jb_j\zeta_j\in {\cal X}(S)$. 
We have to show that $\xi,\zeta$ have disjoint neighborhoods. 

If this is not the case, then any neighborhoods $U_\xi$ of $\xi$ 
and $U_\zeta$ of $\zeta$ intersect nontrivially. 
Since ${\cal X}(S)$ is separable, and since points in ${\cal X}(S)$ 
are closed by construction,
we conclude that there is a sequence 
$\xi^j\subset {\cal X}(S)$ 
which converges both to $\xi,\zeta$. 
But for the notion of convergence used to define the topology ${\cal O}$,
the limit of a converging sequence is unique. Thus ${\cal O}$ is indeed Hausdorff 
which completes the proof the first part of the proposition. 
\end{proof}

\begin{example}\label{limit}
i) Let $\phi\in {\rm Mod}(S)$ be a pseudo-Anosov element.
Then $\phi$ acts as a loxodromic isometry on the curve graph of $S$,
with attracting and repelling fixed points $\nu_+,\nu_-\in \partial {\cal C\cal G}(S)$.
Let $\mu\in {\cal X}(S)$ be any minimal geodesic lamination 
which is distinct from the repelling 
fixed point $\nu_-$ of $\phi$ and hence intersects $\nu_-$. 
Then  
$\phi^j\mu\to \nu_+$ $(j\to \infty)$ in the coarse Hausdorff topology 
and therefore $\phi^j\mu\to \nu_+$ in ${\cal X}(S)$. 

ii) Now let us assume that $S_0\subset S$ is a proper connected
subsurface different from an annulus and a pair of pants 
and that $\phi\in {\rm Mod}(S)$ 
restricts to a pseudo-Anosov mapping class on $S_0$ 
and to the trivial mapping class on $S-S_0$. Let 
$\nu_+,\nu_-\in \partial {\cal C\cal G}(S_0)$ 
be the attracting and repelling geodesic lamination for the action of $\phi$ on $S_0$,
respectively. 
Let furthermore $\mu\not=\nu_- \in {\cal X}(S)$ be any \emph{minimal} 
geodesic lamination on 
$S$. Then there are two possibilities. In the first case, 
$\mu$ is supported in 
$S-S_0$. Then we have $\phi^j(\mu)=\mu$ for all $j$.
However, if $\mu$ intersects $S_0$, then either $\mu=\nu_{-}$ or 
$\mu$ intersects $\nu_-$ and 
we have $\phi^j(\mu)\to \nu_+$ $(j\to \infty)$ in ${\cal X}(S)$.

Namely, if $\mu$ intersects $S_0$ then the subsurface projection of
$\mu$ into any subsurface disjoint from $S_0$ is a collection of
arcs intersecting $\partial S_0$. In particular, the subsurface projection
into any subsurface $V$ of $S-S_0$ is a point of ${\cal C\cal G}(V)$.
Since $\phi$ can be represented by a diffeomorphism which fixes
$S-S_0$ pointwise, it acts trivially on ${\cal C\cal G}(V)$ which 
yields the above statement. \hfill $\blacksquare$
\end{example}

\begin{corollary}\label{countable}
$({\cal X}(S),{\cal O})$ is a Lindel\"of space.
\end{corollary}
\begin{proof} 
We have to show that any open cover of ${\cal X}(S)$ has a countable 
subcover. To this end let ${\cal U}$ be such an open cover.
List the countably many spaces ${\cal J}(\cup_i S_i)$ as 
${\cal J}_1,{\cal J}_2,\dots $. Since for each $i$, the space 
${\cal J}_i$ is separable and metrizable, the restriction of 
${\cal U}$ to ${\cal J}_i$, which is an open covering of
${\cal J}_i$, has a countable subcover, say by sets 
$U_i^1,U_i^2,\dots$.
The standard diagonal argument shows that
 the union ${\cal V}=\cup_{i,j}U_i^j$ consists of countably many sets, 
and for each $i$, the sets from ${\cal V}$ cover ${\cal J}_i$.
Since ${\cal X}(S)=\cup_i{\cal J}_i$ (as a set), this shows that ${\cal V}$
is a countable subcover of the cover ${\cal U}$.
In other words, ${\cal X}(S)$ is a Lindel\"of space as claimed.
\end{proof}

\begin{proposition}\label{compactbd}
$({\cal X}(S),{\cal O})$ is compact.
\end{proposition}
\begin{proof} 
As by Proposition \ref{top2} and 
Corollary \ref{countable}, the space ${\cal X}(S)$ is a separable
Lindel\"of space, moreover it is Hausdorff, 
to show that ${\cal X}(S)$ is compact it suffices to show that 
${\cal X}(S)$ is sequentially compact.

Thus 
let $\xi^j=\sum_ia_i^j\xi_i^j\subset {\cal X}(S)$ be any sequence.
We have to construct a convergent subsequence. 
Since the space
of geodesic laminations equipped with the Hausdorff topology is compact, 
by passing to a subsequence we may assume that the 
geodesic laminations
${\rm supp}(\xi^j)=\cup_i\xi_i^j$ converge in the Hausdorff topology to 
a geodesic lamination $\hat \zeta$ with minimal components 
$\zeta_1,\dots,\zeta_k$. 

For each $i\leq k$ let $S_i\subset S$ be the subsurface of $S$ filled by $\zeta_i$. 
Assume by passing to a subsequence that 
\[\xi^j=\sum_{i=1}^u a_i^j \xi_i^{j} + \sum_{\ell >u}a_\ell^j\xi_\ell^j\]
for all $j$ where for each $i\leq u$, the component $\xi_i^j$ fills $S_i$ and that 
none of the components $\xi_\ell^j$ for $\ell>u$ fills any of the surfaces $S_i$. There is no assumption
on the supports of the laminations $\xi_\ell^j$ for $\ell>u$ (besides of course that 
they are pairwise disjoint and disjoint from the surfaces $S_i$ for $i\leq u$). 
By passing to another subsequence, we 
may assume that for $i\leq u$, the labels $\pm$ of the components $\xi_i^j$ 
are constant along the sequence, and that 
the weights $a_i^j\in (0,1]$ of the components $\xi_i^j$ converge to weights
$b_i\geq 0$. In particular, the sums $1-\sum_{i\leq u}a_i^j$ converge to 
$1-\sum_{i\leq u}b_i=\kappa$. 

Since ${\rm supp}(\xi^j)\to \hat \zeta$ in the Hausdorff topology, we know that for 
each $i\leq u$, the laminations $\xi_i^j$ converge in the coarse Hausdorff topology to
$\zeta_i$ and hence $\xi_i^j$ converges to $\zeta_i$ in 
$\partial {\cal C\cal G}(S_i)$ \cite{H06}. 
Thus if $\kappa=0$ then by the definition of the topology
on ${\cal X}(S)$, we know that $\xi^j\to \sum_{i=1}^u b_i\zeta_i$ and we are done.
Namely, putting $S_{u+1}=S\setminus \cup_{i=1}^uS_i$, in this case we  immediately obtain that 
${\rm pr}_{{\cal Z}(\cup_{i=1}^{u+1}S_i)}\xi^j\to \sum_{i=1}^ub_i\zeta_i$. 
Thus we are left with the case $\kappa >0$. 
Viewing $\xi^j=(\sum_{i\leq u}a_i^j\zeta_i)+(\sum_{i\geq u+1}a_i^j\xi_i^j)$
as points in the join of two subspaces of ${\cal X}(S)$, using the above argument 
it suffices in fact to assume that for no
$j$ there exists a component of ${\rm supp}(\xi^j)$ which fills any of the subsurfaces $S_i$.

Then for each $i$, we can consider the subsurface projection 
${\rm pr}_{S_i}({\rm supp}(\xi^j))$ 
of ${\rm supp}( \xi^j)$ into the 
surface $S_i$. Furthermore, by passing to another subsequence,
we may assume that for all $j$ and all $i\leq k$, 
this subsurface projection is non-empty
since the geodesic lamination $\zeta_i$ which fills $S_i$ 
is contained in the limit with respect to the Hausdorff topology 
of the sequence of laminations ${\rm supp}( \xi^j)$. Put differently, 
we may assume that for each $i$ and all $j$, 
the subsurface projection ${\rm pr}_{S_i}({\rm supp}(\xi^j))$ of 
the lamination ${\rm supp}(\xi^j)$ into the subsurface $S_i$ is a 
coarsely well defined point in 
${\cal C\cal G}(S_i)$. Furthermore, using once more that $\zeta_i$ fills $S_i$ and that
$\zeta_i$ is contained in the Hausdorff limit of the sequence ${\rm supp}(\xi^j)$, 
if we denote by $x_i$ the fixed basepoint in ${\cal C\cal G}(S_i)$, then we know that 
$d_{{\cal C\cal G}(S_i)}({\rm pr}_{S_i}({\rm supp}(\xi^j)),x_i)\to \infty$ $(j\to \infty)$.

By passing to another subsequence and reordering indices, 
we may assume that
\[a_1^j=d_{{\cal C\cal G}(S_1)}({\rm pr}_{S_1}({\rm supp}(\xi^j)),x_1)
\geq a_i^j=
d_{{\cal C\cal G}(S_i)}({\rm pr}_{S_i}({\rm supp}(\xi^j)),x_i)\]
for all $i\geq 2$ and all
$j$. Passing to another subsequence, we may assume furthermore that
$a_i^j/a_1^j\to a_i\in [0,1]$ for all $i\geq 2$. Put $a_1=1$; then we have
$\sum_ua_u\geq 1$ and hence defining $b_i=a_i/\sum_ua_u>0$, we conclude that
$\sum_u b_u =1$. It now follows from the definition of the topology on 
${\cal X}(S)$ that $\xi^j\to \sum_ib_i\zeta_i$. This completes the proof that
${\cal X}(S)$ is sequentially compact.
\end{proof}

\begin{lemma} \label{action}
${\rm Mod}(S)$ acts on ${\cal X}(S)$ as a group of transformations.
\end{lemma}
\begin{proof} Observe first that
by construction, ${\rm Mod}(S)$ acts on ${\cal X}(S)$ as a group of bijections
(equivalently, transformations for the discrete topology). 
Thus it suffices to show that this action is continuous for the topology ${\cal O}$.

By the definition of ${\cal O}$, for this it suffices to show the following. Let 
$\xi^j$ be a sequence converging for the topology ${\cal O}$ to a point $\xi$.
Then for every $\phi\in {\rm Mod}(S)$, the sequence $\phi(\xi^j)$ converges
to $\phi(\xi)$. Namely, if this holds true then any $\phi\in {\rm Mod}(S)$ is a closed map,
that is, it maps closed subsets of $({\cal X}(S),{\cal O})$ to closed sets. Since 
each $\phi$ is a bijection, with inverse $\phi^{-1}$, 
this implies that the preimage under $\phi$ 
of any open set is open and hence $\phi$ is continuous. 


We need to check that 
requirement (2) for convergence is natural for the action of $\phi\in {\rm Mod}(S)$.
Now if $S_1,\dots,S_k$ is a partition of $S$ into disjoint 
subsurfaces, then the same holds true for $\phi(S_1),\dots, \phi(S_k)$, and for any 
geodesic lamination $\nu$, we have ${\rm pr}_{{\cal Y}(\cup_i\phi(S_i))}(\phi(\nu))=\phi 
({\rm pr}_{{\cal Y}(\cup_iS_i)}(\nu))$ up to replacing the basepoints $y_i$ 
of ${\cal C\cal G}(\phi(S_i))$ by 
$\phi(x_i)$. Note that $\phi$ also naturally acts on orientations of simple closed 
curves on $S$ as no oriented simple closed curve on $S$ is freely homotopic to its
inverse and hence $\phi$ acts on labelled simple closed curves. 
As for all $i$, we have 
$d_{{\cal C\cal G}(\phi(S_i))}({\rm pr}_{\phi(S_i)}(\xi^j),\phi(x_i))=
d_{{\cal C\cal G}(S_i)}(\xi^j,x_i)\to \infty$ $(j\to \infty)$ 
and the determination of 
the weights of the limit points are computed using ratios of distances to the basepoint 
defined by subsurface projections, with 
the distances tending to infinity along the sequence, we conclude that the requirement (2) 
in the definition of convergence 
is fulfilled for $\phi(\xi^i)$ if it is fulfilled for $\xi^i$. Thus indeed, ${\rm Mod}(S)$ acts on 
${\cal X}(S)$ as a group of transformations which shows the lemma.
\end{proof}

\begin{definition}\label{geometricbd}
The space $({\cal X}(S),{\cal O})$ is called the \emph{geometric boundary}
of ${\rm Mod}(S)$.
\end{definition}

Let us note another naturality property of the geometric boundary of 
${\rm Mod}(S)$. Namely, 
if $S_0\subset S$ is any essential subsuface, then we can construct a geometric 
boundary ${\cal X}(S_0)$ for the mapping class group
${\rm Mod}(S_0)$ of isotopy classes of homeomorphisms of $S_0$ fixing the boundary pointwise. 
As a set, this is a subset of the geometric boundary of 
$S$ which includes the Gromov boundary of the curve graph for peripheral annuli. 
The above construction immediately yields

\begin{corollary}\label{natural}
If $S_0\subset S$ is any subsurface of $S$, then the geometric boundary of 
${\rm Mod}(S_0)$ is 
a closed subspace of the geometric boundary of ${\rm Mod}(S)$. 
\end{corollary}

From what we achieved so far, the second part of Proposition \ref{details} from the introduction is now immediate.

\begin{proof}[Proof of (2) of Proposition \ref{details}]
Let $\phi\in {\rm Mod}(S)$ be a Nielsen-Thurston mapping class. Then there are pairwise disjoint 
subsurfaces $S_1,\dots,S_{k+1}$, some of them may be annuli, so that $S=\cup_iS_i$, that 
$\phi$ preserves $S_i$ for each $i$ and that the following holds true. If $i\leq k$ and if
$S_i$ is distinct from an annulus, then $S_i$ also is distinct from a pair of pants and $\phi$ acts
on $S_i$ as a pseudo-Anosov mapping class preserving each boundary component of $S_i$. 
If $i\leq k$ and if $S_i$ is an annulus, then $\phi$ acts on $S_i$ as a Dehn (multi)-twist, and the restriction of 
$\phi$ to $S_{k+1}$ is the trivial mapping class. 

If $\xi\in {\cal X}(S)$ is fixed by $\phi$ and not contained in the  obvious fixed point set, then 
$\xi=\sum_j a_j\xi_j$ where either for at least one $j$, the lamination $\xi_j$ is not supported in 
any of the subsurfaces $S_i$, or for at least one $j$ the lamination is supported in one of the surfaces
$S_i$ but it is not fixed by $\phi\vert S_i$. But then Example \ref{limit} and the definition of the topology on 
${\cal X}(S)$ together yield that $\xi$ is not a fixed point for the action of $\phi$ on ${\cal X}(S)$.
\end{proof}

\section{A small boundary for  ${\rm Mod}(S)$}\label{small}

The purpose of this section is to  construct a topology on 
$\overline{\cal T}(S)={\cal T}_\epsilon(S)\cup {\cal X}(S)$ which restricts to the 
given topologies on ${\cal T}_\epsilon(S)$ and on 
${\cal X}(S)$ and such that with respect to this topology, $\overline{\cal T}(S)$ is a compact
${\rm Mod}(S)$-space. 
The construction of this topology is carried out using
Teichm\"uller geometry. 

\begin{remark}\label{dhs}
In \cite{DHS17}, there is a construction of a topology 
on ${\rm Mod}(S)\cup {\cal X}(S)$ (where however the topology on ${\cal X}(S)$
differs from the one we introduced) using the combinatorics of hierarchical hyperbolic
spaces and such that with respect to this topology, 
${\rm Mod}(S)\cup {\cal X}(S)$ is compact. 
Instead our construction uses geometric tools relying 
on Teichm\"uller theory which are not available for hierarchically hyperbolic 
spaces, with the aim at 
capturing features of ${\rm Mod}(S)$ which resemble properties of a 
${\rm CAT}(0)$ group. It is possible that there is a relation of 
our construction to the one in 
 \cite{DHS17}.  A verification would require a translation of our construction into
 the context of hierarchical hyperbolicity, and it is unclear whether this can be done, see
however \cite{DMS25}. 
\end{remark}

\subsection{The thick part of Teichm\"uller space} 

In this section we consider a connected oriented surface $S$ of finite type without boundary different
from a sphere with at most three punctures. The Teichm\"uller space
${\cal T}(S)$ is the space of all complete finite area hyperbolic metrics on $S$.
By the collar lemma for hyperbolic surfaces, there exists
a number $\epsilon_0>0$ not depending on $S$ 
with the following property. For any  hyperbolic metric on $S$, 
any two closed geodesics
$\gamma_1,\gamma_2$ of length
$\ell(\gamma_1),\ell(\gamma_2)\leq \epsilon_0$ are disjoint.

Let ${\rm syst}:{\cal T}(S)\to (0,\infty)$ be the
\emph{systole function} which associates to a point in ${\cal T}(S)$ its systole, that is, the shortest
length of an essential closed curve (closed geodesic) on $S$.
For $\epsilon \leq \epsilon_0$ define the \emph{$\epsilon$-thick part}
${\cal T}_\epsilon(S)$ of the Teichm\"uller space ${\cal T}(S)$ 
of marked hyperbolic metrics on $S$ by
\[{\cal T}_{\epsilon}(S)=\{ X\in {\cal T}(S)\mid
  \text{ syst}(X)\geq \epsilon\}.\]
The following statement is well known. We refer to 
Proposition 1.1 of \cite{JW10} for an explicit account.

\begin{theo}\label{model}
For $\epsilon < \epsilon_0$, the following holds.
  \begin{enumerate}
\item     
The subspace ${\cal T}_\epsilon(S)\subset {\cal T}(S)$ is non-empty, closed, connected
and stable under ${\rm Mod}(S)$.
Its quotient under the action of ${\rm Mod}(S)$ is
  compact. 
\item ${\cal T}_\epsilon(S)$ is a real-analytic manifold with corners and
  hence admits a ${\rm Mod}(S)$-invariant triangulation such that
  ${\rm Mod}(S)\backslash {\cal T}_\epsilon(S)$ is a finite
  $CW$-complex.
\end{enumerate}
\end{theo}

As a consequence, ${\cal T}_\epsilon(S)$ is a topological manifold with boundary and interior
$\mathring{\cal T}_\epsilon(S)=\{X\mid {\rm syst}(X)>\epsilon\} \subset {\cal T}_\epsilon(S)$.

There is a coarsely well defined map
\[\Upsilon:{\cal T}(S)\to {\cal C\cal G}(S)\] which
maps a marked hyperbolic metric to a closed non-contractible curve
 of minimal length. 
Coarsely well defined means that the map depends on choices, but 
the images of a point $X\in {\cal T}(S)$ for 
any two choices of such a map are of distance at most two. 

Call a map $\Psi:{\cal T}(S)\to {\cal T}(S)$ \emph{coarsely
  $\Upsilon$-invariant} if
$d(\Upsilon(\Psi(X)),\Upsilon(X))\leq 2$ for all $X$. The following
is due to Ivanov \cite{Iv02} if one replaces the mapping class group by a torsion free
subgroup. The full version is 
Theorem 1.2 of \cite{JW10}, see 
also Theorem 3.9 of \cite{J14}.

\begin{theo}[Ji-Wolpert]\label{ji}
For $\epsilon <\epsilon_0/3$  there is a ${\rm Mod}(S)$-equivariant
 coarsely $\Upsilon$-invariant deformation retraction   
${\cal T}(S)\to {\cal T}_\epsilon(S)$. 
\end{theo}

The deformation retraction is constructed as follows. First Ji and Wolpert construct a ${\rm Mod}(S)$-invariant 
continuous uniquely integrable vector field $V$ on ${\cal T}(S)$ with the following properties (p.9 of \cite{JW10}).
\begin{enumerate}
\item[(a)] $V ({\rm syst}) =1$ on $\{{\rm syst} \leq 2\epsilon\}$ and
\item[(b)] $V$ vanishes on $\{ {\rm syst} \geq 3\epsilon\}$.
\end{enumerate}
The deformation retraction is then given by the time $\epsilon$-map of the flow defined by $V$.
Note however that the image of ${\cal T}(S)$ under this map is the interior
of ${\cal T}_\epsilon(S)$. Since the time $\epsilon$ map of a continuous flow is a homeomorphism,
we obtain the following statement as an immediate consequence. 

\begin{corollary}\label{homeomorphism}
For every $\epsilon <\epsilon_0/3$ there is a ${\rm Mod}(S)$-equivariant homeomorphism 
$\Lambda_\epsilon:{\cal T}(S)\to \mathring {\cal T}_\epsilon(S)$.
\end{corollary}

For our purpose, the difficulty arises  that we need to construct
contractible subsets of ${\cal T}_\epsilon(S)$ and not of its interior. But the closure of 
a contractible open set in a smooth manifold may not be contractible. 
The following construction will allow us to address this issue.

Define the \emph{small closure} $\overline{A}_{\rm small}$ of a subset 
$A$ of $\mathring {\cal T}_\epsilon(S)$ to be the union of $A$ with the set of all points 
$z\in \partial {\cal T}_\epsilon(S)$
so that $z$ has a neighborhood $U$ in ${\cal T}_\epsilon(S)$ with 
$U\cap ({\cal T}_\epsilon(S)\setminus \partial {\cal T}_\epsilon(S))\subset A$. 
Note that $\overline{A}_{\rm small}\setminus A$ is an \emph{open} subset of $\partial {\cal T}_\epsilon(S)$. 
More precisely, we have.

\begin{lemma}\label{openinclosed}
\begin{enumerate}
\item The small closure in ${\cal T}_\epsilon(S)$ of an open subset of 
$\mathring{\cal T}_\epsilon(S)$ is open in ${\cal T}_\epsilon(S)$.
\item If $U\subset {\cal T}_\epsilon(S)$ is open, then $U\subset 
\overline{U\cap \mathring{\cal T}_\epsilon(S)}_{\rm small}$.
\end{enumerate}
\end{lemma}
\begin{proof}
If $U\subset \mathring{\cal T}_\epsilon(S)$ is open, then as $\mathring {\cal T}_\epsilon(S)\subset 
{\cal T}_\epsilon(S)$ is open, a point $x\in U\subset \overline{U}_{\rm small}$ has a neighborhood
in ${\cal T}_\epsilon(S)$ which is contained in $U$.

On the other hand, if $x\in \overline{U}_{\rm small}\setminus U$ then it follows from the definition
of $\overline{U}_{\rm small}$ that $x$ has a neighborhood in 
${\cal T}_\epsilon(S)$ entirely contained in $\overline{U}_{\rm small}$. This shows the first part of the lemma.

The second part of the lemma is immediate from the definitions.
\end{proof}

\begin{lemma}\label{smallcontr}
The small closure of a contractible subset of ${\cal T}_\epsilon(S)$ is contractible.
\end{lemma}
\begin{proof}
It suffices to deformation retract the small closure $\overline{A}_{\rm small}$ of a contractible subset $A$ 
of ${\cal T}_\epsilon(S)$ into 
$A$. The composition of this deformation retraction with a deformation retraction of $A$ to a point then shows that
$\overline{A}_{\rm small}$ is contractible. 

Since ${\cal T}_\epsilon(S)\subset {\cal T}(S)$ is a manifold with corners, 
for $z\in \overline{A}_{\rm small}\setminus A$ there is an open neighborhood $U_z$ of $z$ in 
$\overline{A}_{\rm small}\cap {\cal T}_\epsilon(S)\setminus {\cal T}_{2\epsilon}(S)$
which is homeomorphic to the set 
$B_0=\{(x_1,\dots,x_n)\in \mathbb{R}^n\mid \sum_ix_i^2<1,x_1\geq 0\}$,
with $z$ corresponding to $0$, and such that $B_0\cap \{x_1>0\}\subset A$. Let again
$V$ be the vector field on ${\cal T}(S)$ with properties (a),(b) as preceding Corollary \ref{homeomorphism}
and let $\psi:B_0\to [0,1]$ be a smooth function with $\psi >0$ on $\sum_ix_i^2<1$ and 
$\psi\equiv 0$ on $\sum_i x_i^2=1$. Via the identification of $B_0$ with the neighborhood $U_z$  
of $z$ in $\overline{A}_{\rm small}$ and constant extension to zero outside this neighborhood, 
we can view $\psi$ as a continuous
function on ${\cal T}_\epsilon(S)$ which is supported in the closure $\overline{U}_z$ of $U_z$. 
The vector field 
$\psi V$ then defines a flow supported in $\overline{U}_z$ 
(which is just a local continuous time change of the flow of $V$) which  
deformation retracts $\overline{U}_z$ into $\overline{U}_z\setminus (U_z\cap (\overline{A}_{\rm small}\setminus A))$. 
Its time one map is 
a homeomorphism onto its image.

As a consequence, for every $z\in \overline{A}_{\rm small}\setminus A$ there is a deformation retraction of 
$\overline{A}_{\rm small}$ which moves a neighborhood of $z$ in $\overline{A}_{\rm small}$ into $A$ 
and such that the intersection of the resulting set with $\partial {\cal T}_\epsilon(S)$ is contained in the complement of 
a neighborhood of $z$ 
in the intersection of $\overline{A}_{\rm small}$ with $\partial {\cal T}_\epsilon(S)$.

Since $\overline{A}_{\rm small}\setminus A$ is an open subset of the topological manifold 
$\partial {\cal T}_\epsilon(S)$ and hence it is a topological manifold, we can find a countable 
covering ${\cal W}$ of $\overline{A}_{\rm small}\setminus A$ so that each set 
$W_i\in {\cal W}$ is the intersection with $\overline{A}_{\rm small}\setminus A$ of an open set
$U_i$ of ${\cal T}_\epsilon(S)$ as described above.
Let $\{\zeta_i\mid i\}$ be a partition of 
unity subordinate to the covering of $\cup_iU_i\supset \overline{A}_{\rm small}\setminus A$
by the sets $U_i$.  
Using the above notations but writing $\psi_i$ for the function $\psi$ on $U_i\sim B_0$,  
for each $i$ the vector field 
$\zeta_i\psi_i V$ defines a local flow which moves $W_i\cap \{\zeta_i >0\}$ into $A$. 
Then the flow of the vector field $\sum_i \zeta_i \psi_iV$ is well defined, and the image of
$\overline{A}_{\rm small}\setminus A$ under its time one map is contained in $A$. This is what we wanted to show.
\end{proof}

\subsection{A topology on ${\cal T}_\epsilon(S)\cup {\cal X}(S)$} 

In this subsection we allow the surface $S$ to be disconnected. 
All components of $S$ may have punctures, but they have empty boundary unless
the component is an annulus. 
The construction of a topology on 
$\overline{\cal T}(S)={\cal T}_\epsilon(S)\cup {\cal X}(S)$ 
is similar to the construction of the topology on ${\cal X}(S)$.
We begin with having a short look at an annulus.

\begin{example}\label{annulus}
In the case $S$ is an annulus, then 
we have ${\cal T}(S)=\mathbb{R}$, ${\cal X}(S)=\{+,-\}$. If we equip 
$\overline{\cal T}(S)$ with the topology of the compactification of $\mathbb{R}$ 
which is homeomorphic to a compact interval and is obtained by
attaching two points $\pm \infty$, then this construction defines 
an ${\cal E\cal Z}$-structure for the infinite cyclic group of Dehn twists along 
the core curve of the annulus.  \hfill $\blacksquare$
\end{example}

If $S=\sqcup_i S_i$ then 
the Teichm\"uller space ${\cal T}(S)=\prod{\cal T}(S_i)$ of $S$ is the product of the 
Teichm\"uller spaces of the components of $S$, and we have
${\cal T}_\epsilon(S)=\prod {\cal T}_\epsilon(S_i)$. 
There exists a constant $\rho=\rho(S)>\epsilon_0$, a so-called 
\emph{Bers constant}, such that any marked hyperbolic surface $X\in {\cal T}(S)$
admits a pants decomposition by simple closed curves 
of $X$-length at most $\rho$ \cite{Bu92}.  
If $X\in {\cal T}_\epsilon(S)$, then
by possibly enlarging $\rho$, we may in fact assume that
$X$ admits a marking $\mu(X)$ consisting of simple closed 
curves of length at most $\rho$. We call such a marking
\emph{short} for $X$. Namely, it is a consequence of the collar lemma
\cite{Bu92} that for any surface $X\in {\cal T}_\epsilon(S_i)$ 
there exists a compact subset of $X$ whose diameter is bounded from above by a constant
only depending on the topological type of $S_i$ and $\epsilon$ which contains every 
simple closed curve on $S_i$. From this fact it is easy to construct from a pants 
deomposition whose components have length smaller than a Bers constant 
an explicit short marking. 

Using again the collar lemma \cite{Bu92}, 
the geometric intersection number between any 
two simple closed curves on $S$ of $X$-length at most $\rho$
is bounded from above by a universal constant. 
In particular, the marking $\mu(X)$ defines a subset of uniformly
bounded diameter in ${\cal C\cal G}(S)$
(see \cite{MM99} for more information). Moreover, as before, 
for every proper essential not necessarily connected subsurface $S_0$ of $S$, 
the subsurface
projections of the marking curves from $\mu(X)$ coarsely define a marking ${\rm pr}_{S_0}(\mu(X))$ of $S_0$.

The following definition is geared at overcoming some purely technical difficulties in the 
construction of an ${\cal E\cal Z}$-structure for ${\rm Mod}(S)$. 

\begin{definition}\label{nice}
A topology on $\overline{\cal T}(S)={\cal T}_\epsilon(S)\cup {\cal X}(S)$ is called \emph{nice} if it restricts to the 
given topologies on ${\cal T}_\epsilon(S)$  and on ${\cal X}(S)$, with ${\cal T}_\epsilon(S)\subset
\overline{\cal T}(S)$ open and dense, and if 
every point $\xi\in {\cal X}(S)$ has a neighborhood basis 
consisting of sets $U_\xi$ so that $U_\xi\cap \mathring{\cal T}_\epsilon(S)$ is 
open and contractible. 
\end{definition}

We are now ready to state a more precise version of the third part of Theorem \ref{main}.
In its formulation, we do not require $S$ to be connected.

\begin{theo}\label{mainprecise}
For a surface $S$ of finite type there exists a topology on $\overline{\cal T}(S)$ with the following 
properties.
\begin{enumerate}
\item The topology is nice.
\item Let $\xi=\sum_{i=1}^k a_i\xi_i\in {\cal X}(S)$ and for each $i\leq k$ let $S_i$ be the surface filled 
by $\xi_i$. 
Put $S_{k+1}=S\setminus \cup_{i=1}^k S_i$. 
A sequence $X^j\subset {\cal T}_\epsilon(S)$ converges to $\xi$ if and only if 
$({\rm pr}_{S_1}(\mu(X^j)),\dots,{\rm pr}_{S_{k+1}}(\mu(X^j)))\to \xi$ in 
${\cal Y}(\cup_{i=1}^{k+1}S_i)$. 
\item The pair $(\overline{\cal T}(S),{\cal X}(S))$ is an ${\cal E\cal Z}$-structure for ${\rm Mod}(S)$.
\end{enumerate}
\end{theo}

The remainder of this article is devoted to the proof of Theorem \ref{mainprecise}.

\begin{example}\label{farey2}
If $S$ is a once punctured torus or a four-holed sphere, then we saw in Example \ref{farey}
that ${\cal X}(S)$ has a natural identification with the Gromov boundary $\partial {\rm Mod}(S)$ 
of the mapping class group of $S$.
Recall that ${\cal X}(S)$ is a disjoint union of the set $\partial {\cal C\cal G}(S)$ 
of minimal filling geodesic laminations 
with the set of all labeled simple closed curves. 

Since ${\rm Mod}(S)$ is virtually free and hence 
a hyperbolic group which acts properly and cocompactly
on ${\cal T}_\epsilon(S)$, the space $\overline{\cal T}(S)$ has a natural topology 
which is inherited from the topology of the union of ${\rm Mod}(S)$ with its Gromov boundary.
The restrictions of this topology to the subsets ${\cal T}_\epsilon(S)$ and ${\cal X}(S)$  are
the given topologies. 

Any proper essential subsurface of $S$ either is a three-holed sphere (which does not play a role
in our discussion) or an annulus. Let $A\subset S$
be such an annulus. We claim that
the topology of $\overline{\cal T}(S)$ 
is consistent with the topology of the compactification $\mathbb{R}\cup \{\pm \infty\}$ 
of the Teichm\"uller space ${\cal T}(A)$ in the sense of property (2) in Theorem \ref{mainprecise}. 

Namely, let $c$ be the core curve of $A$ and let 
$c_+\in {\cal X}(S)$ be the curve $c$ equipped with a label. 
Denote by $\langle T_c\rangle$ the infinite cyclic group of Dehn twists about $c$ and
assume that $c_+$ corresponds to the limit point of the sequence $T_c^k$ as 
$k\to \infty$. Note that $\langle T_c\rangle$ is a quasi-convex subgroup of ${\rm Mod}(S)$.
Let $X\in {\cal T}_\epsilon(S)$
be an arbitrary point. 
With respect to the topology of ${\rm Mod}(S)\cup \partial {\rm Mod}(S)$ as the union 
of a hyperbolic group with its Gromov boundary, with ${\rm Mod}(S)$ identified with the orbit of $X$,
a sequence of points $X_i =\phi_i(X)\subset {\cal T}_\epsilon(S)$
for $\phi_i\in {\rm Mod}(S)$ 
 converges to $c_+\in {\cal X}(S)$ 
if the shortest distance projections 
of the elements $\phi_i$ into the 
quasi-convex infinite cyclic subgroup 
$\langle T_c\rangle$ 
converge to $c_+$.
Translated into properties of the subsurface projections of points in the Farey graph
as explained in Example \ref{oncepuncturedtorus}, 
this just means that the topology on $\overline{\cal T}(S)$ fulfills property (2) 
in Theorem \ref{mainprecise}. 

We can also check that the topology is nice. Namely, recall that
${\cal T}_\epsilon(S)$ can be identified with the complement in the hyperbolic
plane $\mathbb{H}^2$ of a ${\rm Mod}(S)$ invariant 
countable collection of horoballs whose closures are pairwise disjoint.
The horoballs are based at the rational points of $\partial \mathbb{ H}^2$ and are fixed by an 
infinite cyclic subgroup of ${\rm Mod}(S)$ 
of parabolic isometries. 

Let $H\subset \mathbb{H}^2$ be such a horoball, with boundary $\partial H$, fixed by the
parabolic group $G\subset {\rm Mod}(S)$. This is a convex subset of $\mathbb{H}^2$. 
Let $\eta:\mathbb{R}\to \partial H$ be a parameterization of $\partial H$
by arc length.
The geodesics in $\mathbb{H}^2$ which are asymptotic to the 
fixed point of $G$ in $\partial \mathbb{H}^2$ 
foliate $\mathbb{H}^2$ and determine a shortest distance projection
$P:\mathbb{H}^2\setminus \mathring H\to \partial H$. For each integer $m\geq 1$ the
set $U_m=P^{-1}(\eta(m,\infty))\setminus \eta(m,\infty)$ is contractible and 
intersects $\mathring {\cal T}_\epsilon(S)$ in a contractible open set whose 
small closure is a neighborhood of 
the labeled point $\eta(\infty)=c_+$. These neighborhoods define a countable neighborhood
basis of $c_+$ which are small closures of contractible open subsets of $\mathring {\cal T}_\epsilon(S)$.

Alternatively, let $V_m \supset U_m$ be the union of all 
leaves of the foliation which pass through $\eta(m,\infty)$. Clearly $V_m$ is contractible. 
The small closures of the images of the sets $V_m$ under 
the homeomorphism
${\cal T}(S)\to \mathring {\cal T}_\epsilon(S)$ then define another neighborhood basis 
of $c_+$ in the above topology of $\overline{\cal T}(S)$ consisting of small closures of 
contractible open subsets of $\mathring{\cal T}_\epsilon(S)$.

As neighborhood bases of minimal filling laminations will be discussed in detail in a more general
context, we omit the discussion here. \hfill $\blacksquare$
\end{example}

\begin{proposition}\label{topology2}
There exists a Hausdorff topology ${\cal O}_0$ 
on $\overline{\cal T}(S)={\cal T}_\epsilon(S)\cup {\cal X}(S)$ with the property that a 
set $A\subset \overline{\cal T}(S)$ is closed for
${\cal O}_0$ if and only if the following 
holds true.
\begin{enumerate}
\item $A\cap {\cal T}_\epsilon(S)$ is closed in ${\cal T}_\epsilon(S)$, 
and $A\cap {\cal X}(S)$ is closed in ${\cal X}(S)$.
\item If $X_j\subset A\cap {\cal T}_\epsilon(S)$ is a sequence which converges to 
$\xi\in {\cal X}(S)$ in the sense of (2) of 
Theorem \ref{mainprecise}, then $\xi\in A$.
\end{enumerate}
\end{proposition}
\begin{proof}
The proof is analogous to the proof of Lemma \ref{topology}. 
By the Hausdorff property of ${\cal T}_\epsilon(S)$ and Lemma \ref{topology}, note first that 
any limit of a convergent sequence $X_j\subset {\cal T}_\epsilon(S)$ is unique.

To show that the notion of convergence defines a topology on $\overline{\cal T}(S)$ 
with the property that ${\cal T}_\epsilon(S)\subset \overline{\cal T}(S)$ is open and 
a set $A\subset \overline{\cal T}(S)$ is closed if $A\cap {\cal T}_\epsilon(S)$ and 
$A\cap {\cal X}(S)$ are closed and if $A$ contains the limit of any sequence $X_j\subset 
{\cal T}(S)$ which converges to a point in ${\cal X}(S)$, it suffices to verify that
the empty set and the entire space are closed, and the same holds true for finite unions
and arbitrary intersections of closed sets. 
The verification that this is satisfied is 
identical to the argument used in the proof of Lemma \ref{topology}.

Finally we have to show that the topology thus constructed is Hausdorff. 
Since ${\cal X}(S)\subset \overline{\cal T}(S)$ is a closed Hausdorff space 
and 
${\cal T}_\epsilon(S)\subset \overline{\cal T}(S)$ is an open Hausdorff subspace 
of $\overline{\cal T}(S)$,
all we need to show is that two points $\xi\not=\eta\in {\cal X}(S)$ have disjoint neighborhoods.
Now $\xi,\eta$ have disjoint neighborhoods in ${\cal X}(S)$ and hence 
since $\overline{\cal T}(S)$ is separable, 
it suffices to show
that the limit of any sequence $X_i\subset {\cal T}_\epsilon(S)$ converging to a point in 
${\cal X}(S)$ is unique. But this was established in the beginning of this proof.
\end{proof}

Assume from now on that
$\overline{\cal T}(S)$ is equipped with the topology defined in Proposition \ref{topology2}. 
We have to verify that this topology satisfies the properties stated in Theorem \ref{mainprecise}. 

\begin{example}\label{uniofann}
If $S=\sqcup_{i=1}^k S_i$ where for each $i$, $S_i$ is an annulus, then ${\cal T}(S)=\mathbb{R}^k$, viewed
as a ${\rm CAT}(0)$ cube complex with a proper isometric action of $\mathbb{Z}^k$. 
The compactification of $\mathbb{R}^k$ by attaching the visual boundary $S^{k-1}$ defines 
an ${\cal E\cal Z}$-structure for $\mathbb{Z}^k$ with the properties stated in Theorem \ref{mainprecise}. 

Now ${\cal X}(S)$ is the join of $k$ sets consisting of two points each and hence ${\cal X}(S)$ is a sphere
of dimension $k-1$. By Proposition \ref{topology2}, the topology on ${\cal T}_\epsilon(S)\cup {\cal X}(S)$ is 
determined by the properties stated in Theorem \ref{mainprecise} and hence
$\overline{\cal T}(S)$ equals the compactification of $\mathbb{R}^k$ by adding the visual sphere.
In particular, $\overline {\cal T}(S)$ has all properties stated in Theorem \ref{mainprecise}. 
\hfill $\blacksquare$
\end{example}

\begin{proposition}\label{topsmallbd}
The topological space $(\overline{\cal T}(S),{\cal O}_0)$ has
the following properties.  
\begin{enumerate}
\item $\overline{\cal T}(S)$ is 
compact and separable.
\item The mapping class group acts on $\overline{\cal T}(S)$ as
  a group of transformations.
\end{enumerate}
\end{proposition}
\begin{proof} $\overline{\cal T}(S)$ is clearly separable since this holds true
for ${\cal X}(S)$ and ${\cal T}_\epsilon(S)$. 
By Proposition \ref{topology2}, it is a Hausdorff space.

To show that $\overline{\cal T}(S)$ is compact, note that since ${\cal X}(S)$ is compact and 
${\cal T}_\epsilon(S)$ is a Lindel\"of space, the space $\overline{\cal T}(S)$ is 
Lindel\"of. Since $\overline{\cal T}(S)$ also is Hausdorff, 
it suffices to show that $\overline{\cal T}(S)$ is 
sequentially compact, and this follows if we can show that
any sequence $X_i\subset {\cal T}_\epsilon(S)$ 
has a convergent subsequence in $\overline{\cal T}(S)$. 

If the sequence has a bounded subsequence in ${\cal T}_\epsilon(S)$ with respect
to a fixed basepoint $X\in {\cal T}_\epsilon(S)$, then as 
${\cal T}_\epsilon(S)$ is proper, we can extract a converging subsequence. Thus 
it suffices to show the following

\smallskip
\noindent
{\bf Claim:} Any unbounded sequence $X^i\subset {\cal T}_\epsilon(S)$ admits a subsequence
which converges in $\overline{\cal T}(S)$ to a point $\xi\in {\cal X}(S)$.

\smallskip \noindent
{\em Proof of the claim:} The proof of the claim is essentially identical
with the proof of Proposition \ref{compactbd}.

Since the space of geodesic laminations on $S$ equipped with the Hausdorff topology is
compact, by extracting a subsequence we may assume that the sets of 
all simple closed curves contained in the marking $\mu(X^i)$ 
converge in the Hausdorff topology to a finite union $\cup_{j=1}^\ell \beta_j$ 
of (not necessarily minimal) geodesic laminations. Note that
as some of the curves in $\mu(X^i)$ may intersect, these laminations are not 
necessarily disjoint, that is, $\cup_{j=1}^\ell \beta_j$ may not be a lamination in its
own right.  
However, since the number of components of $\mu(X^i)$ 
is uniformly bounded, the same holds true for the number of limit laminations.

Let $\zeta_1,\dots,\zeta_s$ be the collection of all minimal components of the laminations $\beta_u$ 
which are distinct
from simple closed curves. The number of such components is uniformly bounded.
Each of the laminations $\zeta_j$ fills a subsurface 
$S_j$ of $S$ which is different from an annulus or a pair of pants. 
Thus $\zeta_j$ is a point in the Gromov boundary of the curve 
graph ${\cal C\cal G}(S_j)$ of $S_j$. 

Now for $j\leq s$, a sequence $c^i_j$ of simple closed curves on the surface $S_j$ converges to 
$\zeta_j$ in ${\cal C\cal  G}(S_j)\cup \partial {\cal C\cal G}(S_j)$ if and only if
their geodesic representatives for some fixed hyperbolic metric on $S_j$ 
converge to $\zeta_j$ in the coarse Hausdorff topology. 
As the diameter of the subsurface projection of 
$\mu(X^i)$ to $S_j$ is bounded independent of $i$, hyperbolicity 
of ${\cal C\cal G}(S_j)$ implies that 
the subsurface projection to $S_j$  
of any of the curves in $\mu(X^i)$ which
intersects $S_j$ converges in the coarse Hausdorff topology to $\zeta_j$. 
 As a consequence,
none of the limits in the Hausdorff topology of any sequence of components of 
$\mu(X^i)$ can intersect $\zeta_j$. 

By a similar argument, if $\zeta_j$ is a closed curve component, then we can consider
the subsurface projections of a component of  $\mu(X^i)$ to an annulus $A(\zeta_j)$ with core
curve $\zeta_j$. Up to passing to a further subsequence, we may assume that these projections
are either bounded along the sequence, or converge to one of the two 
boundary components of the curve graph of $A(\zeta_j)$. In the first case call $\zeta_j$ 
\emph{unlabeled}. In the second case, label $\zeta_j$ with the corresponding point in the 
Gromov boundary of the curve graph of $A(\zeta_j)$ and note by the reasoning used in 
the previous paragraph, no labeled 
simple closed curve component $\zeta_j$ can be intersected
by another component $\zeta_\ell$.

By reordering, let $\zeta_1,\dots,\zeta_k$ be the components of the limit laminations $\beta_u$ 
which either are distinct
from simple closed curves or which are labeled simple closed curves. 
We claim that $k\geq 1$, that is, that there is at least one lamination with this property.
Namely, if $c$ is an unlabeled simple closed curve, represented by a closed geodesic
for the base surface $X$, and if with respect to the Hausdorff topology on compact subsets
of $X$ a limit of the sequence $\mu(X^i)$ contains $c$ as an unlabeled component,
then no component 
of a limit of the sequence $\mu(X^i)$ in the Hausdorff topology 
can spiral about $c$ and hence $c$ is a component of $\mu(X^i)$ for all but 
finitely many $i$. If $k=0$ then 
this holds true for any limit point of the sequence $\mu(X^i)$ in the Hausdorff topology. 
But $\mu(X^i)$ is a marking
of $S$ for all $i$ and hence decomposes $S$ into disks and once punctured disks and
consequently the sequence $\mu(X^i)$ is bounded. But this contradicts the assumption that
the sequence $X^i\subset {\cal T}_\epsilon(S)$ is an unbounded sequence.

By what we showed so far,
$\hat \zeta=\cup_{j=1}^k\zeta_k$ is a geodesic lamination. Furthermore,
if $S_j$ is the subsurface of $S$ filled by $\zeta_j$, then 
$d_{{\cal C\cal G}(S_j)}({\rm pr}_{S_j}(\mu(X^i)),x_j)\to \infty$ where as before, 
$x_j\in {\cal C\cal G}(S_j)$ is a fixed basepoint for ${\cal C\cal G}(S_j)$.

If $\hat \zeta$ is minimal and fills $S$ then $X_i\to \hat \zeta\in \overline{\cal T}(S)$ and we are done.
Otherwise we argue as in the proof of Proposition \ref{compactbd}.
Namely, put $S_{k+1}=S\setminus \cup_iS_i$. Then the projections
${\rm pr}_{S_{k+1}}(\mu(X_i))$ define a bounded subset of the curve graph
${\cal C\cal G}(S_{k+1})$, and the same holds true for projections
${\rm pr}_V(\mu(X_i))$ for any proper subsurface $V$ of $S_{k+1}$.
But then the construction in the proof of Proposition \ref{compactbd} 
produces a subsequence of the sequence $\mu(X^i)$ 
with the property that 
$({\rm pr}_{S_1}(\mu(X^i)), \dots ,{\rm pr}_{S_{k+1}}(\mu(X^i)))$ 
converges in
${\cal Y}(\cup_{i=1}^{k+1}S_i)$ to $\xi=\sum_i a_i\zeta_i$ for $a_i\geq 0$, $\sum_ia_i=1$. 
By the requirement (2), we then have
$X_i\to \xi\in \overline{\cal T}(S)$ which completes the proof of the claim. 
\hfill $\blacksquare$

\smallskip

To summarize, we showed that 
$\overline{\cal T}(S)$ is sequentially compact Hausdorff Lindel\"of space and 
hence it is compact.

We are left with showing that ${\rm Mod}(S)$ acts on 
$\overline{\cal T}(S)$ as a group of transformations.
However, as ${\rm Mod}(S)$ acts on ${\cal T}_\epsilon(S)$ and on 
${\cal X}(S)$ as a group of transformations, and it maps subsurfaces of $S$ to subsurfaces of the 
same topological type, moreover 
the definition of convergence which determines
the topology ${\cal O}_0$ is natural with respect to the action of 
${\rm Mod}(S)$ on subsurfaces and subsurface projections, 
this is indeed the case.
The proposition is proven.
\end{proof}

\begin{theo}\label{smallbd}
${\cal X}(S)$ is a small boundary for ${\rm Mod}(S)$. A pseudo-Anosov mapping class acts
on ${\cal X}(S)$ with north-south dynamics. In particular, the action of ${\rm Mod}(S)$
on ${\cal X}(S)$  is strongly proximal. 
\end{theo}
\begin{proof} We showed so far that ${\cal X}(S)$ defines a boundary of ${\cal T}_\epsilon(S)$ and hence
of ${\rm Mod}(S)$ since ${\rm Mod}(S)$ acts properly and cocompactly on ${\cal T}_\epsilon(S)$.  
Furthermore, by Example \ref{limit}, a pseudo-Anosov element acts on ${\cal X}(S)$ 
with north-south dynamics and hence the action of ${\rm Mod}(S)$ on 
${\cal X}(S)$ is strongly proximal. 

We are left with showing that the right action of ${\rm Mod}(S)$ induces the identity.
However, this action just consists of a change of basepoint. As a sequence of points 
of uniformly bounded distance from a convergent sequence converges to the same point,
this yields the statement of the theorem.
\end{proof}

%
%

\section{Neighborhood bases}\label{sec:neighborhoodbasis}

The main goal of this section is to construct for a point in 
${\cal X}(S)\subset \overline{\cal T}(S)$
an explicit neighborhood basis
in $\overline{\cal T}(S)$ 
consisting of small closures of open contractible subsets of $\mathring{\cal T}_\epsilon(S)$.
Here by a contractible subset of ${\cal T}_\epsilon(S)$ we mean a subset $V$ which 
is a contractible space with respect to
the subspace topology.  The construction is the most involved part of the proof
of Theorem \ref{main}. It is 
 carried out in three steps, each of which is contained
 in a separate subsection.   
 

\subsection{A neighborhood basis for minimal filling laminations}\label{subsec:filling}
In this subsection we prove the following result for a connected finite type surface $S$.
 
\begin{proposition}\label{curvegraphbdneigh} 
Every point $\xi\in \partial{\cal C\cal G}(S)\subset {\cal X}(S)$ has a countable 
neighborhood basis in 
$\overline{\cal T}(S)$ consisting of sets whose intersections with 
${\cal T}_\epsilon(S)$ are small closures of contractible open subsets of 
$\mathring {\cal T}_\epsilon(S)$ and hence are contractible.
\end{proposition}

The proof of Proposition \ref{curvegraphbdneigh}
introduces the ideas used in the general case, but it is technically
easier. We begin with summarizing some tools used to relate Teichm\"uller theory
to combinatorial geometry. 

By \cite{MM99}, for any surface $V$ of finite type there is 
a number $p>0$ only depending on the complexity of $V$ such that 
the image under the map $\Upsilon$ of a Teichm\"uller geodesic $\gamma:\mathbb{R}\to 
{\cal T}(S)$ is an \emph{unparameterized $p$-quasi-geodesic} in 
${\cal C\cal G}(V)$. This means the following. There is an increasing
homeomorphism $\sigma:(a,b)\subset \mathbb{R}\to \mathbb{R}$ such that
the map $\Upsilon \circ \gamma\circ \sigma:(a,b)\to {\cal C\cal G}(S)$ is 
a $p$-quasi-geodesic. This quasi-geodesic may be bounded, one-sided infinite or
two-sided infinite. A sufficient but not necessary condition for being 
one-sided infinite in the positive direction is that the geodesic recurs in the positive direction
to the thick part ${\cal T}_\epsilon(S)$ for arbitrarily large times. 
Since the directions of Teichm\"uller geodesic rays with this property are dense
in the cotangent bundle of Teichm\"uller space and the map $\Upsilon$ is coarsely
Lipschitz with respect to the Teichm\"uller metric on ${\cal T}(S)$ and the metric 
on ${\cal C\cal G}(S)$, 
up to increasing $p$, 
any geodesic segment $\alpha:[0,n]\to {\cal C\cal G}(S)$ can be extended
to a $p$-quasi-geodesic ray $\alpha:[0,\infty)\to {\cal C\cal G}(S)$. Note that this
is \emph{not} true for geodesics in ${\cal C\cal G}(S)$ \cite{BM15}.

Let $\xi\in \partial {\cal C\cal G}(S)$. Choose a basepoint 
$X_0\in {\cal T}_\epsilon(S)$ and let $c$ be a pants curve from $\mu(X_0)$.
For $j,\ell\geq 0$ define
\[W(\xi,j,\ell)\subset {\cal T}_\epsilon(S)\]
to be the set of all hyperbolic metrics $X\in {\cal T}_\epsilon(S)$ with the
following properties.
\begin{enumerate}
\item 
$d_{{\cal C\cal G}(S)}(\mu(X), c)\geq j$. 
\item 
A geodesic in ${\cal C\cal G}(S)$ connecting $c$ 
to $\mu(X)$ 
can be extended to a $p$-quasi-geodesic in ${\cal C\cal G}(S)$ whose
endpoint is contained in the ball of radius $e^{-\ell}$ about $\xi$ 
in $\partial {\cal C\cal G}(S)$, where the metric on 
$\partial {\cal C\cal G}(S)$ is the Gromov distance $d_{c}$ constructed from  
the basepoint $c$.  
\end{enumerate}

\begin{lemma}\label{closure}
For each $\xi\in \partial {\cal C\cal G}(S)$ 
the closures of the sets $W(\xi,j,\ell)$ in $\overline{\cal T}(S)$ define a neighborhood 
basis of $\xi$ in $\overline{\cal T}(S)$.
\end{lemma}
\begin{proof}
We show first that for each $\xi,j,\ell$ the closure of $W(\xi,j,\ell)$ in $\overline{\cal T}(S)$ is a neighborhood
of $\xi$. 
Since ${\cal T}_\epsilon(S)$ is dense in
  $\overline{\cal T}(S)$ and by Proposition \ref{topsmallbd}, 
  $\overline {\cal T}(S)$ is a compact separable Hausdorff space, 
it suffices to show the following.
Let $(X_m)\subset {\cal T}_\epsilon(S)$ be a sequence converging in
$\overline{\cal T}(S)$ to $\xi$; then $X_m\in W(\xi,j,\ell)$ for all
sufficiently large $m$.

Now as $\xi\in \partial {\cal C\cal G}(S)$, by the definition of the topology of
$\overline{\cal T}(S)$, we know that $\mu(X_m)\to \xi$ in 
${\cal C\cal G}(S)\cup \partial {\cal C\cal G}(S)$. Together with hyperbolicity of 
${\cal C\cal G}(S)$ and the definition of the topology on the union of a 
hyperbolic geodesic metric space with its Gromov boundary, 
this immediately implies that
$X_m\in W(\xi,j,\ell)$ for all sufficiently large $m$.

A similar argument also shows that the sets $W(\xi,j,\ell)$ define a neighborhood basis of 
$\xi$. Namely, note that the sets $W(\xi,j,\ell)$ are 
  \emph{nested}:
  If $j^\prime>j$, then $W(\xi,j^\prime,\ell)\subset W(\xi,j,\ell)$, and if 
$\ell^\prime >\ell$ then $W(\xi,j,\ell^\prime)\subset W(\xi,j,\ell)$.   
Thus since ${\cal X}(S)$ is a compact Hausdorff space,
to show that the closures $\overline{W(\xi,j,\ell)}$ in $\overline{\cal T}(S)$ 
of the 
sets $W(\xi,j,\ell)$ define a neighborhood basis of
$\xi$ in $\overline{\cal T}(S)$, it suffices to show that
$\cap_{j>0,\ell>0}\overline{W(\xi,j,\ell)}=\{\xi\}$.

To see that this is indeed the case note first that
$\xi\in \overline{W(\xi,j,\ell)}$ for all $j,\ell$ 
and hence as these sets are compact, the point 
$\xi$ also is contained in
the intersection of these sets. Furthermore, the following holds true.

For each $j$ let $X_j\in W(\xi,j,j)$; 
then the distance of $\mu(X_j)$ to the base curve $c$ 
in ${\cal C\cal G}(S)$ tends to infinity with $j$.
This implies that the sequence $X_j$ can not have a convergent subsequence in
${\cal T}_\epsilon(S)$. Thus by compactness of $\overline{\cal T}(S)$, we observe that
$\cap_{j} W(\xi,j,j)\subset {\cal X}(S)$. But for each $j$, the set $\mu(X_j)$ is contained 
in a $p$-quasi-geodesic connecting the basepoint to a point in the $e^{-j}$-ball about $\xi$
in $(\partial {\cal C\cal G}(S),d_c)$ and therefore $X_j\to \xi$ in 
$\overline{\cal T}(S)$. Since $X_j\in W(\xi,j,j)$ was arbitrary, this implies that
$\cap_j\overline{W(\xi,j,j)}=\{\xi\}$ as claimed. 
\end{proof}

The following consequence of Lemma \ref{closure} will be used in 
Section \ref{Sec:metrizability}. 

\begin{corollary}\label{cor}
There are countably many open sets $U_i\subset \overline{\cal T}(S)$ which contain
a neighborhood basis for every $\xi\in \partial {\cal C\cal G}(S)$.
\end{corollary}
\begin{proof}  Since $\partial {\cal C\cal G}(S)$ is separable, we can choose a countable
dense subset $\{\xi_i\mid i\}\subset \partial{\cal C\cal G}(S)$. 
Let $d_c$ be the Gromov metric on $\partial {\cal C\cal G}(S)$ 
with respect to the basepoint $c$. Then the open $d_c$-balls $B_{i,j}$ of radius $e^{-j}$ about the points 
$\xi_i$ ($i,j>0$) 
define a basis of the topology of $\partial {\cal C\cal G}(S)$. 

By Lemma \ref{closure} and the definitions, the interiors of the countably many
closed subsets $\overline{W(\xi_i,j,\ell)}$ 
of $\overline{\cal T}(S)$ have the properties stated in 
the corollary. 
\end{proof}

Any minimal filling geodesic lamination $\xi$ decomposes $S$ into a union of ideal
polygons. 
Each of these polygons which is not an ideal triangle 
can be subdivided by adding isolated leaves which
connect two non-adjacent cusps of the polygon.  
The various ways to subdivide these polygons 
determine a finite collection 
$\xi^0,\dots,\xi^k$ of distinct geodesic laminations
which contain $\xi$ as a sublamination. 
Assume that $\xi^0=\xi$, that is, $\xi^0$ is the unique lamination among
the laminations $\xi^j$ which does not contain any isolated leaf. 

Let $d_{H}$ be the Hausdorff metric
 on the space of compact subsets of a fixed hyperbolic surface $X\in {\cal T}_\epsilon(S)$.
Denote as before by ${\rm Min}_\cup(S)$ the space of geodesic laminations on $X$ which are unions of 
disjoint minimal components. Equivalently, the only isolated leaves of a geodesic lamination
in ${\rm Min}_\cup(S)$ are simple closed curves. 
 As before, let 
${\rm supp}:{\cal X}(S)\to {\rm Min}_\cup(S)$
be the map which associates to a point $\zeta=\sum_ia_i\zeta_i$ $(a_i>0)$ the support 
${\rm supp}(\zeta)=\cup_i \zeta_i$. 
We have 

\begin{lemma}\label{minimal}
For $i>0$ let 
\[U_i=\cup_j \{\beta\in {\rm Min}_\cup(S)\mid d_H(\beta, \xi^j)\leq 1/i\}\]
and write
$V_i=\{\zeta\in {\cal X}(S)\mid {\rm supp}( \zeta) \in U_i\}$. 
Then the sets $V_i$ form a neighborhood basis of $\xi\in \partial {\cal C\cal G}(S)\subset {\cal X}(S)$ in ${\cal X}(S)$.
\end{lemma}
\begin{proof} Clearly $\xi=\xi^0\in V_i$ for all $i$. 
%
We first show that for each $i$ the set $V_i$ is a neighborhood of $\xi$. 
For this it suffices to show that for every 
sequence $\zeta_\ell\subset {\cal X}(S)$ converging to $\xi$ and any $i$, we have
$\zeta_\ell\in V_i$ for all but finitely many $\ell$. 

By the first requirement for convergence in the definition of the topology on
${\cal X}(S)$, we know that ${\rm supp}(\zeta_\ell)$ 
converges in the \emph{coarse} Hausdorff topology to $\xi_0={\rm supp}(\xi)$.
By compactness of the space of geodesic laminations of $S$ with respect to the 
Hausdorff topology,  by passing to a subsequence 
we may assume that the sequence ${\rm supp}(\zeta_\ell)$ converges in the Hausdorff
topology to a geodesic lamination $\zeta$. Then $\zeta$ contains
$\xi^0$ as a sublamination and hence $\zeta=\xi^s$ for some $s\leq k$.
By definition, this implies that ${\rm supp}(\zeta_\ell)\in U_i$ and hence 
$\zeta_\ell\in V_i$ for all sufficiently
large $\ell$ as predicted. Thus indeed, each of the sets
$V_i$ is a neighborhood of $\xi$.

To show that the sets $V_i$ form a neighborhood basis for $\xi$, note that
$V_{i+1}\subset V_i$ and hence it suffices to show that
$\cap_iV_i=\{\xi\}$. However, this is immediate from the definitions and the 
fact that the preimage of ${\rm supp}(\xi)$ under the support map 
${\rm supp}$ consists of the single point $\xi$.
\end{proof}



A \emph{measured geodesic lamination} on the surface $S$ 
is a geodesic lamination together with a transverse invariant measure.
The space ${\cal M\cal L}$ of measured geodesic laminations is equipped with the 
weak$^*$ topology. The quotient of ${\cal M\cal L}$ under the natural action of 
$(0,\infty)$ by scaling is the space ${\cal P\cal M\cal L}$ of 
\emph{projective measured geodesic laminations}. This space is homeomorphic
to the sphere $S^{6g-7+2m}$.
To put Lemma \ref{minimal} into proper context and for later use, we relate
the subset $\partial {\cal C\cal G}(S)\subset {\cal X}(S)\subset \overline{\cal T}(S) $ 
to the space ${\cal P\cal M\cal L}$.

To this end we use a more geometric view on ${\cal P\cal M\cal L}$.
Fix again a point $X\in {\cal T}_\epsilon(S)$. The 
cotangent space $T_X^*{\cal T}(S)$ of Teichm\"uller space
at $X$ can be identified with the space of measured geodesic 
laminations on $S$. Or, equivalently, by the Hubbard Masur theorem (the main theorem of \cite{HM79}, expressed
in the language of measured foliations),
every measured geodesic lamination $\nu$ on $S$ is the \emph{vertical measured geodesic
lamination} of a unique marked quadratic differential $q(\nu)$ for the complex structure on $S$ defined 
by $X$. With this identification, we can associate to $\nu\in {\cal M\cal L}$ the 
point $\gamma_\nu(1)$ where 
$\gamma_\nu:[0,\infty)\to {\cal T}(S)$ is the Teichm\"uller geodesic  starting at $X$ whose initial 
(co)-velocity equals $q(\nu)$. 
This construction defines the \emph{Teichm\"uller exponential map} 
$\exp_X:{\cal M\cal L}\cup \{0\}\to {\cal T}(S)$ at $X$ which is a homeomorphism.

Every nontrivial quadratic differential on $X$ defines a singular euclidean metric on $X$.
Associating to $\nu$ the area ${\rm area}(q(\nu))$ of the metric defined by $q(\nu)$ defines a 
positive $\mathbb{R}_+$-homogeneous function  
on ${\cal M\cal L}$ depending on $X$. Associating to 
$[\nu]\in {\cal P\cal M\cal L}$ the unique measured lamination
$\rho([\nu])$ with ${\rm area} (q(\rho[\nu]))=1$  then 
defines a section $\rho:{\cal P\cal M\cal L}\to {\cal M\cal L}$. In this way 
we can identify ${\cal P\cal M\cal L}$ with the sphere of unit directions for the 
Teichm\"uller metric at $X$.

The \emph{support} ${\rm supp}(\nu)$ 
of a measured geodesic lamination $\nu$ is a point in the space 
${\rm Min}_\cup(S)$.
Each of its components 
is equipped with a transverse invariant measure and hence
it is a measured geodesic
lamination in its own right.

Let as before $p>1$ be a control constant with the following properties. 
\begin{itemize}
    \item The image under the map $\Upsilon$ of any Teichm\"uller geodesic 
    is an unparameterized $p$-quasi-geodesic in ${\cal C\cal G}(S)$ (see \cite{MM99}). 
    \item Every geodesic segment in ${\cal C\cal G}(S)$ can be extended to a 
    $p$-quasi-geodesic ray.
\end{itemize}

Consider a minimal filling geodesic lamination $\xi\in \partial {\cal C\cal G}(S)$. 
Let $P(\xi)\subset {\cal P\cal M\cal L}$ be the set of 
all projective measured geodesic laminations which are
supported in $\xi$. This is a non-empty 
compact polytope of dimension $\leq 3g-3+m$ whose extreme
points are the ergodic projective transverse measures supported in $\xi$. 
In particular, $P(\xi)$ is compact and contractible. 
Since $P(\xi)$ is contractible
and since ${\cal P\cal M\cal L}$ is homeomorphic to a sphere of dimension $6g-7+2m$,
we can find a descending 
chain $V_1\supset V_2\supset \cdots$ of closed contractible neighborhoods
of $P(\xi)$, each of which is homeomorphic to a closed ball,  
such that $V_{j+1}\subset \mathring V_j$ and 
$\cap_jV_j=P(\xi)$. 

In the sequel we use the terminology 
\emph{small closure $\bar A_{\rm small}$ in $\overline{\cal T}(S)$} 
of a set $A\subset {\cal T}_\epsilon(S)$ to denote the union of the small closure
of $A$ in ${\cal T}_\epsilon(S)$ with the intersection with ${\cal X}(S)$ of the closure 
$\bar A$ of $A$ in $\overline {\cal T}(S)$. Thus for any set $A\subset {\cal T}_\epsilon(S)$, 
$\bar A_{\rm small}\cap {\cal X}(S)$ is closed, but $\bar A_{\rm small}\cap {\cal T}_\epsilon(S)$
may be open.

\begin{lemma}\label{descend}
Let $V_1\supset V_2\supset \cdots $ be
 a descending chain 
of closed contractible neighborhoods 
of $P(\xi)$ in ${\cal P\cal M\cal L}$, each of which is homeomorphic to a closed ball,
with $\cap_iV_i=P(\xi)$.
Let $\Lambda_\epsilon:{\cal T}(S)\to \mathring {\cal T}_\epsilon(S)$
be the homeomorphism from Proposition \ref{homeomorphism} and let 
$\exp_X:{\cal M\cal L}(S)\to {\cal T}(S)$ be the Teichm\"uller exponential map at
$X$. Then for each $j>0$, the small closure $\overline{Z(i,j)}_{\rm small}$ 
in $\overline{\cal T}(S)$ of the open set 
\[Z(i,j)=\Lambda_\epsilon\{\exp_X(t \rho[\nu])\mid t> j, 
[\nu]\in \mathring V_i\}\] is a neighborhood of $\xi$, and neighborhoods
of this form define a neighborhood basis of $\xi$.
\end{lemma}
\begin{proof}
We divide the proof of the lemma into two claims.

\smallskip
\noindent
{\bf Claim 1:} For all $i,j$, the small closure $\overline{Z(i,j)}_{\rm small}$ of $Z(i,j)$ in 
$\overline{\cal T}(S)$ is a neighborhood of $\xi$.

\smallskip \noindent
\emph{Proof of Claim 1:} 
Let ${\cal F\cal M\cal L}\subset {\cal P\cal M\cal L}$ be the subset of all projective
measured geodesic laminations whose support is a minimal geodesic lamination which 
fills up $S$. By Lemma 3.2 of \cite{H09}, the support map 
$F:{\cal F\cal M\cal L}\to \partial {\cal C}(S)$ which associates to 
a point in ${\cal F\cal M\cal L}$ its support is continuous and closed.
Thus the image $F({\cal F\cal M\cal L}\setminus \mathring V_i)$ is a closed subset of 
$\partial {\cal C\cal G}(S)$ which does not contain $\xi$. As a consequence, if we consider
the Gromov metric on $\partial {\cal C\cal G}(S)$ based at 
$\Upsilon(X)$, then 
there exists a number $T(i)>0$ so that the ball of radius 
$e^{-T(i)}$ about $\xi$ with respect to this metric 
is disjoint from 
$F({\cal F\cal M\cal L}\setminus \mathring V_i)$.

By the choice of the control constant $p>1$ and hyperbolicity, 
there exists a number $\tau(i)>T(i)$ with the following property.
Let $[\nu]\in {\cal F\cal M\cal L}\setminus \mathring V_i$;
then the endpoint of 
a $p$-quasi-geodesic ray in ${\cal C\cal G}(S)$ which starts at the basepoint 
$\Upsilon(X)$ and which passes through
a point on the $p$-quasi-geodesic 
$\Upsilon \{\exp_X(t \rho [\nu])\mid t\geq 0\}$ of distance at least $\tau(i)$ to $\Upsilon(X)$ 
is not contained 
in the ball of radius $e^{-T(i)/2}$ about $\xi$.

Since the homeomorphism  $\Lambda_\epsilon:{\cal T}(S)\to \mathring{\cal T}_\epsilon(S)$ is coarsely
$\Upsilon$-invariant, the map 
$\exp_X:{\cal M\cal L}\cup \{0\}\to {\cal T}(S)$ 
is a homeomorphism and ${\cal F\cal M\cal L}\subset {\cal P\cal M\cal L}$ is dense, 
it follows that for all $i,j$ 
there exists $\ell$ so that the set
$W(\xi,\ell,\ell)\subset {\cal T}_\epsilon(S)$ 
constructed in Lemma \ref{closure} 
is contained in the closure of the set 
$\Lambda_\epsilon(\{\exp_X(t\rho[\nu])\mid t\geq  j,[\nu]\in \mathring V_i\})$.
 Thus by Lemma \ref{closure}, 
$\overline{Z(i,j)}_{\rm small}$ is a neighborhood of $\xi$ in $\overline{\cal T}(S)$. 
\hfill $\blacksquare$
\smallskip 

The proof of the lemma is completed once we established the following. In its 
formulation, $\overline{Z(i,j)}$ is the closure of $Z(i,j)$ in $\overline{\cal T}(S)$. We do not have
to take the small closure here. 

\smallskip
\noindent
{\bf Claim 2:} Let $W$ be a neighborhood of $\xi$ in $\overline{\cal T}(S)$; then 
there exists some $i,j$ so that $\overline{Z(i,j)}\subset W$.
\smallskip

\noindent \emph{Proof of Claim 2:}
By Claim 1, each of the sets $\overline{Z(i,j)}$ is a neighborhood of 
$\xi$ and hence contains $\xi$. 
Furthermore, these neighborhoods are nested: If $i_1\leq i_2$ and 
$j_1\leq j_2$ then $\overline{Z(i_1,j_1)}\supset \overline{Z(i_2,j_2)}$. 
Thus since the sets $\overline{Z(i,j)}$ are moreover closed and hence compact, 
it suffices to show that $\cap_{i,j} \overline{Z(i,j)}=\{\xi\}$.

Since the Teichm\"uller exponential map $\exp_X$ at $X$ is a homeomorphism,
we clearly
have $\cap_{i,j}\overline{Z(i,j)}\subset {\cal X}(S)$.  
On the other hand,  the
map $\Upsilon:{\cal T}(S)\to {\cal C\cal G}(S)$ is coarsely Lipschitz, 
and for $\nu\in P(\xi)$, the $p$-quasigeodesic $t\to \Upsilon(\exp(t \rho[\nu]))$ has infinite
diameter. This implies that for any $k>0$ there are numbers 
$i(k)>0, m(k)>0$ so that for all $[\eta]\in V_{i(k)}$, 
the diameter of the image under $\Upsilon$ of the Teichm\"uller geodesic segment 
$\exp_X([0,m(k)]\rho[\eta])$ is at least $k$.
As a consequence, if $X_i\in \overline{Z(i,i)}$ 
for each $i$, then by compactness of $\overline{{\cal T}(S)}$, up to passing to 
 a subsequence the sequence $X_i$ converges to a point $\zeta\in {\cal X}(S)\cap \partial 
  {\cal C\cal G}(S)$. That this point has to coincide with $\xi$
 is an immediate consequence of the discussion in the proof of Claim 1 above. 
This completes the proof of the claim.
\hfill $\blacksquare$
\end{proof}

\begin{lemma}\label{retraction}
The sets $\overline{Z(i,j)}_{\rm small}\cap {\cal T}_\epsilon(S)$ 
are contractible.
\end{lemma}
\begin{proof}
Since for each $i$ the set $\mathring V_i$ is a contractible subset of 
the space of projective measured geodesic laminations, 
identified with the unit sphere in the cotangent space of 
${\cal T}(S)$ at $X$, the set 
\[H(i,j)=\cup_{[\nu]\in \mathring V_i}\{\exp_X(t \rho[\nu])\mid t > j\}\subset 
{\cal T}(S)\]
is open and contractible since it is homeomorphic to 
$\mathring V_i\times (j,\infty)$. This uses the fact that the Teichm\"uller exponential map
at $X$ is a homeomorphism of $T_X^*{\cal T}(S)$ onto ${\cal T}(S)$. 

But $Z(i,j)$ is the image of 
$H(i,j)$ under the homeomorphism $\Lambda_\epsilon:{\cal T}(S)\to \mathring {\cal T}_\epsilon(S)$ and hence $Z(i,j)$ is contractible.
Then by Lemma \ref{smallcontr}, the small closure of $Z(i,j)$
in ${\cal T}_\epsilon (S)$ is contractible as well.  
\end{proof}

\begin{proof}[Proof of Proposition \ref{curvegraphbdneigh}]
By Lemma \ref{descend}, the small closure 
$\overline{Z(i,j)}_{\rm small}$ of $Z(i,j)$ in $\overline {\cal T}(S)$ 
is a neighborhood of $\xi$. Lemma \ref{retraction} shows that its 
intersection with ${\cal T}_\epsilon(S)$ is contractible.
Using again Lemma \ref{descend}, the countably many
such sets define a countable neighborhood basis of $\xi$ in $\overline{\cal T}(S)$
whence the proposition.
\end{proof}

\subsection{Neighborhoods of minimal filling laminations for disconnected surfaces}\label{bdproduct} 

In this section we consider a disjoint union $S=\sqcup_{i=1}^k S_i$ of finitely many 
connected surfaces of finite type. 
Our goal is to construct for any point in 
\[{\cal E}=\partial {\cal C\cal G}(S_1)*\cdots *\partial {\cal C\cal G}(S_k)\subset {\cal X}(S)\] 
a neighborhood basis in 
$\overline{\cal T}(S)$ consisting of sets whose intersections with 
${\cal T}_\epsilon(S)$ are small closures of open contractible subsets of 
$\mathring {\cal T}_\epsilon(S)$.

\begin{remark}\label{directproducts}
In \cite{Ti11}, it was shown that if two groups $\Gamma_1,\Gamma_2$ admit 
${\cal E\cal Z}$-structures $(X_1,Z_1)$ and $(X_2,Z_2)$, then the direct product 
$\Gamma_1\times \Gamma_2$ admits an ${\cal E\cal Z}$-structure consisting of 
a compactification of the product $(X_1\setminus Z_1)\times (X_2\setminus Z_2)$ 
by adding the join $Z_1*Z_2$. Unfortunately,  we can not use this 
result directly in an inductive step 
as we need more precise information for the proof of 
Theorem \ref{mainprecise}. The argument in \cite{Ti11} is based on 
the definition of an ${\cal E\cal Z}$-structure using deformation retractions of neighborhoods of 
boundary points \cite{B96} and a subtle 
combination of these deformation retractions which leads to deformation retractions of 
neighborhoods of points in $Z_1*Z_2$ in the space 
$X_1\times X_2\cup  Z_1*Z_2$. This construction is not very well suited for essentially working
with non-proper spaces. 
\hfill $\blacksquare$
\end{remark}

The set ${\cal E}$ is the set of formal sums
$\sum_i a_i\xi_i$ where $\xi_i\in \partial {\cal C\cal G}(S_i)$ and $a_i \geq 0, \sum_ia_i=1$.
Choose a basepoint $X=(X_1,\dots,X_k)\in {\cal T}_\epsilon(S)=\prod {\cal T}_\epsilon(S_i)$ and let 
$x_i\in {\cal C\cal G}(S_i)$ be a component of a short marking $\mu(X_i)$ for $X_i$. 
Following Section \ref{subsec:filling}, we begin with constructing a neighborhood basis in $\overline{\cal T}(S)$ 
for each $\xi\in {\cal E}$, and in a second step we improve this basis to a basis
consisting of small closures of contractible open subsets of $\mathring{\cal T}_\epsilon(S)$.

Thus let $\xi=\sum_ia_i\xi_i$ and assume by reordering that there exists some $\ell \leq k$ such that
$a_i>0$ if and only if $1\leq i\leq \ell$. For $j,\ell, m>0$ define  
\[W(\xi,j,\ell,m)\subset {\cal T}_\epsilon(S)\] to be the set of all hyperbolic metrics $Y=(Y_1,\dots,Y_k)$ 
on $S$ with the following
properties. In their formulation, we view a short marking $\mu(Y)$ of a hyperbolic metric 
$Y\in {\cal T}_\epsilon(S_i)$ as a bounded subset of ${\cal C\cal G}(S_i)$, and we continue to use the 
notation ${\rm pr}_V$ to denote the subsurface projection into a subsurface $V$. 
\begin{enumerate}
\item $d_{{\cal C\cal G}(S_i)}(\mu(Y_i),x_i)\geq j$ for $1\leq i\leq \ell$.
\item $\frac{d_{{\cal C\cal G}(S_i)}(\mu(Y_i),x_i)}{d_{{\cal C\cal G}(S_1)}(\mu(Y_1),x_1)}\in 
[(1-\frac{1}{m}) \frac{a_i}{a_1},(1+\frac{1}{m})\frac{a_i}{a_1}]$ for all $1\leq i\leq \ell$.
\item For any subsurface $V$ of $\cup_{i>\ell}S_i$, it holds
\[d_{{\cal C\cal G}(V)}({\rm pr}_V(\mu(Y)), {\rm pr}_V(\mu(X)))< \frac{1}{m} d_{{\cal C\cal G}(S_1)}(\mu(Y_1),x_1).\]
\item For each $i\leq \ell$ a geodesic in ${\cal C\cal G}(S_i)$ connecting $x_i$ to 
$\mu(Y_i)$ can be extended to a $p$-quasi-geodesic in ${\cal C\cal G}(S_i)$ whose endpoint is contained
in the ball of radius $e^{-\ell}$ about $\xi_i$ in $\partial {\cal C\cal G}(S_i)$, where the metric on 
$\partial {\cal C\cal G}(S_i)$ is the Gromov distance $d_{x_i}$ constructed from the basepoint $x_i$.
\end{enumerate}

The following statement is completely analogous to Lemma \ref{closure}, and the proof of Lemma \ref{closure}
carries over without modification and will be omitted.

\begin{lemma}\label{closure2}
For each $\xi\in {\cal E}$ the closures of the sets $W(\xi,j,\ell,m)$ in $\overline{\cal T}(S)$ define a neighborhood basis 
of $\xi$ in $\overline{\cal T}(S)$.
\end{lemma}

Using Lemma \ref{closure2}, we obtain the following analog of Corollary \ref{cor} which will be used in Section \ref{Sec:metrizability}.

\begin{corollary}\label{cor2}
There are countably many open sets $U_j\subset \overline{\cal T}(S)$ which contain a neighborhood basis for 
every $\xi\in {\cal E}$. 
\end{corollary}
\begin{proof} For each $i\leq k$ choose a countable dense subset $\xi_i^j$ of 
$\partial {\cal C\cal G}(S_i)$. Then the set 
$\{\sum_ia_i\xi_i^{j_i}\mid a_i\in \mathbb{Q},\sum_ia_i=1, j_i>0\}$ is countable and dense in 
${\cal J}(\cup_iS_i)$ by the definition of the join. 

It then suffices to show that the closures in $\overline{\cal T}(S)$ of the countably many sets 
$W(\sum_ia_i\xi_i^{j_i},j,\ell,m)$ contain a neighborhood basis for any $\xi\in {\cal E}$. That this holds true follows 
from the definition of the topology on $\overline{\cal T}(S)$ as in the proof of Corollary \ref{cor}.
\end{proof}

Recall from Section \ref{subsec:filling} that for each $i$ 
the choice of a basepoint $X_i\in {\cal T}_\epsilon(S_i)$ determines a section 
$\rho_i:{\cal P\cal M\cal L}(S_i)\to {\cal M\cal L}(S_i)$. Assume that the basepoint
$x_i\in {\cal C\cal G}(S_i)$ is a component of the pants decomposition of 
$\mu(X_i)$. 

For simplicity of notation, 
call a function $f:\mathbb{R}\to \mathbb{R}$ \emph{coarsely non-decreasing},
with control constant $q>0$, if 
we have $f(t)\geq f(s)-q$ for all $s\leq t$. 
Then for every 
projective measured geodesic lamination
$[\nu_i]$ on $S_i$ 
the function
\[t\to d_{{\cal C\cal G}(S_i)}(\Upsilon(\exp_{X_i}(t \rho_i[\nu_i])),x_i)\] 
is coarsely non-decreasing, with control constant
only depending on the complexity of $S_i$ \cite{MM99}. 
The following  was shown in \cite{H09}.

\begin{lemma}\label{continuous}
There exists a continuous 
function
\[\delta_{x_i}:{\cal T}(S_i)\to [0,\infty)\]
which is at uniformly bounded distance from the function 
$Y_i\to d_{{\cal C\cal G}(S_i)}(\Upsilon(Y_i),x_i)$.
\end{lemma}

To construct open contractible 
subsets of $\prod \mathring {\cal T}_\epsilon(S_i)$ whose small closures define 
neighborhoods of $\sum_ia_i\xi_i$ in $\overline{\cal T}(\cup_i S_i)$, 
we shall control the speed of progress in the 
curve graph of each of the surfaces $S_i$. To this end note that by Lemma \ref{continuous}, 
for every Teichm\"uller geodesic $\gamma:\mathbb{R}\to {\cal T}(S_i)$ starting at the fixed basepoint
$X_i$, the 
function $t\to \delta_{x_i}(\gamma(t))$ is coarsely non-decreasing and continuous.
We use this to construct 
a new parameterization 
of a Teichm\"uller 
geodesic starting from $X_i$ 
which encapsulates its progress in the curve graph. The construction
is based on the following
elementary observation. Here the distance between two functions 
$f,g:J\subset \mathbb{R}\to \mathbb{R}$ is defined 
as 
$\Vert f-g\Vert =\sup \{\vert f(t)-g(t)\vert \mid t\}$.

\begin{lemma}\label{monotone}
  Let $f:\mathbb{R}^n\to [0,\infty)$ be a continuous
 function whose restriction to each ray
$t\to tx$ $(x\in S^{n-1}\subset \mathbb{R}^n)$ is 
coarsely non-decreasing, with fixed control constant $q>0$. 
Then 
\[u=\inf\{g\mid g\geq f,g \text{  continuous, non-decreasing on rays}\}\]
is non-decreasing on rays, continuous
and at distance at most $q$ from $f$. 
\end{lemma}
\begin{proof}
For $x\in S^{n-1}$ and $t\geq 0$ put 
\[u(tx)=\max\{f(sx)\mid s\leq t\}.\] 
This makes sense since $f$ is continuous. By definition, $u$ is non-decreasing on rays, 
$u\geq f$ and $u-f\leq q$ as $f$ is coarsely non-decreasing.

Since $f$ is continuous, it is also immediate that $u$ is continuous.
This shows the lemma.
 \end{proof}

Let $f_{i}$ be the function on
$T_{X_i}^*{\cal T}(S_i)\sim \mathbb{R}^{m_i}$ constructed in 
Lemma \ref{monotone} from the function
$\delta_{x_i}\circ \,\exp_{X_i}$.
For each $[\nu]\in {\cal P\cal M\cal L}(S_i)$ the restriction of the function 
$f_i$ to the ray $t\rho_i[\nu]$ $(t\in (0,\infty))$ 
is non-decreasing, but it may be constant on
arbitrarily large intervals. However, by replacing $f_i$ by $f_i+\alpha_i$
where $\alpha_i(t\rho_i[\nu])=a(t)$ for a smooth strictly increasing function
$[0,\infty)\to [0,1)$, 
we may assume that  the function $f_i$
has the following properties.
\begin{enumerate}
\item The function $f_i:T^*_{X_i}{\cal T}(S_i)\to [0,\infty)$ is continuous
and strictly increasing on rays starting at $0$.
\item $\sup \vert f_i-\delta_{x_i}\circ \exp_{X_i}\vert  \leq q+1$. 
\end{enumerate} 
In particular, if $f_i \vert \{t\rho_i[\nu] \mid t\in (0,\infty)\}$ 
is unbounded, then 
$f_i\vert \{t\rho_i[\nu]\mid t\}$ is a homeomorphism onto $[0,\infty)$.

Put $\tau [\nu]=\sup\{f_i(t\rho_i[\nu])\mid t\}$. Note that 
$\tau [\nu]=\infty$ if the support of the geodesic lamination $[\nu]$ on $S_i$ fills $S_i$.

Since $f_i$ is continuous and its restriction to each ray 
$\{t\rho_i[\nu]\mid t\geq 0\}$ is a homeomorphism onto $[0,\tau[\nu])$,  
it can be inverted. 
We then can define a function 
$g_{[\nu]}$ on $[0,\tau [\nu])$ 
by 
\[g_{[\nu]}(t)=\{s\mid f_i(s\rho[\nu])=t\}.\]
Using this function, we obtain a parameterization 
$t\to \hat \gamma_ {[\nu]}(t)$ of the Teichm\"uller geodesic 
$t\to \exp_{X_i}(t\rho_i[\nu])$ on the interval $[0,\tau [\nu])$ 
by defining 
\begin{equation}\label{hatgamma}
\hat \gamma_{[\nu]}(t)=\exp_{X_i}(g_{[\nu]}(t)\rho_i[\nu]).\end{equation}
With this definition, we know that
$\vert d_{{\cal C\cal G}(S_i)}(x_i,\hat \gamma_{[\nu]}(t))-t \vert \leq b$
where $b>0$ is a universal constant not depending on $t$ or $i$.

By reordering of components, consider $\xi=\sum_{i=1}^\ell a_i\xi_i\in 
{\cal X}(\cup_iS_i)$ where $1\leq \ell \leq k$, $a_i>0$ and
$\xi_i\in \partial {\cal C\cal G}(S_i)$ 
for all $i$. 
For $1\leq i\leq \ell$ we apply the construction which was carried out in Lemma \ref{descend}:
let $V^1_i\supset V^2_i\supset \cdots$ be a
closed descending chain of contractible neighborhoods of the polytope
$P(\xi_i)$ of projective measured geodesic laminations
supported in $\xi_i$ in the sphere ${\cal P\cal M\cal L}(S_i)$
of projective measured geodesic laminations on $S_i$. 
If $S_i$ is an annulus, then 
by convention, ${\cal P\cal M\cal L}(S_i)$ consists of two points. 
We assume that each of the sets $V_i^j$ is homeomorphic to a closed ball and that 
for each $j\geq 1$ there exists a deformation retraction
$R_i^j:V_i^j\to V^{j+1}_i$ which maps 
$V_i^j\setminus V^{j+1}_i$ into $V^{j+1}_i\setminus V_i^{j+2}$. By definition,
$R_i^j$ equals the 
identity on $V_i^{j+1}$. 
We also may assume that there exists an increasing
sequence $m(j)\to \infty$ so that for every $i\leq \ell$ and every 
$[\nu]\in V_i^j$
the following properties are satisfied. 
\begin{enumerate}
\item 
$\tau [\nu]\geq 2m(j)$.
\item If the support of $\zeta\in V_i^j$ is minimal
and fills, and if $c$ is a shortest distance projection of 
${\rm supp}(\zeta)\in \partial {\cal C\cal G}(S_i)$ into a $p$-quasi-geodesic connecting the basepoint 
$x_i$ to $\xi_i$, then $d_{{\cal C\cal G}(S_i)}(c,x_i)\geq 2m(j)$.
\end{enumerate}
Recall to this end that $\tau[\nu]=\infty$ for every
$\nu\in P(\xi_i)$ since $\xi_i$ is minimal and filling by assumption, and 
that a shortest distance projection of ${\cal C\cal G}(S_i)$ into any
$p$-quasigeodesic connecting $x_i$ to $\xi_i$  
extends to $\partial {\cal C\cal G}(S_i) \setminus \xi_i$.

For a pair of points $X,Y\in {\cal T}(S_i)$ define
\[\hat d_{\cal T}(X,Y)=\max \{d_{{\cal C\cal G}(V)}({\rm pr}_V(\mu(X)),{\rm pr}_V(\mu(Y)))
 \mid V\}\]
where the maximum is over all subsurfaces $V$ of $S_i$ and $\mu(X),\mu(Y)$ are short markings.

\begin{theo}[Theorem B of \cite{R14}]\label{increase}
For any Teichm\"uller geodesic $\gamma:[0,\infty)\to {\cal T}(S_i)$, 
the function $t\to \hat d_{\cal T}(\gamma(0),\gamma(t))$ is coarsely non-decreasing, with control constant
not depending on $\gamma$. 
\end{theo}  
\begin{proof}
By Theorem B of \cite{R14}, 
there is 
a number $p>0$ only depending on the complexity of $S$ such that
for every subsurface $V$ of $S$,  
the image under the map ${\rm pr}_V\circ \mu$ of a Teichm\"uller geodesic $\gamma:\mathbb{R}\to 
{\cal T}(S)$ is an unparameterized $p$-quasi-geodesic in 
${\cal C\cal G}(V)$. 
This quasi-geodesic may be bounded, one-sided infinite or
two-sided infinite. Since ${\cal C\cal G}(V)$ is a hyperbolic geodesic metric space, 
this implies that 
the path $t\to {\rm pr}_V\circ \mu \circ \gamma(t)$ coarsely does not backtrack: There exists
a universal constant $q>0$ not depending on the subsurface $V$ such that 
for $0\leq s\leq t$, it holds
\[d_{{\cal C\cal G}(V)}({\rm pr}_V(\mu(\gamma(0))), {\rm pr}_V(\mu(\gamma(t)))) \geq
d_{{\cal C\cal G}(V)}({\rm pr}_V(\mu(\gamma(0))),{\rm pr}_V(\mu(\gamma(s))))-q.\] 
As the projections ${\rm pr}_V(\mu(X))$ only coarsely determine a point in the curve
graph of $V$, the distances in this formula are only coarsely well defined, but this does not
affect the validity of the estimate. 

As a consequence,
for every subsurface $V$ of $S_i$ and every Teichm\"uller geodesic $\gamma:
[0,\infty)\to {\cal T}(S_i)$ the function 
\[t\to d_{{\cal C\cal G}(V)}({\rm pr}_V(\mu(\gamma(0))), {\rm pr}_V(\mu(\gamma(t))))\]
is coarsely non-decreasing, with control constant $q$ not depending on $V$. Then the same holds
true for $\hat d_{\cal T}$. 
\end{proof}

The following proposition is the technically most involved
part of the proof of our main theorem. In its formulation, 
we denote by $\Lambda_i=\Lambda_{\epsilon,i}:{\cal T}(S_i)\to 
\mathring {\cal T}_\epsilon(S_i)$ 
a homeomorphism as constructed in 
Proposition \ref{homeomorphism}. Then 
$\Lambda_\epsilon=\prod \Lambda_{\epsilon,i}$ is a homeomorphism of
$\prod {\cal T}(S_i)$ onto $\prod \mathring {\cal T}_\epsilon(S_i)$.

\begin{proposition}\label{contractible}
Assume that 
$\xi=\sum_{i=1}^k a_i\xi_i\in {\cal X}(S)$ is such that $\sum_i a_i=1$, that 
$a_i>0$ precisely if $i\leq \ell$ 
and that $\xi_i\in \partial {\cal C\cal G}(S_i)$
for all $i$. 
For integers $j,n\geq 1$ and for $\delta >0$ 
there is an open subset $E(j,n,\delta)$ of $\prod {\cal T}(S_i)$ 
with the following property.
\begin{enumerate}
    \item $E(j,n,\delta)$ 
    is contractible. 
\item $E(j,n,\delta)\subset E(j^\prime,n^\prime,\delta^\prime)$ for 
$j\geq j^\prime, n\geq n^\prime, \delta \leq \delta^\prime$.
\item The small closures of the sets 
$\Lambda_\epsilon E(j,n,\delta)$ ($j\geq 1,n\geq 1,\delta >0$)
in $\overline{\cal T}(\cup_iS_i)$ 
define a neighborhood basis of $\xi$.
\end{enumerate}
\end{proposition} 
\begin{proof} For $i\leq \ell$ let $V_i^j$ be as above. 
For each $[\nu_i]\in V_i^j$ choose a parameterization of the geodesic 
$t\to \exp_{X_i}(t \rho_i[\nu_i])$ on $[0,\tau[\nu_i])$ as constructed 
in equation (\ref{hatgamma}) above. 
Note that by the choice of the constants 
$m(j)$, the domain of definition 
of this parameterization contains the interval
$ [0,m(j)]$, and 
the restriction of this parameterization to $[0,m(j)]$ 
depends
continuously on $[\nu_i]$. 
Denote by $\hat \gamma_{[\nu_i]}:[0,\tau [\nu_i])\to {\cal T}(S_i)$ this 
parameterization. 

Theorem \ref{increase} shows
that for any Teichm\"uller geodesic $\gamma:[0,\infty)\to {\cal T}(S_i)$, the function
$t\to \hat d_{\cal T}(\gamma(0),\gamma(t))$ is coarsely non-decreasing, with 
fixed control constant $q>0$. Put 
$\tilde d_{\cal T}(\gamma(0),\gamma(t))=
\sup_{s\leq t} \hat d_{\cal T}(\gamma(0),\gamma(s))$. By uniqueness of Teichm\"uller geodesics
between any pair of points, this defines a function ${\cal T}(S_i)\times {\cal T}(S_i)\to [0,\infty)$
which however may not be symmetric.
For any Teichm\"uller geodesic $\gamma$, the function 
$t\to \tilde d_{\cal T}(\gamma(0),\gamma(t))$ is non-decreasing.

For an $\ell$-tuple $(j_1,\dots,j_\ell)\in \mathbb{N}^\ell$ put
$m(j_1,\dots,j_\ell)=\min \{m(j_i)\mid i\}$. 
For $i\leq \ell$ and $j\geq 1$ put $\hat V_i^j =V_i^j\setminus V^{j+1}_i$.
For
$([\nu_1],\dots,[\nu_\ell])\in \hat V^{j_1}_1\times \dots \times \hat V^{j_\ell}_\ell$ 
and $\delta >0$, $n<m(j_1-1, \dots,j_\ell -1)/2$  define 
\begin{align*}
  F(n,\delta,[\nu_1],\dots,[\nu_\ell]) &
           =\{(\hat \gamma_{[\nu_1]}(t_1),\hat \gamma_{[\nu_2]}(t_2),\dots,
\hat \gamma_{[\nu_\ell]}(t_\ell),z_{\ell+1},\dots,z_k)\in 
\prod {\cal T}(S_i)\mid \\
t_i\geq n, & \quad
\vert  t_i/t_1-a_i/a_1\vert  
\leq \delta  \text{ if }t_i < m(j_1-1, \dots, j_\ell-1) \, \text{  for }  i\leq \ell, \\
&
\tilde d_{{\cal T}(S_i)}(X_i,z_i)<\delta t_1 \text{ for }i\geq \ell+1\}.\end{align*}
It is important to note that the constraints on the relation between the points 
in the components of a tuple in the set $F(n,\delta,[\nu_1],\dots,[\nu_\ell])$ become
stronger as the components $[\nu_i]$ get close to $P(\xi_i)$ as measured by the 
nested sequence $V_i^j$, and that the constraint
is determined by the component in the tuple which is furthest away from $P(\xi_i)$. 

\smallskip\noindent
{\bf Claim 1:} 
The set 
$E(j,n,\delta)=\cup_{[\nu_i]\in V^j_i}F(n,\delta,[\nu_1],\dots,[\nu_\ell])$
is contractible for every $n\leq m(j)$.

\smallskip\noindent
\emph{Proof of Claim 1.}
Note first that if $(z_1,\dots,z_\ell,z_{\ell+1},\dots,z_k)\in 
E(j,n,\delta)$ then the same holds true for 
$(z_1,\dots,z_\ell,z^\prime_{\ell+1},\dots,z_k^\prime)$ for any
$z_i^\prime$ which is contained in the Teichm\"uller geodesic connecting
$X_i$ to $z_i$ and all $i\geq \ell+1$. Thus retracting component wise the last
$k-\ell$ components $z_i$ to the basepoint $X_i$ $(i\geq \ell+1)$ along the unique Teichm\"uller geodesic connecting
$X_i$ to $z_i$ 
and keeping the remaining components 
fixed defines a retraction of $E(j,n,\delta)$ to 
$E(j,n,\delta)\cap \{(z_1,\dots,z_k)\mid z_i=X_i\text{ for }\ell+1\leq i\leq k\}.$
In particular, in the remainder of the construction, it suffices to assume that 
$\ell=k$.

Next observe that $E(j+1,n,\delta)\subset E(j,n,\delta)$ for all $j,n,\delta$.
We construct a homotopy of $E(j,n,\delta)$ into
its subset $E(j+1,n,\delta)$ as follows.

The set
\begin{equation*}
 S(m(j-1)) =
 \{(t_1,\dots,t_{k})\in [n,\infty)^k\mid 
\vert t_i/t_1- a_i/a_1\vert \leq \delta
  \text{ if } t_i< m(j-1) \, \forall i\}\end{equation*}
admits a deformation retraction onto its subset
\[S(m(j))= \{(t_1,\dots,t_{k})\in [n,\infty)^k\mid  \vert t_i/t_1- a_i/a_1\vert \leq \delta
  \text{ if } t_i< m(j) \, \forall i\}.\]
Namely, define a homotopy 
$h:[0,1]\times [n,\infty)\to [n,\infty)$ by
\[h(u,t)=\begin{cases}
\min\{t(1-u+u(m(j)/m(j-1))),m(j)\} \text{ if } t < m(j) \\
t \text { if } t\geq m(j).
\end{cases}\] 
Then for any $t_1,t_2\in [n,m(j-1))$ and $u$ with $h(u,t_1)<m(j),
h(u,t_2)<m(j)$ 
we have
$h(u,t_1)/h(u,t_2)=t_1/t_2$. As a consequence, 
the map 
$(u,(t_1,\dots,t_k))\to (h(u,t_1),\dots, h(u,t_k))$ 
preserves $S(m(j-1))$, and it defines a homotopy of 
$S(m(j-1))$ into $S(m(j))$.

Composing this deformation of the domain $S(m(j-1))$ into $S(m(j))$ 
with the map 
\[(t_1,\dots,t_k)\to 
(\hat \gamma_{[\nu_1]}(t_1),\dots, \hat \gamma_{[\nu_k]}(t_k))\] defines a homotopy 
of $E(j,n,\delta)$ into its subset 
\begin{align*} \Xi= E(j,n,\delta)\cap 
\{(\hat \gamma_{[\nu_1]}(t_1),\dots \hat \gamma_{[\nu_k]}(t_k))\mid 
\vert t_i/t_1- a_i/a_1\vert \leq \delta \text{ if } t_i< m(j) \,\forall i\}.
\end{align*} 

The deformation retractions $R_i^j:[0,1]\times V_i^j\to V_i^j$ of 
$V_i^j$ onto $R_i^j(V_i^j\times \{1\})=V^{j+1}_i$ 
induce a deformation retraction  
\[R^j: [0,1]\times 
V^j_1 \times \cdots  \times  V^j_k\to 
V^{j}_1 \times \cdots  \times V^{j}_k\]
onto 
$V^{j+1}_1 \times  \cdots  \times V^{j+1}_k$ by applying the deformation retractions 
$R^j_i$ component wise. 
Since for each $i$, the image of $V^j_i\setminus V^{j+1}_i$ is contained 
in $V^{j+1}_i\setminus V^{j+2}_i$, and $R^j_i$ equals the identity on $V_i^{j+1}$,  
we obtain a deformation retraction
of $\Xi$ onto its subset $E(j+1,n,\delta)$ by defining 
\[(s,(\gamma_{[\nu_1]}(t_1),\dots, \gamma_{[\nu_k]}(t_k))) \to  
(\gamma_{R_1^j(s,[\nu_1])}(t_1),\dots, \gamma_{R^j_k(s,[\nu_k])}(t_k)).\]
The composition of these two homotopies 
yields a homotopy 
of $E(j,n,\delta)$ into
$E(j+1,n,\delta)$. 

Now $\cap_j E(j,n,\delta)=
\cup_{[\nu_i]\in P(\xi_i )}F(n,\delta,[\nu_1],\dots, [\nu_k])$, and as $P(\xi_i)$ is 
contractible for all $i$, this
set is contractible as well. Concretely, for tuples 
$([\nu_1],\dots,[\nu_k])\in \prod P(\xi_i)$, 
we have $m(j)=\infty$ and 
the constraint on the time parameters $t_i$ in the definition of the set
$F(n,\delta,[\nu_1],\dots,[\nu_k])$ holds for all $t_i\geq n$. Since 
each of the sets $P(\xi_i)$ are contractible, we can define a retraction of 
$\cap_j E(j,n,\delta)$ onto $F(n,\delta,[\nu_1],\dots, [\nu_k])$ for a fixed choice
of $[\nu_i]\in P(\xi_i)$ and all $i$, keeping time parameters $t_i$ fixed. We are then left with
contracting the time parameters, which is easily possible by the definition of the time constraint. 
This completes the proof of the claim.
\hfill $\blacksquare$
\smallskip

So far we constructed from a tuple of 
contractible neighborhoods $V^j_i$ $(i=1,\dots,k)$ and numbers $j>0,\delta >0$ a contractible 
subset $E(j,n,\delta)$ 
of ${\cal T}(S)=\prod {\cal T}(S_i)$. We aim at using these sets to construct contractible 
neighborhoods of $\xi$ in $\overline{\cal T}(\cup_iS_i)$. 

\smallskip\noindent
{\bf Claim 2:}
For fixed $(j,n,\delta)$, 
if $X^u\subset 
\prod {\cal T}_\epsilon(S_i)$
is a sequence converging to $\xi$, then
$X^u\in \overline{\Lambda_\epsilon E(j,n,\delta)}$ for large enough $u$.
\smallskip

\noindent
\emph{Proof of Claim 2:} 
Let $X^u=(X_1^u,\dots,X_k^u) \subset \prod {\cal T}_\epsilon(S_i)$ be a sequence converging to $\xi$.  
Let $j>0,n>0, \delta >0$ be fixed. 
We show first that $X^u\in E(j,n,\delta)$ for large enough $u$.

For $i\leq \ell$ let $[\nu_i]\in {\cal P\cal M\cal L}(S_i)$ be such that ${\rm supp}([\nu_i])=\xi_i$. 
Let $\hat \gamma_i:[0,\infty)\to {\cal C\cal G}(S_i)$ be the $p$-quasi-geodesic constructed
as a reparameterization of the unparameterized quasi-geodesic 
$\gamma_i(t)=\Upsilon(\exp_{X_i} t[\nu_i])$ as before
and let $\Pi_i:{\cal C\cal G}(S_i)\to \hat \gamma_i$ be
a shortest distance projection. We then have
\begin{equation}\label{ratio}
  d_{{\cal C\cal G}(S_i)}(\Pi_i(\Upsilon(X_i)),\Pi_i(\Upsilon(X_i^u))) /
d_{{\cal C\cal G}(S_1)}(\Pi_1(\Upsilon(X_1)),\Pi_1(\Upsilon(X_1^u)))\to a_i/a_1.\end{equation}
Furthermore, for $j\geq \ell+1$
it holds
\begin{equation}\label{hatd}
\hat d_{\cal T}(X_j,X_j^u)/\min_{i\leq \ell}d_{{\cal C\cal G}(S_i)}
(\Pi_i(\Upsilon(X_i)),\Pi_i(\Upsilon(X_1^u)))\to 0.\end{equation}

Let $[\eta_i^u]\in {\cal P\cal M\cal L}(S_i)$ and $t_i^u\geq 0$ be such that 
$X_i^u=\hat \gamma_{[\eta_i^u]}(t_i^u)$. Then for all $i\leq \ell$, we have 
$t_i^u\to \infty$ $(u\to \infty)$, moreover 
by Lemma \ref{descend} and its proof,
it holds $[\eta_i^u]\to P(\xi_i)$ $(u\to \infty)$. Thus 
for large enough $u$ and all $i\leq \ell$, we have 
$[\eta_i^u]\in V_{j}^i$.
As $t_i^u\to \infty $ $(u\to \infty)$ for all $i$, equation (\ref{hatd}) shows that 
$\hat d_{\cal T}(X_i,X_i^u) < \min_{m\leq \ell} \delta t_m^u/2$
for sufficiently large $u$ and all $i\geq \ell +1$.
By the definition of the set $E(j,n,\delta)$, this implies that 
$X^u\in E(j,n,\delta)$ for large enough $u$ if and only if 
this holds true for 
$(X_1^u,\dots,X_\ell^u,X_{i+1},\dots,X_k)$ (here as before, $X_i$ is the basepoint). 
Consequently it follows as in the beginning of this proof  that 
it suffices to assume that $a_i>0$ for all $i\leq k$, in other words, that 
$\ell=k$. 

Now invoking the condition (\ref{ratio}) and the fact that 
$t_i^u\to \infty$ for all $i$, for all sufficiently large $u$ we have
$t_i^u>2n$ and 
\[ \vert d_{{\cal C\cal G}(S_i)}(\Pi_i(\Upsilon(X_i)),\Pi_i(\Upsilon(X_i^u)))/
d_{{\cal C\cal G}(S_1)}(\Pi_1(\Upsilon(X_1)),\Pi_1(\Upsilon(X_1^u)))-a_i/a_1\vert 
<\delta/2\]
for all $i$ 
which together with the definitions implies that $X_u\in E(j,n,\delta)$.

We are left with observing that in fact $X^u\in \overline{\Lambda_\epsilon(E(j,n,\delta))}$. 
However, the map $\Lambda_\epsilon=\Lambda_{1,\epsilon}\times \cdots \times \Lambda_{k,\epsilon}$ 
is coarsely $\Upsilon$-invariant for each $i$.
As the defining properties of the sets $E(j,n,\delta)$ only depend on 
distances in the curve graph of the surfaces $S_i$, we conclude that 
$X^u\in \overline{\Lambda_\epsilon(E(j,n,\delta))}$ for large $u$. 
\hfill $\blacksquare$

Claim 2 shows that each of the sets $\overline{\Lambda_\epsilon E(j,n,\delta)}$ is the intersection
with ${\cal T}_\epsilon(S)$ of a neighborhood of $\xi$. 
By the definition of the \emph{small} closure of a subset of 
${\cal T}(S)$, the small closures of 
the sets $\Lambda_\epsilon(E(j,n,\delta))$ are then neighborhoods of $\xi$ as well.

To complete the proof of the proposition it remains to show that the small closures
$\overline{\Lambda_\epsilon E(j,n,\delta)}_{\rm small}$ 
of the sets $\Lambda_\epsilon E(j,n,\delta)$
form a neighborhood basis of $\xi$ in $\overline {\cal T}(\cup_iS_i)$. 
By Proposition \ref{topology2} and 
Proposition \ref{topsmallbd}, 
the space $\overline{\cal T}(\cup_iS_i)$ is a compact Hausdorff space. 
Thus the complement of an open neighborhood $V$ of $\xi$ is compact and does not 
contain $\xi$ and hence since the sets $E(j,n,\delta)$ are nested, 
we only have to show that the intersection
$\cap_{j,n,\delta} \overline{\Lambda_\epsilon E(j,n,\delta)}
\supset \cap_{j,n,\delta} \overline{\Lambda_\epsilon E(j,n,\delta)}_{\rm small}=\{\xi\}$. 
As $\xi$ clearly is contained in this intersection, it suffices to show that 
it is unique with this property. 

To see that this is the case recall from 
the end 
of the proof of Claim 1 that $\cap_j E(j,n,\delta)=
\cup_{[\nu_i]\in P(\xi_i )}F(n,\delta,[\nu_1],\dots, [\nu_k])$. 
Furthermore, using again the definitions, the intersection
$\cap_n \overline{\Lambda_\epsilon (\cup_{[\nu_i]\in P(\xi_i )}F(n,\delta,[\nu_1],\dots,[\nu_k]))}$ is contained in ${\cal X}(S)$, 
and finally 
$\cap_\delta \cap_n\overline{\Lambda_\epsilon(\cup_{[\nu_i]\in P(\xi_i )}F(n,\delta,[\nu_1],\dots,[\nu_k]))}=\{\xi\}$ which is what we wanted to
show.
\end{proof}

The following corollary summarizes the results of this subsection.

\begin{corollary}\label{basis10}
Each $\xi=\sum_ia_i\xi_i\in {\cal E}$ 
admits a countable neighborhood basis 
in $\overline{\cal T}(\cup_iS_i)$ consisting of 
small closures of open contractible subsets of $\prod \mathring {\cal T}_\epsilon(S_i)$.
\end{corollary}
\begin{proof} By Proposition \ref{contractible}, for each $\xi=\sum_i a_i\xi_i$ with 
$\xi_i\in \partial {\cal C\cal G}(S_i)$ the countably many sets 
$E(j,n,\frac{1}{m})$ $(j,\ell,m \in \mathbb{N})$ are open and contractible, and the small closures
of their images under $\Lambda_\epsilon$ define a neighborhood basis of $\xi$ in 
$\overline{\cal T}(S)$ whose intersections with ${\cal T}_\epsilon(S)$ are small closures
of contractible open as stated in the corollary. 
\end{proof}

\subsection{Neighborhoods of arbitrary points}\label{sec:arbitrary}

In this section we  complete the proof of the first part of Theorem \ref{mainprecise}. 
The difficulty is as follows. In Section \ref{bdproduct}, 
we considered a disjoint union 
$S_1,\dots,S_k$ of subsurfaces of $S$. Denote by $S_i^*$ the surface obtained from 
$S_i$ by replacing each boundary component by a puncture. 
We constructed for any point $\sum_ia_i \xi_i\in {\cal J}(\cup_i S_i)={\cal J}(\cup_iS_i^*)$ 
(the notations are as in Section \ref{titsbd}) a neighborhood basis of $\xi$ in 
$\overline{\cal T}(\cup_iS_i^*)=\prod {\cal T}_\epsilon(S_i^*)\cup {\cal X}(S_1^*)* \cdots *{\cal X}(S_k^*)$ 
consisting of open sets whose intersections with 
$\prod {\cal T}_\epsilon(S_i^*)$ are contractible. From the description of the topology on 
$\overline{\cal T}(S)$ in Section \ref{small}, we also know how to construct from this neighborhood
basis a neighborhood basis of $\xi$ viewed as an element of ${\cal X}(S)\subset \overline{\cal T}(S)$.
However, there is no straightforward mechanism which can be applied to guarantee that the 
intersections with ${\cal T}_\epsilon(S)$ of the sets from 
this neighborhood basis can be chosen to be contractible. 

Example \ref{farey2} illustrates both this difficulty and its solution. Namely, in the case 
of the once punctured torus $S$,
when we want to construct a contractible neighborhood basis for a fixed point of a Dehn twist, 
viewed as a point in the Gromov boundary of the hyperbolic group ${\rm Mod}(S)$, then  
this task is not well adapted to the geometry  of Teichm\"uller space, which equals the
hyperbolic plane. But we can 
think  of Teichm\"uller space as foliated by geodesics which are 
forward asymptotic to this fixed point, viewed as a boundary point of the hyperbolic plane, 
 and use this foliation to construct the desired neighborhood basis. We shall use an extension of this
 strategy to arbitrary Teichm\"uller spaces to complete the proof of Theorem \ref{mainprecise}. 

Define the \emph{complexity} $\kappa(S)$ of a connected surface of 
genus $g\geq 0$ with $m\geq 0$ holes (which can be boundary components or
punctures) as
\[\kappa(S)=3g-3+m\]
if $S$ is not a sphere with two holes, that is,  an annulus. If $S$ is an annulus 
then define $\kappa(S)=0$. If $S=\sqcup_{i=1}^m S_i$ is a disjoint union of 
connected surfaces $S_i$ then define 
$\kappa(S)=\sum_i \kappa(S_i)$.

We proceed 
by induction on the complexity $\kappa(S)$ of the possibly disconnected surface $S$. 
Example \ref{uniofann} contains the case $\kappa(S)=0$, so assume that part (1) of 
Theorem \ref{mainprecise} has been established for all 
surfaces of complexity at most $k-1$ for some $k-1\geq 0$. 
Let $S$ be a possibly disconnected surface of complexity
$\kappa(S)=k$. 
By Section \ref{subsec:filling} and Section \ref{bdproduct}, we are left with constructing neighborhood bases for 
points $\xi=\sum_{i=1}^ma_i\xi_i\in {\cal X}(S)$ 
where each $\xi_i$ fills a proper subsurface $S_i$ of $S$ (which may be a connected component of $S$) 
and that furthermore there exists at least one $i$ such that $\xi_i$ does not fill a connected component of $S$.
In particular, ${\rm supp}(\xi)$ fills a (possibly disconnected) subsurface of $S$, and the boundary of this subsurface is 
non-empty and contains at least 
one non-peripheral simple closed curve $c\subset S$.

Let $c$ be such a simple closed curve. Then $S_c=S\setminus c$
is a (possibly disconnected) surface of complexity $k-1$ and (with a small abuse of 
notation) we can write $S=S_c\sqcup A_c$ where $A_c$ is the annulus with core curve $c$.
We then can view $\xi$ as an element in the geometric boundary of the 
disconnected surface $S_c^*\sqcup A_c$ where again $S_c^*$ denotes the finite type surface obtained 
from $S_c$ by replacing each boundary component by a puncture. 
Since the complexity of $S_c^*$ is at most $k-1$, 
by the induction hypothesis, the first part of Theorem \ref{mainprecise} holds true for 
$S_c^*\sqcup A_c$.

The infinite cyclic group generated by the left Dehn twist $T_c$ about $c$
equals the mapping class group of $A_c$. The stabilizer ${\rm Stab}(c)$ 
of $c$ in the mapping class group 
${\rm Mod}(S)$ fits into the exact sequence
\begin{equation}\label{dehnexact}
1 \to \langle T_c\rangle  \to {\rm Stab}(c)\to {\rm Mod}(S_c^*)\to 1,\end{equation}
however this sequence does not split in general.

Let $S^\prime$ be the connected component of $S$ containing $c$. We are interested
in analyzing the geometric relation between ${\cal T}(S^\prime)$ and 
${\cal T}(S^\prime\setminus c)$. 
To make the notations less cumbersome, let us assume for the moment that $S$ is connected,
that is, in what follows, we tacitly replace $S^\prime$ by $S$ in the notation and resume the 
case of a disconnected surface later on.

Consider the \emph{augmented Teichm\"uller space} ${\cal T}^{\rm aug}(S)$ of $S$ \cite{Wo03,Ya04}. 
This is a stratified space whose
open stratum of maximal dimension equals 
the Teichm\"uller space ${\cal T}(S)$.
For each multi-curve $\beta$ on $S$ there exists a stratum ${\cal S}(\beta)$ which equals the 
Teichm\"uller space of the \emph{noded} surface $(S\setminus \beta)^*$ obtained from $S\setminus \beta$
by replacing each boundary component by a puncture (and remembering that the punctures are 
identified at the nodes). 
This Teichm\"uller space 
is a direct product of Teichm\"uller spaces, one for each component of $S\setminus \beta$.
The strata in the boundary of 
${\cal S}(\beta)$ correspond to multi-curves containing $\beta$ as a subset.

To an essential (that is, non-peripheral) 
 simple closed curve $c$ on the finite area hyperbolic surface $Y\in {\cal T}(S)$ we can associate not only 
its hyperbolic length $\ell_Y(c)$ 
but also its \emph{extremal length} $e_Y(c)$. By a result of Maskit \cite{Ma85}, if $\ell_Y(c)$ is small 
then $e_Y(c)$ is uniformly proportional to $\ell_Y(c)$: it holds $\ell_Y(c)/e_Y(c)\to \pi$ locally uniformly 
in $Y$ as $\ell_Y(c)\to 0$, with precise global ratio bound in the region in ${\cal T}(S)$ 
where $\ell_Y(c)$ is sufficiently small. 
We do not need to know the precise definition of $e_Y(c)$, all what matters
for our purpose is how it can be computed. Namely, by the Hubbard Masur theorem \cite{HM79}, 
for each $Y\in {\cal T}(S)$ there exists a unique holomorphic \emph{quadratic differential} $q_c(Y)$ on $Y$,
that is, a meromorphic section of $K_Y\otimes K_Y$ 
which defines a finite area singular flat metric on $Y$ with the following property. 

The marked Riemann surface $Y$ is glued
from an Euclidean cylinder $C$ of area one, with core curve freely homotopic to $c$, of 
circumference $w(Y)$ and height $h(Y)$, by pairwise identifying arcs of the same euclidean length 
in the boundary of $C$. Note that the cylinder $C$ is foliated by closed Euclidian geodesics which we 
call \emph{horizontal} in the sequel, referring to standard coordinates in the Euclidian plane. 
The extremal length of $c$ is then computed as 
$e_Y(c)=w(Y)/h(Y)$. The quadratic differential determined by this singular flat metric 
is called a \emph{one-cylinder Strebel differential}, 
with core curve $c$, and it represents a point in the cotangent space of ${\cal T}(S)$ at $Y$.
The map which associates to $Y\in {\cal T}(S)$ the one cylinder Strebel differential $q_c(Y)$ with core
curve $c$ is an analytic section $q_c$ of the bundle of area one quadratic differentials over ${\cal T}(S)$.
As \cite{HM79} only covers the case when $S$ is compact but we have to work with arbitrary finite 
type surfaces, 
we refer to Theorem 21.1 of \cite{St84} for a precise reference with proof.

For small $\delta <\epsilon/2\pi$ consider the set $N(c)$  of 
points $Y\in {\cal T}(S)$ so that the extremal length of $c$ for the conformal structure 
defined by $Y$ is at most $\delta$. By \cite{Ma85} we may assume that the hyperbolic 
length of $c$ for points in $N(c)$ is smaller than $\epsilon$. The set $N(c)$ can be thought of as 
a tubular neighborhood
of the stratum 
${\cal S}(c)$ of noded Riemann surfaces 
in the augmented  Teichm\"uller space ${\cal T}^{\rm aug}(S)$.
Its boundary $\partial N(c)$  
is invariant under the action of the infinite cyclic group of Dehn twists $\langle T_c\rangle$ about $c$. 
Namely, the action of  ${\rm Stab}(c)\subset {\rm Mod}(S)$ 
preserves the extremal length of $c$. 

The \emph{horocycle flow} $h_t$ acts on $q_c({\cal T}(S))$ as follows. 
For each $Y\in {\cal T}(S)$, the flow line 
$t\to h_t(q_c(Y))$ is obtained
by shearing (that is, twisting) 
the one-cylinder differential $q_c(Y)$ with constant speed along its core curve, where the 
speed parameter is chosen so that $h_1(q_c(Y))= q_c(T_c(Y))$ \cite{St84}. 
As $q_c$ is an embedding, this flow descends to a twist flow on ${\cal T}(S)$ by defining 
(with a small abuse of notation) $h_t(Y)=q_c^{-1}(h_t(q_c(Y)))$.  

To be more explicit, 
as is familiar for twisting a hyperbolic metric on $S$ along a simple closed geodesic,
one can think of this shearing operation as cutting $S$ open along a horizontal closed geodesic in the interior of 
the flat cylinder and gluing both sides back with a twist whose size is prescribed by the time parameter $t$.
As this twisting operation preserves height and circumference of the cylinder, this flow indeed 
preserves $q_c(\partial N(c))$ and hence projects to a flow on $\partial N(c)$. 
The flow lines foliate $\partial N(c)$ and are invariant under the infinite cyclic group 
$\langle T_c \rangle$. Note that the horocycle flow preserves the  \emph{criticial graph} 
of the differential, which is the marked metric graph obtained as the image of the 
boundary of the flat cylinder under the gluing operation which recovers the 
Riemann surface $Y$. If $S$ is closed then the critical graph is compact, otherwise
it can be compactified as a marked metric graph by adding the punctures. 
This compactification is connected if $c$ is non-separating, and it is disconnected if $c$ is separating. 
The diameter of each component 
of this compactification equals at most twice the circumference of the flat cylinder.

We use this and the fundamental result of Strebel (Theorem 23.5 of \cite{St84}, we refer to 
Section 2 of \cite{Zv02} for a precise account on what we need) in the next lemma.
Recall that the choice of a base marking of $S$ coarsely determines for each 
$Y\in {\cal T}(S)$ a \emph{twist parameter} $\tau(Y,c)\in \mathbb{Z}$ about $c$, unique up to an error of 
$\pm 1$, which is just the subsurface projection of $\mu(Y)$ into an annulus with core curve $c$.

\begin{lemma}\label{section}
There is a continuous surjective map 
$\Pi:\partial N(c)\to {\cal S}(c)$ and a continuous map $\sigma:{\cal S}(c)\to \partial N(c)$ with 
the following properties.
\begin{enumerate}
\item For each $Y\in \partial N(c)$, the preimage $\Pi^{-1}(\Pi(Y))$ equals the flow line
$t\to h_t(Y)$ of the horocycle flow. 
\item $\Pi\circ \sigma={\rm Id}$.
\item There exists a constant
 $b>0$ so that $\tau(\sigma(Y),c)\in [-b,b]$ for all 
    $Y\in {\cal S}(c)$.
\end{enumerate}
\end{lemma}
\begin{proof}
We begin with defining the projection $\Pi$. Let $Y\in \partial N(c)$ and consider the 
one cylinder Strebel differential $q_c(Y)$ with core curve $c$. Its critical graph 
${\cal G}(Y)$ is an embedded
graph in $S$, equipped with the restriction of the flat metric. Thus ${\cal G}(Y)$ is a metric
ribbon graph with two faces, that is, the boundary of a small tubular neighborhood of 
${\cal G}(Y)$ in $S$ has two components which correspond to the two boundary components of the 
cylinder which determines $q_c(Y)$. 

By Proposition 2.3 of \cite{Zv02}, this ribbon graph determines uniquely a meromorphic quadratic 
differential $q_c(\Pi(Y))$ 
on a marked Riemann surface $\Pi(Y)\in {\cal S}(c)$ with one node corresponding to the closed curve $c$.
The differential has a double pole 
at each of the two punctures corresponding to the node and 
perhaps a simple pole at a puncture of the surface $\Pi(Y)$ but no
other pole. The residues at the two double poles are negative real numbers and equal 
to the negative of the circumference of the cylinder defining
$q_c(Y)$. In particular, by the definition of $\partial N(c)$, these residues do not depend on $Y$.
Proposition 2.3 of \cite{Zv02} shows that the map $\Pi:\partial N(c)\to {\cal S}(c)$ is continuous. 
As $\Pi(Y)$ is constructed geometrically by attaching a one-sided infinite cylinder to the each 
face of the critical graph of the one cylinder Strebel differential $q_c(Y)$, viewed as a 
marked metric ribbon graph, 
clearly $\Pi^{-1}(\Pi(Y))$ contains the orbit of $Y$ under the horocycle flow $h_t$, and it equals this 
orbit as the orbit is uniquely determined by the (marked) critical graph of the Strebel differential. 
This shows the first part of the
lemma.

We next observe that the map $\Pi$ is surjective. 
Let $r>0$ be the circumference of the cylinders of the Strebel differentials for points in $\partial N(c)$. 
By Theorem 23.5 of \cite{St84}, given a noded Riemann surface $Z\in {\cal S}(c)$, and viewing the 
node as a pair $z_1,z_2$ of punctures, there is a unique meromorphic quadratic differential on $Z$ with 
two double poles with residue $-r$ at the punctures $z_1,z_2$ 
and such that if we denote by $D_i$ the disk domain formed by the closed trajectories surrounding 
$z_i$ $(i=1,2)$ then $Z=\cup_i \overline{D_i}$. Note to this end
that none of the components of $S\setminus c$ are twice punctured spheres, hence Theorem 23.5 of \cite{St84} can 
be applied. 
The differential may have a simple pole
at some punctures of $S\setminus c$ different from the node. 
But then the critical graph of the differential is a marked metric 
ribbon graph, and to this marked metric ribbon graph one can attach a cylinder of height 
$1/r$ to obtain a Riemann surface in ${\cal T}(S)$, equipped  
with a Strebel differential with core curve $c$, and this Riemann surface  
is contained in 
$\partial N(c)$ by construction. The resulting surface is not unique, but any
two choices of such a surface are contained in the 
same orbit of the horocycle flow as their critical graphs are marked isometric, 
and this flow line is mapped to $Z$ by the map $\Pi$. 

The construction in the previous paragraph associates to $Z\in {\cal S}(c)$ a flow line for $h_t$ in 
$\partial N(c)$ in a continuously varying fashion. To promote this construction to a map $\sigma$
with properties (2),(3) in the lemma, note that by continuity, 
for each point $z\in {\cal S}(c)$ we can 
find a neighborhood $U_z$ of $z$ and a local section $\sigma_z:{\cal S}(c)\to \partial N(c)$ 
for the projection $\Pi$ so 
that $\tau(\sigma_z(y),c)\in [-m,m]$ for all $y\in U_z$ and some fixed constant $m>0$. 
This makes sense since the ambiguity in the definition
of the function $\tau$ is at most an additive constant one, moreover $\tau$ is coarsely continuous
in the sense that values of $\tau$ at nearby points only differ by a universal additive constant. 
Using a partition of unity and the fact that the fiber of $\Pi$ is contractible, 
these local sections can be patched together to a global section with the 
properties in the lemma. 
\end{proof}

A section $\sigma$ as in Lemma \ref{section} is an embedding of 
${\cal S}(c)$ into $\partial N(c)$. This embedding can be used to construct 
a homeomorphism 
$\Sigma:{\cal S}(c)\times \mathbb{R}\to \partial N(c)$ which is 
equivariant with respect to the action of $\mathbb{Z}$ on 
$\mathbb{R}$ by translation and the action of the infinite cyclic group $\langle T_c\rangle$ of Dehn
twists about $c$ on $\partial N(c)$. As ${\cal S}(c)\times \mathbb{R} $ is just
the product of the Teichm\"uller space ${\cal T}(S_c^*)$ and the Teichm\"uller space of the annulus
with core curve $c$, we can view $\partial N(c)$ also in this way. 

The \emph{Teichm\"uller geodesic flow} $\phi_t$ 
acts on the area one Strebel differentials with core curve $c$, 
that is, on the section $q_c$, by scaling the circumference of the cylinder 
which determines $Y\in {\cal T}(S)$ by $e^{t/2}$ and its height by $e^{-t/2}$. 
Thus ${\cal T}(S)$ is foliated by 
Teichm\"uller geodesics whose co-velocities are such differentials. By the above 
discussion, the foliation is invariant under the subgroup ${\rm Stab}(c)$ of ${\rm Mod}(S)$.
Moreover, if $\gamma:\mathbb{R}\to {\cal T}(S)$ is such a Teichm\"uller geodesic, 
then the extremal length of $c$ along $\gamma$ is strictly increasing and hence the geodesic
intersects $\partial N(c)$ transversely in a single point. 
The 
Teichm\"uller geodesic $\gamma_Y$ through a point $\gamma_Y(0)=Y\in \partial N(c)$ converges 
as $t\to -\infty$ to $\Pi(Y)$ in 
the augmented Teichm\"uller space, as can easily be seen from rescaling of the 
Strebel differential to keep the critical graph, viewed as a metric ribbon graph, constant. 

The following lemma is a consequence of the article \cite{R14}.

\begin{lemma}\label{proalongray}
There exists a number $D>0$ with the following property. 
Let $Y\in \partial N(c)$; then for any not necessarily proper
subsurface $V$ of $S_c=S\setminus c$ 
we have ${\rm diam}({\rm pr}_{V}(\gamma_Y(-\infty,\infty)))\leq D$.
\end{lemma}
\begin{proof}
Let $t\to q(t)$ be the cotangent line of the geodesic $\gamma_Y$. 
In the flat metric defined by $q(t)$, 
the surface $S_c$ is degenerate, 
that is, the critical graph is a deformation retraction of $S_c$. 
 Any simple closed curve $\alpha$ in $S_c$ 
 is then homotopic to 
 a closed edge path in the critical graph $G$ of the differential, 
 and a closed edge path of minimal length is unique up to 
 parameterization and is the geodesic representing the free homotopy class of $\alpha$ for the 
 locally ${\rm CAT}(0)$-metric $q(t)$.
 
 The singular flat metric of $q(t)$ is obtained from the singular flat metric of
 $q(0)$ by multiplying the horizontal length, that is, the 
 circumference of the cylinder, with $e^{t/2}$, and the vertical length, that is, the height, with 
 $e^{-t/2}$. As a consequence, if for an essential  subsurface $V$ of $S_c$ 
 we define the size ${\rm size}_{q(t)}(V)$ 
 of $V$ with respect to the metric $q(t)$ as the $q(t)$-length of the shortest essential closed curve 
 in $V$ and if for a non-peripheral simple closed curve $\alpha\subset S_c$ 
 we denote by
 $\ell_{q(t)}(\alpha)$ the length of its $q(t)$-geodesic representative, 
 then 
 \[\log \frac{{\rm size}_{q(t)}(V)}{\ell_{q(t)}(\alpha)} \]
 does not depend on $t$. 
 From Theorem 3.1 of \cite{R14}, one deduces that the extremal length 
 of any non-peripheral simple closed curve in $S_c$ along the Teichm\"uller geodesic 
 $t\to \gamma_Y(t)$ is bounded from below by
 a universal positive constant. The same holds true for the extremal length of 
 $c$ along the ray $\gamma_Y[0,\infty)$,  
 which is exponentially increasing along the ray. 
 
 Together with the results from Section 5 of \cite{R14}, one deduces 
 that for any (not necessarily proper) subsurface $V$ of 
 $S_c$ the \emph{active interval} for $V$ is empty along $\gamma_Y(-\infty,\infty)$.  
 Theorem A of 
 \cite{R14} then shows that for each such $V$, the diameter of the projection 
 $\{{\rm pr}_V(\mu(\gamma_Y(t))) \mid t\in [0,\infty)\}$ is bounded from above by a universal 
 constant. This is what we wanted to show.
 \end{proof}

Let us move now back to the more general situation when the surface $S$ may be disconnected
and we apply the above discussion to the connected component $S^\prime$ 
of $S$ containing the simple closed curve $c$.
Since ${\cal T}(S)$ is the product of the Teichm\"uller spaces of the components of $S$, 
the foliation of ${\cal T}(S^\prime)$ induces a foliation of ${\cal T}(S)$ into geodesics which 
are constant on the components different from $S^\prime$.

Recall that the Teichm\"uller space of the annulus $A_c$ is naturally identified with the real line $\mathbb{R}$.
Start with a countable family ${\cal V}=\{V_i\mid i\}$ 
of open contractible subsets of $\mathring{\cal T}_\epsilon(S_c^*)\times \mathbb{R}$ 
whose small closures define a neighborhood basis of 
$\xi$ in $\overline{\cal T}(S_c^*\sqcup A_c)={\cal T}_\epsilon(S_c^*)\times \mathbb{R} \cup 
{\cal X}(S_c^*)*{\cal X}(A_c)$.
 Such a neighborhood basis exists since by the induction hypothesis, part (1) of Theorem \ref{mainprecise} holds
 true for $S_c\sqcup A_c$. 
 Let $\Lambda_\epsilon: {\cal T}(S_c^*)\to \mathring {\cal T}_\epsilon(S_c^*)$ be the coarsely
 $\Upsilon$-invariant homeomorphism from Corollary \ref{homeomorphism}. 
 By induction, we may in fact assume that $V_i$ is the image of 
  an open and contractible
 subset of ${\cal T}(S_c^*)\times \mathbb{R}$ under the map $\Lambda_\epsilon \times {\rm Id}$. 
 Denote by $W_i\subset \partial N(c)$ the image of this set  
 under the identification of $\partial N(c)$ with ${\cal T}(S_c^*)\times \mathbb{R}$ using the section $\sigma$.

\begin{proposition}\label{final}
Put $E_i=\{\gamma_Y(-\infty,\infty) \mid Y\in W_i\subset \partial N(c)\}$; then the sets 
$\Lambda_\epsilon(E_i)\subset \mathring {\cal T}_\epsilon(S)$ are open and contractible, and 
their small closures in $\overline{\cal T}(S)$ 
define a neighborhood basis of $\xi$ in $\overline{\cal T}(S)$. 
\end{proposition} 
\begin{proof} Since 
the Teichm\"uller geodesics determined by the one-cylinder Strebel differentials with core curve $c$
foliate ${\cal T}(S)$ and $\partial N(c)$ is transverse to these geodesics, the 
set $E_i$ admits a deformation retraction onto $W_i$. Thus since the sets 
$W_i$ are contractible,
the same holds true for the sets $E_i$ and for the sets $U_i=\Lambda_\epsilon(E_i)$.


We have to show that the small closures of the sets $U_i\subset \mathring {\cal T}_\epsilon(S)$ 
in $\overline{\cal T}(S)$ define a 
neighborhood basis of $\xi$ in $\overline{\cal T}(S)$. As $\cap_i \overline{U}_i=\{\xi\}$ since this holds true for 
$\cap_i \overline{V_i}$, this is the case if 
for any sequence $X_u\subset \mathring{\cal T}_\epsilon(S)$ converging to $\xi$, all but finitely many
$X_u$ are contained in $U_i$. 

By Lemma \ref{proalongray}, for each $Y\in \partial N(c)$ and any subsurface $V$ of $S_c^\prime$, 
the diameter of the subsurface projection ${\rm pr}_{V}(\mu(\gamma_Y(-\infty,\infty)))$ is uniformly bounded,
independent of $Y$. By the definition of the topology on $\overline{\cal T}(S)$, 
coarse $\Upsilon$-invariance of the projection $\Lambda_\epsilon$ and the choice of the neighborhoods $U_i$, 
for sufficiently large $u$ the subsurface projections of short markings of $X_u$ into 
all of the subsurfaces of $S\setminus c$ are close to the projections for points in $V_i$.
But by the above, the sets $U_i$ are constructed precisely in such a way that subsurface projections into 
subsurfaces of $S\setminus c$ are close to subsurface projections for points in $U_i$. 
As a consequence, for large enough $u$ we have $X_u\in U_i$. 

To summarize, for each $i$ the small closure $\overline{U}_{i,{\rm small}}$ in $\overline{\cal T}(S)$ 
of the set $U_i$ 
is indeed a neighborhood of $\xi$ in ${\cal T}(S)$. Hence the
sets $\overline{U}_{i,{\rm small}}$ define 
a neighborhood basis of $\xi$ in $\overline{\cal T}(S)$ as claimed in the proposition.
\end{proof}

\section{Proof of the main Theorem}\label{Sec:metrizability}


The goal of this section is to complete the proof of Theorem \ref{mainprecise}
and of the corollaries. 

\begin{proposition}\label{metrizable}
$\overline{\cal T}(S)$ is metrizable.   
\end{proposition}
\begin{proof}
By Uryson's theorem, a second countable Hausdorff space is
metrizable. As by Proposition \ref{topsmallbd} the space 
$\overline{\cal T}(S)$ is Hausdorff, 
it suffices to show that $\overline{\cal T}(S)$ is second countable. Since 
${\cal T}_\epsilon(S)$ is second countable, for this it suffices to show that there are 
countably many open sets in $\overline{\cal T}(S)$ which contain a neighborhood 
basis for every $\xi\in {\cal X}(S)$.

Now ${\cal X}(S)$ is a countable union of sets of the form 
${\cal J}(\cup_{i=1}^kS_i)$ where  $S_1,\dots,S_k$ is a collection of pairwise disjoint
subsurfaces of $S$. Since there only countably many such collections, 
it suffices to find for a given \emph{maximal} collection 
$S_1,\dots,S_k$ of disjoint subsurfaces of $S$ a countable collection of open sets in 
$\overline{\cal T}(S)$ which contain a neighborhood basis for each 
$\xi\in {\cal J}(\cup_{i=1}^kS_i)$.

For each $i$ let $S_i^*$ be the surface obtained from $S_i$ by replacing each boundary component 
by a puncture if $S_i$ is different from an annulus, and put $S_i^*=S_i$ if $S_i$ is an annulus. 
For each $i$ consider the marking graph ${\cal M}(S_i^*)$ of $S_i^*$, which can be thought of 
as the marking graph of $S_i$ but forgetting information on the boundary components. 
There is a 
coarsely well defined  projection
${\rm pr}_{S_i}:{\cal T}_\epsilon(S)\to {\cal M}(S_i^*)$ 
which associates to $Y\in {\cal T}_\epsilon(S)$ the projection of a short marking 
for $Y$ to a marking on 
$S_i^*$. Put
\[{\rm pr}_{\cup_iS_i}(Y)=({\rm pr}_{S_1}(Y),\dots, {\rm pr}_{S_k}(Y))\in \prod {\cal M}(S_i^*).\]

Corollary \ref{cor2} shows that there is a countable collection 
$\{W(\sum_i a_i\xi_i^{j_i},j,\ell,m)\mid a_i,j,\ell,m\}=\{U_j\mid j\}$ of 
subsets of ${\cal T}_\epsilon(\cup_iS_i^*)$ which contain a neighborhood basis for 
every $\xi\in {\cal J}(\cup_iS_i)={\cal J}(\cup_iS_i^*)\subset {\cal X}(\cup_iS_i^*)\subset {\cal X}(S)$, and 
points in these sets are 
uniquely determined by their short markings. Then the preimages under the projection 
${\rm pr}_{\cup_iS_i}$ of 
these sets are countably many subsets of ${\cal T}_\epsilon(S)$. 

If we denote by $V_j\subset {\cal T}_\epsilon(S)$ the set determined in this way by the subset 
$U_j$ of ${\cal T}_\epsilon(\cup_iS_i)$, 
then by the definition of the topology on 
$\overline{\cal T}(S)$, the closures of the sets $V_j$ in $\overline{\cal T}(S)$ 
are neighborhoods of points of ${\cal J}(\cup_iS_i)$ in $\overline{\cal T}(S)$, 
and hence the same holds true for their interiors. 
As the collection of subsurfaces $S_i$ of $S$ was chosen to be maximal and the sets $U_j$ contain
a neighborhood basis for every $\xi\in {\cal J}(\cup_iS_i)$, the sets $V_j$ then determine a neighborhood
basis in $\overline{\cal T}(S)$ for the points in ${\cal J}(\cup_iS_i)$. 
This completes the proof of the proposition.
\end{proof}

As a consequence of Proposition \ref{metrizable}, the pair $(\overline{\cal T}(S),{\cal X}(S))$ is a pair of 
compact metrizable spaces, with ${\cal X}(S)$ nowhere dense in $\overline{\cal T}(S)$. 

Recall that the \emph{covering dimension} of a topological space $X$ is 
the minimum of the numbers $n\geq 0$ so that the following holds true. Any open cover
${\cal U}$ of $X$ has a refinement ${\cal V}$ so that a point in $X$ is contained in at most
$n+1$ of the sets $V\in {\cal V}$. With this terminology, the covering dimension of 
$\mathbb{R}^n$ is $n$, and hence the covering dimension of any subset of $\mathbb{R}^n$ 
equipped with the subspace topology is at most $n$. In particular, the covering dimension 
of ${\cal T}(S)$ equals $6g-6+2m$. 

The following result relies on work of Gabai \cite{Ga14}, see also \cite{BB19}. 

\begin{proposition}\label{covX}
The covering dimension of ${\cal X}(S)$ is finite
\end{proposition}
\begin{proof}
If $S$ is an annulus then ${\cal X}(S)$ consists of two points and there is nothing to show.
Consider next a four-holed sphere or a one-holed torus $S$. By Example \ref{oncepuncturedtorus} and 
Example \ref{farey}, 
the geometric boundary as a topological space
is homeomorphic to the Gromov boundary of the hyperbolic group ${\rm PSL}(2,\mathbb{Z})$.
Since the group ${\rm PSL}(2,\mathbb{Z})$ is  
virtually free, the  
boundary ${\cal X}(S)$ of $S$ is a Cantor set, which has covering dimension zero.

Let $X$ and $Y$ be compact spaces with covering dimensions $m,n$. We claim that the covering 
dimension of the join $X*Y$ is at most $m+n+1$. To see that this is the case recall 
that $X*Y$ is the quotient of $X\times Y\times [0,1]$ under an equivalence relation $\sim$ which 
is only nontrivial on $X\times Y\times \{0\}$ and $X\times Y\times \{1\}$. The projection
$X\times Y\times [0,1]\to X\times Y\times [0,1]/\sim$ maps 
$X\times Y\times \{0\}$ to $X\times \{0\}$ and maps
$X\times Y\times \{1\}$ to $Y\times \{1\}$. Thus we have $X*Y=X\cup Y \cup C$ 
where $X\subset X*Y$ is the closed set which is 
the quotient of $X\times Y\times \{0\}$, $Y\subset X*Y$ is 
the closed set which is the quotient of 
$X\times Y\times \{1\}$, and the set $C$ is homeomorphic to $X\times Y\times (0,1)$. 

By Alexandrov's definition of dimension (see Theorem 3.4 of \cite{Dr18}), we have
${\rm dim}(A\times B)\leq {\rm dim}(A)+ {\rm dim}(B)$ and hence 
${\rm dim}(C)\leq {\rm dim}(X)+{\rm dim}(Y) +1$. 
The compact space 
$X*Y$ is the union of the closed subset $X\cup Y$ with $C$ and hence 
the theorem of Menger and Uryson (see Theorem 3.1 of \cite{Dr18}) shows that
${\rm dim}(X*Y)
=\max\{{\rm dim}(X\sqcup Y),{\rm dim}(C)\}
\leq {\rm dim}(X)+{\rm dim}(Y)+1=m+n+1$ as claimed.

Since for a disconnected surface $S=\sqcup_{i=1}^k S_i$ it holds
${\cal X}(S)={\cal X}(S_1)*\cdots *{\cal X}(S_k)$, it now suffices 
to show the proposition for connected surfaces. 
Assume by induction that the proposition was established for all connected surfaces of 
complexity at most $k-1\geq 1$. Let $S$ be a connected surface of complexity $k$. We have
${\cal X}(S)=\partial {\cal C\cal G}(S)\sqcup {\cal Y}$ (disjoint union) 
where ${\cal Y}=\cup 
{\cal X}(S_1)*\cdots *{\cal X}(S_p)$ and 
the union in the definition of ${\cal Y}$ is over all disjoint collections of proper subsurfaces 
$S_1,\dots,S_p$ of $S$. The union ${\cal Y}$ is not disjoint. 

The number 
of disjoint surfaces in one of the joins appearing in the 
definition of ${\cal Y}$  is uniformly bounded in terms of $k$.
Thus by the induction hypothesis and the above dimension estimate for joins, applied inductively,  
there exists a number $n>0$ which bounds from above the covering dimension
of each of the sets ${\cal X}(S_1)*\cdots *{\cal X}(S_p)$. 
Example \ref{join}
 shows that as subsets of ${\cal X}(S)$, the sets 
 ${\cal X}(S_1)*\cdots *{\cal X}(S_p)$ 
 are closed 
and hence compact. As a consequence, 
the subspace ${\cal Y}$
of ${\cal X}(S)$, equipped with the induced topology, 
 is a $\sigma$-compact Hausdorff space as it is 
a countable union of compact spaces. 

If $K\subset {\cal Y}$ is compact,
then $K$ can be represented as a countable 
union of the compact spaces $K\cap {\cal X}(S_1)*\cdots *{\cal X}(S_p)$.
Then the countable union theorem Theorem 3.2 of \cite{Dr18} shows that
${\rm dim}(K)=\sup \{{\rm dim}(K\cap {\cal X}(S_1)*\cdots *{\cal X}(S_p))\}$
where the supremum is over all disjoint unions of proper subsurfaces of $S$. By the above 
estimate for the dimension of the spaces 
${\cal X}(S_1)*\cdots *{\cal X}(S_p)$, we have
\[{\rm dim}(K\cap {\cal X}(S_1)*\cdots *{\cal X}(S_p))\leq n\] for any such intersection.
But then the dimension of ${\cal Y}$ is at most $n$
(see p.316 of \cite{Mu14} for a sketch of a proof). 

Following \cite{Ga14}, if $g$ is the genus of $S$ and if $m$ is the number of its punctures, then 
the covering dimension of $\partial {\cal C\cal G}(S)$
is at most $4g-5+2m$. 
Then by the Uryson-Menger formula (see Theorem 3.3 of \cite{Dr18}), the 
dimension of the compactum ${\cal X}(S)$ is at most 
\[{\rm dim}({\cal X}(S))={\rm dim}(\partial {\cal C\cal G}(S))+{\rm dim}({\cal Y})+1 \]
and hence it is finite.
\end{proof}

As a consequence, we obtain

\begin{corollary}\label{finitedim}
The pair $(\overline{\cal T}(S),{\cal X}(S))$ is a pair of spaces of finite dimension.
\end{corollary}
\begin{proof}
By Proposition \ref{covX}, the dimension of ${\cal X}(S)$ is finite.
As the compactum $\overline{\cal T}(S)={\cal T}_\epsilon(S)\cup {\cal X}(S)$ is a union of 
two subspaces of finite dimension, with ${\cal X}(S)\subset \overline{\cal T}(S)$ closed, 
we have 
\[{\rm dim}(\overline{\cal T}(S))=\max\{{\rm dim}({\cal T}_\epsilon(S)),{\rm dim}({\cal X}(S))\}<\infty.\]
\end{proof}


The action of ${\rm Mod}(S)$ on $\overline{\cal T}(S)$ is ${\cal U}$-small if
for any open covering ${\cal U}$ of the compact space 
$\overline{\cal T}(S)$ and for any compact subset $K\subset {\cal T}_\epsilon(S)$, 
the image of $K$ under all but finitely many elements $\psi\in {\rm Mod}(S)$ is 
completely contained 
in one of the sets from the covering ${\cal U}$ and hence in the intersection of one of these
sets with ${\cal T}_\epsilon(S)$.  
Informally, the images of any 
compact subset of ${\cal T}_\epsilon(S)$ only accumulate near the boundary 
${\cal X}(S)\subset \overline{\cal T}(S)$. 

\begin{proposition}\label{prop:small}
The action of ${\rm Mod}(S)$ on ${\cal T}(S)$ is ${\cal U}$-small for every
open covering ${\cal U}$ of  $\overline{\cal T}(S)$.
\end{proposition}
\begin{proof}
Let ${\cal U}$ be an open covering of $\overline{\cal T}(S)$.
By compactness, we may extract a finite subcovering, so we may
assume that ${\cal U}$ is in fact finite, that is, we have 
${\cal U}=\cup_{0\leq i\leq m}U_i$ for some open sets
$U_i\subset \overline{\cal T}(S)$. 

Arguing by contradiction, assume that there exists 
a compact set $K\subset {\cal T}_\epsilon(S)$
and infinitely many distinct elements $\phi_j\in {\rm Mod}(S)$ such that
$\phi_jK\not\subset U_i$ for any $i\leq m$ and all $j$.

Let $X\in K$.
Since the action of 
${\rm Mod}(S)$ on ${\cal T}_\epsilon(S)$ is properly discontinuous  and
$\overline{\cal T}(S)$ is compact, 
we conclude that up to passing to a subsequence, the sequence
$\phi_jX$ converges in $\overline{\cal T}(S)$ to a point $\xi\in {\cal X}(S)$.
Now recall that as the action of ${\rm Mod}(S)$  is isometric for the 
Teichm\"uller metric and $K\subset {\cal T}_\epsilon(S)$ is compact, all diameters 
of subsurface projections 
for short markings of all points in $\phi_j(K)$ are uniformly bounded, 
independent of $j$. 
Then it follows from the definition of the topology on $\overline{\cal T}(S)$  
that $\phi_jK\to \xi$. Hence if $p\leq m$ is such that $\xi\in U_p$ the 
$\phi_jK\subset U_p$ for all sufficiently large
$j$. This is a contradiction which completes the proof of the proposition.
\end{proof}

\begin{proof}[Proof of Theorem \ref{mainprecise}]
Let $S$ be any (possibly disconnected) surface of finite type. 

By Proposition \ref{topology2}, there exists a 
topology on $\overline{\cal T}(S)={\cal T}_\epsilon(S)\cup {\cal X}(S)$ 
with property (2) in the statement of
Theorem \ref{mainprecise}. By Proposition \ref{topsmallbd}, with respect to 
this topology, the space $\overline{\cal T}(S)$ is compact and separable, and 
${\rm Mod}(S)$ acts on $\overline{\cal T}(S)$ as a group of transformations. Moreover,
by construction, ${\cal X}(S)$ is a closed subset of $\overline{\cal T}(S)$, and 
${\cal T}_\epsilon(S)$ is open and dense.

Proposition \ref{metrizable} shows that $\overline{\cal T}(S)$, equipped with  this topology, is metrizable, 
and by Corollary \ref{finitedim}, the pair $(\overline{\cal T}(S),{\cal X}(S))$ is a pair of spaces
of finite dimension. Thus $(\overline{\cal T}(S), {\cal X}(S))$ is a pair of finite dimensional
compact metrizable spaces, with ${\cal X}(S)$ nowhere dense in 
$\overline{\cal T}(S)$.

We next verify that the pair $(\overline{\cal T}(S),{\cal X}(S))$ is an
${\cal E\cal Z}$-structure for ${\rm Mod}(S)$. 
Theorem \ref{model} and Theorem \ref{ji} confirm that ${\cal T}_\epsilon(S)$ is contractible and 
locally contractible, which is requirement (1) in Definition \ref{boundary}. 
By Proposition \ref{final}, every point $\xi\in {\cal X}(S)$ admits a neighborhood basis 
consisting of sets whose intersections with ${\cal T}_\epsilon(S)$ are contractible, whence
requirement (2) in Definition \ref{boundary}.
That the action of ${\rm Mod}(S)$ on ${\cal T}_\epsilon(S)$ is properly discontinuous and cocompact
is well known and does not require an additional proof, so requirement (3) holds true.
Furthermore, by Proposition \ref{prop:small}, the action of 
${\rm Mod}(S)$ on $\overline{\cal T}(S)$ is ${\cal U}$-small, 
verifying requirement (4). That the action of 
${\rm Mod}(S)$ extends to $\overline{\cal T}(S)$ was shown in Proposition \ref{topsmallbd}.
As a consequence, $(\overline{\cal T}(S),{\cal X}(S))$ is indeed an ${\cal E\cal Z}$-structure for 
${\rm Mod}(S)$. 

To complete the proof of Theorem \ref{mainprecise}, it now suffices to verify that 
the topology on $\overline{\cal T}(S)$ is nice. If $S$ is connected, then this follows from 
Proposition \ref{curvegraphbdneigh} and Proposition \ref{final}. If $S$ is disconnected, then it 
is a consequence of Proposition \ref{contractible} and Proposition \ref{final}. 
\end{proof}

From now on, we only consider connected surfaces $S$. The following corollary completes the proof of 
part (1) of Proposition \ref{details}.

\begin{corollary}\label{covering}
We have ${\rm dim}(\partial {\cal C\cal G}(S))\leq {\rm vcd}({\rm Mod}(S))-1$.
\end{corollary}
\begin{proof}
Since ${\cal X}(S)$ is a ${\cal Z}$-set for a torsion free finite index subgroup $\Gamma$
of ${\rm Mod}(S)$, 
the cohomological dimension of ${\cal X}(S)$ equals 
${\rm vcd}({\rm Mod}(S))-1$ \cite{B96}. Furthermore, 
this dimension also equals the covering dimension of ${\cal X}(S)$ \cite{B96}.

Now as $\partial {\cal C\cal G}(S)$ is embedded in ${\cal X}(S)$, it is equipped with
the subspace topology. This means that any open covering of $\partial {\cal C\cal G}(S)$
is the restriction of an open covering of ${\cal X}(S)$. Such a covering then has a
${\rm vcd}({\rm Mod}(S))-1$-finite refinement and hence the same holds true for the 
refinement of the original cover of $\partial {\cal C\cal G}(S)$.
\end{proof}

The following conjecture is taken from \cite{BB19}. We believe that the 
results in this work support this conjecture.

\begin{conjecture}
For any surface $S$ of finite type, 
${\rm asdim}({\rm Mod}(S))={\rm vcd}({\rm Mod}(S))$.
\end{conjecture}

We are left with showing 
Corollary \ref{concrete} and Corollary \ref{bdcurvegr}  from the introduction.

\begin{proof}[Proof of Corollary \ref{concrete}]
By Theorem \ref{main}, ${\rm Mod}(S)$ admits an 
${\cal E\cal Z}$-structure $(\overline{X}, Z)$ where $X=\overline{X}\setminus Z$
is a manifold with boundary of dimension $6g-6+2m$. Assume that $6g-6+2m\geq 5$.
By Lemma 2.3 of \cite{FL05}, there exists a new ${\cal E\cal Z}$-structure for 
${\rm Mod}(S)$ obtained by doubling $X$ along the boundary. 
By Proposition 2.1 of \cite{FL05}, this 
structure is of the form $(\overline{Y},Z)$ where $\overline{Y}$ is a manifold with 
boundary of dimension $6g-5+2m$.

Proposition 2.2 of \cite{FL05} then shows that another application of this construction
to the pair $(\overline{Y},Z)$ results in an ${\cal E\cal Z}$-structure given
by the unit ball in $\mathbb{R}^{6g-4+2m}$ and a subset $Z$ of its boundary, the 
sphere of dimension $6g-5+2m$. This is what we wanted to show.
\end{proof}

\begin{proof}[Proof of Corollary \ref{bdcurvegr}]
The corollary follows from the fact that $\partial {\cal C\cal G}(S)$ is embedded in ${\cal X}(S)$, and by Corollary
\ref{concrete} and its proof, ${\cal X}(S)$ embeds into a manifold of dimension $6g-6+2m$ and
into the sphere $S^{6g-5+2m}$.
\end{proof}

\begin{question}
What is the smallest dimension $n$ so that ${\rm Mod}(S)$ admits an ${\cal E\cal Z}$ structure on a pair
$(\mathbb{D}^n,\Delta)$ where $\Delta$ is a subset of $S^{n-1}$?
\end{question}

\bigskip

\noindent
MATHEMATISCHES INSTITUT DER UNIVERSIT\"AT BONN\\
ENDENICHER ALLEE 60, \\
D-53115 BONN, GERMANY\\

\bigskip\noindent
email: ursula@math.uni-bonn.de

\end{document}